\documentclass[11pt,reqno,a4paper]{amsart}

\usepackage{amsmath}
\usepackage{amssymb}
\usepackage{amsthm}

\usepackage[margin=1in]{geometry} 

\usepackage{enumitem}

\usepackage{tikz}

\theoremstyle{plain}

\newtheorem{theorem}{Theorem}
\newtheorem{lemma}[theorem]{Lemma}
\newtheorem{proposition}[theorem]{Proposition}
\newtheorem{corollary}[theorem]{Corollary}
\theoremstyle{definition}
\newtheorem{definition}[theorem]{Definition}

\newtheorem*{remark}{Remark}

\numberwithin{theorem}{section}

\numberwithin{equation}{section}

\usepackage{xcolor}
\usepackage{hyperref}

\numberwithin{equation}{section}

\allowdisplaybreaks

\keywords{}

\title{Fluid-Structure interactions with Navier- and full-slip boundary conditions}

\author{Anton\'in \v Ce\v s\'ik}
\address{Mathematics Institute, University of Warwick, Zeeman Building, Coventry CV4 7AL, United Kingdom} \email{antonin.cesik@warwick.ac.uk}

\author {Malte Kampschulte}
\address{Faculty of Mathematics and Physics, Charles University, Sokolovsk\'{a} 83, 18675, Prague, Czech Republic}
\email {kampschulte@karlin.mff.cuni.cz}

\author{Sebastian Schwarzacher}
\address{Department of Mathematics, Uppsala University, Box 480
    751 06 Uppsala \&  Faculty of Mathematics and Physics, Charles University,
Sokolovsk\'{a} 83, 18675, Prague, Czech Republic }
\email{schwarz@karlin.mff.cuni.cz}
\date{\today}

\begin{document}
\begin{abstract}
We show the existence of weak solutions to the fluid-structure interaction problem of a largely deforming viscoelastic bulk solid with a viscous fluid governed by the incompressible Navier-Stokes equations. In contrast to previous works, the fluid is allowed to slip along the solid boundary; namely, the so called Navier-slip boundary conditions are considered. Such boundary conditions naturally involve the time-changing outer normal of the fluid domain. Hence, their dependence on the varying geometry is one degree higher than in the previously considered no-slip case, which makes it necessary to adjust the concept of weak coupled solutions. Two classes of test functions are introduced: test functions that are continuous over the fluid-solid domain, and fluid-only test functions with nonzero tangential component at the boundary. The weak equations are established until the point of contact, and moreover, compatibility with the strong formulation is shown.
\end{abstract}

 \maketitle
%  \tableofcontents
%\noteMalte[inline]{I am not quite sure about the ``...with slip'' in the title. On the other hand ``...with Navier- and full-slip boundary conditions'' is a bit lengthy. I also think the second half of the abstract could be a bit condensed.}

\section{Introduction}
In this work, we introduce and study the existence for a concept of weak solutions for a largely deforming visco-elastic solid interacting with a fluid governed by the Navier-Stokes equations with Navier slip boundary conditions at the interface. The construction and conception of solutions follows the variational approach introduced some years ago \cite{benesovaVariationalApproachHyperbolic2023}. It allows for treatment of largely deforming bulk solids interacting with fluids, when the solid deformation is actually defining the time-changing fluid domain. Slipping at the interface for weak solutions includes the possibility of shearing velocities in the tangential direction of the interface of arbitrary magnitude, while the speed in the normal direction is fully determined by the motion of the solid. Hence, the velocity of the fluid flow and the solid deformation are coupled differently in tangential and normal directions, and thus the {\em difference depends on the time-changing tangent plane of the separating interface}. This higher order dependence made it necessary to relax the concept of weak solutions in comparison to previous works~\cite{benesovaVariationalMethodsFluid2023}.
 The difference however only occurs at the interface. In particular, the local weak formulations in the interior of both fluid and solid, respectively, stay unchanged. Moreover, it is shown in the present work that such weak solutions are strong solutions the moment they possess sufficient regularity.

There are by now numerous works on weak solutions for fluid-structure interaction problems, where the elastic deformation is changing the geometry of the fluid-domain. Many results have been established for lower dimensional structures, such as plates or shells, interacting with fluids. See~\cite{vcanic2021moving} for an overview on analytic and numeric efforts. Further, for existence results of weak solutions with {\em no-slip boundary conditions} see \cite{muhaExistenceWeakSolution2013, lengelerWeakSolutionsIncompressible2014, muhaFluidstructureInteractionIncompressible2015,breitCompressibleFluidsInteracting2018, schwarzacherWeakStrongUniquenessElastic2022, muhaExistenceRegularityWeak2022,  machaExistenceWeakSolution2022, muhaExistenceWeakSolution2013,  kampschulteUnrestrictedDeformationsThin2023, breitNavierStokesFourierFluidsInteracting2023}. And for existence results of weak solutions with slip-boundary conditions see~\cite{gerard2014existence,muha2016existence,
liuCompressibleFluidstructureInteraction2024,tawri2024stochastic}.
%\noteMalte[inline]{Sebastian, you probably know the literature better. Could you insert a few words mentioning why our result is different from the results cited here?}

One motivation to study slip-boundary conditions in the context of fluid-structure interactions is the so-called no contact paradox, also known as the Cox-Brenner paradox. It says that if slipping at the boundary is excluded, rigid solids cannot collide in a viscous incompressible fluid~\cite{hillairet2007lack}. In contrast, if slip boundary conditions are considered, contact happens~\cite{hillairet2009collisions}. {\em This is possible as the tangential fluid velocity can rise to infinity at the contact point.} Even though {\em mathematically} the question on whether collisions for deformable elastic solids are possible {\em without slipping} is at an early stage~\cite{grandmont2016existence}, there are strong numerical indications that the paradox holds also for deforming elastic solids~\cite{gravina2022contactless,fara2024geometric}.

The variational approach devised in  \cite{benesovaVariationalApproachHyperbolic2023} gave rise to several new results in this setting~\cite{breitCompressibleFluidsInteracting2024,benesovaVariationalMethodsFluid2023} including the analysis of contact of solids in the compressible and the fluid-less case~\cite{kampschulteGlobalWeakSolutions2024,cesikInertialEvolutionNonlinear2024}. This paper extends these efforts allowing for \textit{slip} at the fluid-solid interface. As such boundary conditions allow for arbitrary tangential fluid-speeds, it can also be seen as a starting point  for further studies on the effect of the fluid on contact and bouncing of elastic solids.

We employ the physical setup of a viscoelastic bulk solid immersed in an incompressible Navier-Stokes fluid as in \cite{benesovaVariationalApproachHyperbolic2023}. We fix a container $\Omega\subset\mathbb R^d$ and consider both the fluid and the solid to be confined to $\Omega$. The solid is described with respect to a Lagrangian reference configuration $Q\subset\mathbb R^d$. The time-dependent solid deformation is then $\eta\colon (0,T) \times Q\to \Omega$, so that $\eta(t)=\eta(t,\cdot)$ is the deformation of the solid at the time $t\in(0,T)$. The spatial dimension is $d\geq 2$, with $d=3$ and $d=2$ being the most physically relevant ones. We assume that the fluid occupies the rest of the container not occupied by the solid. That is, the fluid at time $t$ is defined in the domain $\Omega(t)=\Omega\setminus \eta(t,Q)$. and is determined by the velocity $v(t)\colon \Omega(t)\to \mathbb R^d$ and the pressure $p(t)\colon\Omega(t)\to \mathbb R$.

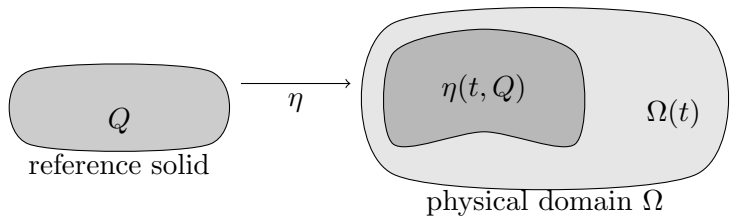
\begin{figure}[h]
\begin{center}
\begin{tikzpicture}[scale=.8]
\draw[black, fill=black, fill opacity=.2] plot [smooth cycle] coordinates {(0,0) (3,0) (3,1.2) (0,1.2)};
\draw[black] (1.5,0) node[above] {$Q$} (1.5,0) node[below] {reference solid};
\draw[black,->] (3.5,1) --  node[below] {$\eta$} (5.3,1);
\draw[black, fill=black, fill opacity=.1] plot [smooth cycle] coordinates  {(6,-.5) (11,-.5) (11,2) (6,2)};
\draw[black] (10,.5) node[right] {$\Omega(t)$};
\draw (8.5,-.6) node[below] {physical domain $\Omega$};
\draw[black, fill=black, fill opacity=0.2] plot [smooth cycle] coordinates {(6,0) (7.5,.2) (9,0) (9,1.5) (7.5,1.9) (6,1.5)};
\draw[black] (7.5,.5) node[above] {$\eta(t,Q)$};
\end{tikzpicture}
\end{center}
\caption{Geometrical depiction of the fluid-structure interaction problem.}
\end{figure}

Now let us describe the equations of motion. We assume, similarly as in previous works \cite{benesovaVariationalApproachHyperbolic2023, cesikInertialEvolutionNonlinear2024, breitCompressibleFluidsInteracting2024}, the hyperelastic case, where the Piola-Kirchhoff stress tensor of the solid can be derived from the energy and dissipation potentials. Thus, the momentum equation for the solid is 
\begin{equation}\label{eqn:solid-intro}
\rho_s\partial_{tt}\eta +DE(\eta)+D_2R(\eta,\partial_t\eta)=\rho_s f \quad \text{in }Q
\end{equation}
where $\rho_s$ is the (reference) solid density and $f\colon (0,T)\times \Omega\to \mathbb R^d$ the external force in the Eulerian configuration. The fluid will be assumed to be satisfy an incompressible Navier-Stokes equation, so that
\begin{equation}
\rho_f(\partial_t v +v\cdot \nabla v)= \nu \Delta v -\nabla p + \rho_f f\quad \text{in } \Omega(t),
\end{equation}
with the incompressibility condition
\begin{equation}
\operatorname{div} v= 0 \quad \text{in }\Omega(t).
\end{equation}
The \emph{kinematic coupling} to the fluid-solid will be through an \emph{impermeability condition}. That is, it holds 
\begin{equation}
v\cdot n=(\partial_t\eta \circ\eta^{-1}) \cdot n \quad \text{on } \partial \eta(t,Q)
\end{equation}
where $n$ is the inner unit normal vector to the fluid-solid interface $\partial \eta(t,Q)$ (i.e. also the outer unit normal vector to $\partial \Omega(t)$ on the interface part). On the outer boundary of the container, we also prescribe the impermeability condition $$v\cdot n = 0 \quad \text{on }\partial\Omega.$$
%\noteMalte[inline]{It does not matter in this equation, as it appears on both sides of the equation, but in general which normal? I think it was nowhere stated in the paper at all. Since this is the first occurence, I fixed it here, in the hope that it is consistent with the later stuff, but please verify.}
In addition to the kinematic coupling condition, there will be equalities of stress with the slipping law, and also of the solid hyperstress (which arises due to the presence of second gradients). If $\sigma_s$ and $\sigma_f$ denote the solid and fluid-stress tensor in Eulerian configuration respectively, then this results in
\begin{equation}
 (\sigma_s n-\sigma_f n) = a (\partial_t \eta \circ \eta^{-1} - v) \text{ in } \partial \Omega(t)
\end{equation}
where $a \geq 0$ is a constant determining the slip. Note that due to the kinematic coupling, the right hand side is pointing strictly in tangential direction, i.e.\ the normal stresses always coincide.

%\noteMalte[inline]{I am not fully happy with this, but we should at least say something here. Also note that the sign above depends on the direction of the normal. I hope I picked the right sign, but please check.}

To fully express these boundary conditions in terms of $\eta$ and $v$ requires careful translation between Eulerian and Lagrangian coordinates. Thus we will fully state them later in the definition of strong solutions, namely Definition \ref{def:strong-solution-full}. The {\em Navier-slip boundary condition} means here that the tangential part of the stress is equal to a friction term depending on the tangential velocity difference between fluid and solid~\eqref{eq:slip}. In the special case of free slipping (no friction) all tangential movement is free and hence the corresponding stresses vanish in the strong formulation.

Further, the solid is presumed to satisfy the \emph{no-interpenetration of matter} in the form of the Ciarlet-Ne\v cas condition 
\begin{equation}
|\eta(Q)|=\int_Q\det\nabla\eta(x)\,dx.
\end{equation}
Finally, we complete the system by prescribing the initial conditions 
\begin{equation}\label{eqn:init-cond-intro}
\eta(0)=\eta_0,\quad \partial_t\eta(0)=\eta_*, \quad v(0)=v_0, \quad \Omega(0)=\Omega \setminus \eta_0(Q).
\end{equation}

In contrast to no-slip boundary conditions, in case of slip the tangential component of the fluid velocity is not fixed at the interface. Hence, as stated above, the kinematic coupling condition reduces to impermeability. In the framework of weak solution the test functions have to be adapted accordingly. Typically for fluid-structure interactions, the space of test function is non-linearly related to the solution. In contrast to the no-slip case, in the case of slip the \emph{normal vector} of the geometry (that comes from the time-changing solid deformation) influences the test-function space. Hence the gradient of the solid deformation directly influences the weak-formulation. In order to construct a solution, this higher order influence is reduced to the fluid equation.

This explains why our weak formulation consists of two types of test functions -- what we call the \emph{coupled} and the \emph{fluid-only} test functions. Coupled test functions are defined to be continuous over $\Omega$, and the corresponding Lagrangian test function for the solid is pulled from $\eta(Q)$ back to $Q$ by $\eta$. Fluid-only test functions are defined on $\Omega(t)$ and have normal component zero on $\partial\eta(\cdot,Q)$. In essence only the tangential part of the fluid-only test function on $\partial\eta(\cdot,Q)$ is important, and this is what gives rise to the slip. This is in contrast to \cite{benesovaVariationalApproachHyperbolic2023} where the only test functions are the continuous coupled ones, which results in matching tangential stresses, corresponding to a no-slip boundary condition. Please note that the splitting of the weak formulation has to be performed on all levels of approximation. This is necessary due to the regularity drop between solids (2nd order) and fluids (1st order). Related to this deviation it turns out suitable to derive the proper a-priori estimates already on the level of discrete in time approximation (see below). Here we rely on our findings in \cite{cesikStabilityConvergenceTime2023}.

We then show the following existence result.

\begin{theorem}
The fluid-structure interaction problem \eqref{eqn:solid-intro}-\eqref{eqn:init-cond-intro} has a weak solution in the sense of Definition \ref{def:weak-solution-full}, until the time of the first solid-solid collision.
\end{theorem}

The paper is organized as follows. In Section \ref{sec:preliminaries} we discuss preliminary material on moving domains and the corresponding function spaces. In Section \ref{sec:strong-and-weak-formulation} we introduce the full weak formulation of the problem, as well as the strong formulation, and show their equivalence for regular solutions. In Section \ref{sec:variational-existence-scheme} we construct the solution by a variational time-stepping scheme by multiple levels of approximation.

\subsection{Assumptions on energy and dissipation}

We work in the context of \textit{nonsimple} elastic materials \cite{kruzikMathematicalMethodsContinuum2019}, where the energy depends on higher gradient. It has been observed by \cite{healeyInjectiveWeakSolutions2009} that under suitable growths of the energy, a well-known issue of obtaining lower bound on the Jacobian \cite{ballOpenProblemsElasticity2002} is circumvented, thus obtaining an Euler-Lagrange equation for finite energy configurations is always possible. As in previous results, we stick to this setting here.

The set of admissible deformations is defined as 

\begin{equation}\label{eqn:admissible-deformations}
\mathcal E:= \left\{\eta\in W^{2,q}(Q;\mathbb R^d):\eta(Q)\subset \Omega, \det\nabla\eta>0, |\eta(Q)|=\int_Q\det\nabla\eta(x)\,dx\right\}
\end{equation}
where $q>d$, so that we have the embedding $W^{2,q}(Q;\mathbb R^d)\rightharpoonup C^{1,1-d/q}(Q;\mathbb R^d)$.
Here we specify the assumptions on the energy $E$ and dissipation $R$. We assume that the \emph{elastic energy potential} $E \colon W^{2,q}(Q; \mathbb R^d)\to (-\infty, \infty]$ has the following properties: 
\begin{enumerate}[label=(E.\arabic*), ref=E.\arabic*]
\item \label{as:E-bounded-below} There is $E_{\min} > - \infty$ such that $E(\eta) \ge E_{\min}$ for all $\eta \in W^{2, q}(Q; \mathbb R^d)$. Further, for $\eta \in W^{2, q}(Q; \mathbb R^d)$ with $\inf_Q\det \nabla \eta > 0$ it holds $E(\eta) < \infty$.
\item \label{as:E-lower-bound-det} For every $E_0  \ge E_{\min}$ there exists $\varepsilon_0 > 0$ 
such that $E(\eta) \le E_0$ implies $\det \nabla \eta \ge \varepsilon_0$ .
\item \label{as:E-bounded-on-w2q} For every $E_0 \ge E_{\min}$ there exists $C$ such that $E(\eta) \leq E_0$ implies 
\(
\|\nabla^2 \eta\|_{L^q} \le C 
\).
\item \label{as:E-wlsc} $E$ is weakly lower semicontinuous. That is, for $\eta_{k} \rightharpoonup \eta$ in $W^{2, q}(Q; \mathbb R^d)$ it holds
\[
E(\eta) \le \liminf_{k \to \infty} E(\eta_k).
\] Further, $E$ is strongly continuous in $W^{2, q}(Q; \mathbb R^d)$.
\item \label{as:E-differentiable} $E$ is differentiable for all $\eta\in\mathcal E$, with the derivative $DE(\eta) \in (W^{2, q}(Q; \mathbb R^d))^*$ given by 
\[
 DE(\eta) \langle \varphi \rangle = \left.\frac{d}{ds} E(\eta + s \varphi))\right|_{s = 0}. 
\]
Furthermore, $DE$ is bounded on sublevel sets of $E$ and 	continuous with respect to strong $W^{2,p}(Q;\mathbb R^d)$ convergence, for some $p<q$.
\item \label{as:E-minty-property} The derivative $DE$ satisfies 
\[
\liminf_{k \to \infty}  (DE(\eta_k) - DE(\eta)) \langle \eta_k - \eta \rangle \ge 0
\]
for all $\eta_k \rightharpoonup \eta$ in $W^{2, q}(Q; \mathbb R^d)$. Further, $DE$ satisfies the following \emph{Minty-type property}: If
\[
\limsup_{k \to \infty} (DE(\eta_k) - DE(\eta))\langle \eta_k - \eta \rangle \le 0,
\]
then $\eta_k \to \eta$ in $W^{2, q}(Q; \mathbb R^d)$.
\item \label{as:E-nonconvexity-estimate} $E$ satisfies the following \textit{non-convexity estimate}: For all $E_0\geq E_{\min}$ there exists $C_1\geq 0$ such that for all $\eta_1,\eta_0\in\mathcal E$ with $E(\eta_1),E(\eta_0)\leq E_0$ it holds $$DE(\eta_1)\langle\eta_1-\eta_0\rangle \geq E(\eta_1)-E(\eta_0)-C_1\|\nabla\eta_1-\nabla\eta_0\|_{L^2}^2.$$
\end{enumerate}

For the \emph{dissipation potential} $R \colon W^{2, q}(Q; \mathbb R^d) \times W^{1, 2}(Q; \mathbb R^d) \to [0, \infty)$, we assume the following properties:

\begin{enumerate}[label=(R.\arabic*), ref=R.\arabic*]
\item \label{as:R-wlsc} $R$ is weakly lower semicontinuous in its second argument. That is, for all $\eta \in W^{2, q}(Q; \mathbb R^d)$ and all $b_{k} \rightharpoonup b$ in $W^{1, 2}(Q; \mathbb R^d)$ it holds
\[
R(\eta, b) \le \liminf_{k \to \infty} R(\eta, b_k)
\]
\item \label{as:R-two-homogeneous} $R$ is $2$-homogeneous in its second argument, that is, 
\[
R(\eta, \lambda b) = \lambda^2 R(\eta, b),\quad \lambda \in \mathbb R.
\]
\item \label{as:R-korn-inequality} $R$ admits the following \emph{Korn-type inequality}: For every $\varepsilon_0 > 0$, there is $K_R$ such that
\[
K_R \|b\|_{W^{1,2}}^2 \le \|b\|_{L^2}^2 +  R(\eta, b)
\]
for all $\eta \in \mathcal{E}$ with $\det \nabla \eta > \varepsilon_0$ and all $b \in W^{1, 2}(Q; \mathbb R^d)$. 
\item \label{as:R-differentiable} $R$ is differentiable in its second argument, with the derivative $D_2R(\eta, b) \in (W^{1, 2}(Q; \mathbb R^d))^*$. 
Further, the map $(\eta, b) \mapsto D_2R(\eta, b)$ is bounded and weakly continuous with respect $\eta$ and $b$. This means that for all $\varphi \in W^{1, 2}(Q; \mathbb R^d)$ and all $\eta_k \rightharpoonup \eta$ in $W^{2, q}(Q; \mathbb R^d)$ and $b_k \rightharpoonup b$ in $W^{1, 2}(Q; \mathbb R^d)$ it holds 
\[
\lim_{k \to \infty}  D_2R(\eta_k, b_k) \langle \varphi \rangle =  D_2R(\eta, b) \langle \varphi \rangle.
\]
\end{enumerate}

The model case which satisfies these assumptions is 
$$
\begin{aligned}
R\left(\eta, \partial_t \eta\right) & :=\int_Q\left|\left(\nabla \partial_t \eta\right)^T \nabla \eta+(\nabla \eta)^T\left(\nabla \partial_t \eta\right)\right|^2 d x=\int_Q\left|\partial_t\left(\nabla \eta^T \nabla \eta\right)\right|^2 d x, \\
E(\eta) & := \begin{cases}\int_Q\left[\frac{1}{8}\left|\nabla \eta^T \nabla \eta-I\right|_{\mathcal C}+\frac{1}{(\operatorname{det} \nabla \eta)^a}+\frac{1}{q}\left|\nabla^2 \eta\right|^q\right] d x & \text { if det } \nabla \eta>0 \text { a.e. in } Q, \\
+\infty & \text { otherwise, }\end{cases}
\end{aligned}
$$
denoting $\left|\nabla \eta^T \nabla \eta-I\right|_{\mathcal C}:=\left(\mathcal C\left(\nabla \eta^T \nabla \eta-I\right)\right) \cdot\left(\nabla \eta^T \nabla \eta-I\right)$ with $\mathcal C$
being a positive definite tensor of elastic constants, and $a>\frac{q d}{q-d}$. For a more detailed discussion of this model see \cite{benesovaVariationalApproachHyperbolic2023}, and in particular for the discussion of the non-convexity estimate \eqref{as:E-nonconvexity-estimate} see \cite{cesikStabilityConvergenceTime2023}.

%\todo{\bf CHECK WHAT NEEDS TO BE DONE IN ORDER TO CLAREFY THE ERROR WITH $p<q$ replaced by $q$ above.}

\subsection{The approximation}

The construction of the solution will go through three levels of approximation that we shall briefly describe here. 

\subsubsection*{Spatial regularization -- $\kappa$ level } 
At the $\kappa$-level we introduce in the problem a $W^{k_0+2,2}$-regularization of the solid energy and dissipation, as well as a $W^{k_0,2}$-regularization of the fluid dissipation. Here $k_0$ is chosen so large that $W^{k_0,2}\hookrightarrow W^{2,q}$. The equations with this regularization are then quadratic in the highest order. Thus using $(\partial_t\eta,v)$ as a test function is admissible as long as $\kappa>0$. The two gradients more for the regularity of $\eta$ are used in the approximation of test functions in Proposition \ref{prop:approximation-coupled-test-functions}, since there we use the spatial extension of the normal $n$ (recall this is the normal to $\partial\eta(\cdot,Q)$) to the fluid domain.

\subsubsection*{Time-delayed equation -- $h$ level} 
The $h$-level corresponds to replacing the equations by a \emph{time-delayed equation} where the second time derivative is discrete with scale $h$, and the first time derivative is continuous. So that then $$\partial_{t t} \eta \approx \tfrac{\partial_t \eta-\partial_t \eta(t-h)}{h}, \quad \partial_t v+v \cdot \nabla v \approx \tfrac{v \circ \Phi_h-v(t-h)}{h}
$$
where $\Phi_h$ is the flow map, defined by $\partial_t\Phi_t=v(t)\circ \Phi_t$ and $\Phi_0=\operatorname{id}$. The flow map has to be constructed along with the solution. 

\subsubsection*{Minimizing movements -- $\tau$ level}
The first level is the minimizing movement approximation of the time-delayed problem with $\tau$-steps. Note that the time-delayed equation can, on an interval of length $h$, be seen as a gradient flow where $\partial_t\eta(\cdot-h)$ and $v(\cdot -h)$ can be seen as a given external force. Indeed, these quantities are already known from the previous $h$-interval. Note that to get the $\tau$-independent estimate, we do this already on the discrete $\tau$ level based on \cite{cesikStabilityConvergenceTime2023}.

The approximation passes to the limits in the following order $\tau\to 0$, $h\to0$, $\kappa\to0$. This then results in the solution of the full problem.

For a more detailed discussion of the scheme we refer the reader to \cite{benesovaVariationalApproachHyperbolic2023}. Keep in mind that in this original paper, the parameters $\kappa$ and $h$ are tied together and $h\to0$, $\kappa\to0$ is performed simultaneously. We chose to separate these two parameters for more clarity and also for future work which will include collisions, since these happen only as a result of the $\kappa\to 0$ limit (see Corollary \ref{cor:no-contact-regularization}).
\section{Preliminaries}\label{sec:preliminaries}

Here we state some notions for time-dependent domains and the respective
function spaces. These will be later used for the Eulerian fluid domain.

\subsection{Moving domains}\label{sec:moving-domains}

Consider a domain variable in time, that is let
\(\Omega(t)\subset \mathbb R^d\) be a Lipschitz domain for each
\(t\in(0,T)\), such that the time-space domain
\[\Omega_T:=(0,T)\times \Omega(t)= \{(t,x)\in (0,T)\times \mathbb R^d: x\in \Omega (t) \} = \bigcup_{t\in (0,T)} \{t\}\times \Omega(t)\]is
open in \(\mathbb R^{d+1}\). We say it is a moving domain, if
\(\Omega_T\) addtionally has Lipschitz boundary in \(\mathbb R^{d+1}\). We
adopt the notation that by \(\partial_x\Omega_T\) we mean only the
lateral (spatial part) of the boundary, that is
\[\partial_x\Omega_T:= \{(t,x)\in(0,T)\times \mathbb R^d: x\in \partial \Omega(t)\} = \bigcup _{t\in(0,T)} \{t\}\times \partial \Omega(t)\]
so that the true time-space boundary in \(\mathbb R^{d+1}\) of
\(\Omega_T\) is
\[\partial\Omega_T= \partial_x\Omega_T \cup \{0\}\times \overline\Omega(0) \cup \{T\}\times \overline\Omega(T)\]
Denote by \(n(t)= n(t,\cdot)\colon \partial \Omega(t)\to \mathbb R^d\)
the outer unit normal to the Lipschitz boundary of \(\Omega(t)\), defined
\(\mathcal H^{d-1}\)-a.e. on \(\partial\Omega(t)\). Since we assume
\(\Omega_T\subset\mathbb R^{d+1}\) to be also a Lipschitz domain, it has
\(\mathcal H^d\)-a.e. defined normal
\(\tilde n\colon \partial\Omega_T\to \mathbb R^{d+1}\) which necessarily
is of the form \[\tilde n (t,x)=\begin{cases}
(-\tilde n_t(t,x), n(t,x)), & t\in(0,T), x\in \partial\Omega(t),\\
(-1,0), & t=0, x\in \Omega(0),\\
(1,0), & t=T, x\in \Omega(T),
\end{cases}\] for some
\(\tilde n_t\colon \partial_x \Omega_T \to \mathbb R\) which we may call the \emph{normal velocity} (note that in all
three cases the value on the right is written as
\((a,b)\in\mathbb R\times \mathbb R^d\)). Note that $\tilde n$ is not a unit vector, rather it is chosen so that the spatial part $n$ is unit. We denote by $dS$ the surface measure on $\partial\Omega(t)$, and $dn$ denotes the vector-valued measure $dn:= n\,dS$.
%picture:
%!{[}{[}2024\_01\_29 11\_13 Office Lens.jpg{]}{]}
The domain \(\Omega(t)\) said is transported by the vector field
\(w\colon \overline{\Omega_T}\to\mathbb R^d\), if
\(w\cdot n=\tilde n_t\).

\subsection{Transport theorem}\label{sec:transport-theorem}

Now let us prove a variant of the Reynolds transport theorem for this type of
moving domain. We do not explicitly refer to a transport field in the
statement. However if such a field exists, i.e.~\(\Omega(t)\) is
transported by \(w\) in the sense above, we recover the standard
Reynolds transport theorem.

\begin{theorem}[Transport theorem]\label{thm:transport-theorem}
 Let \(\Omega(t)\) be a moving domain as above and
\(u\in C^1(\overline \Omega_T)\). Then it holds for almost all \(t\in(0,T)\)
\[\frac{d}{dt}\int_{\Omega(t)}u(t,x)\,dx = \int_{\Omega(t)}\partial_t u(t,x)\,dx + \int_{\partial\Omega(t)
} u(t,x)\tilde n_t (t,x) \,d\mathcal H^{d-1}(x) \] 
\end{theorem}

\begin{proof}
 Denote
\(U=(u,0)\colon \Omega_T\to \mathbb R\times \mathbb R^d\). Pick
\(0\leq t<s\leq T\) and invoke the divergence theorem for \(U\) on the
Lipschitz domain \(\Omega_{s,t}:=\Omega_T \cap (s,t)\times \mathbb R^d\)
to obtain
\[\int_{\Omega(s)} u(s,x)\,dx-\int_{\Omega(t)}u(t,x)\,dx = \int_t^s \int_{\Omega(r)} \partial_t u(r,x)\,dx\,dr + \int_t^s \int_{\partial\Omega(r)} u(r,x) \tilde n_t(r,x)\,d\mathcal H^{d-1}(x)\,dr.\]
Let us now pick
\(t\in(0,T)\) a Lebesgue point of
\(t\mapsto \int_{\Omega(t)} u(t,x)\,dx\) and of \(t\mapsto\tilde n_t\).
Then we compute, by above,
\[\begin{aligned}\frac{d}{dt}\int_{\Omega(t)}u(t,x)\,dx = \lim_{s\to t} \frac{\int_{\Omega(s)}u(s,x)\,dx-\int_{\Omega(t)}u(t,x)\,dx}{s-t}  \\ = \int_{\Omega(t)} \partial_t u(t,x)\,dx +\int_{\partial \Omega(t)}u(t,x)\tilde n_t(t,x)\,d\mathcal H^{d-1}(x), \end{aligned}\]
which finishes the proof.
\end{proof}

\begin{remark}
 It is enough that \(u\) differentiable in time and
having a trace, so that all manipulations in the preceding proof go
through. However we will use it only for \(C^1\) strong solutions in Theorem \ref{thm:weak-strong-compatibility}.
\end{remark}

\subsection{Spaces on a moving domain}\label{sec:spaces-on-a-moving-domain}
For the fluid, we will need to deal with Sobolev spaces on a changing domain and talk about the relevant convergences, in the absence of a reference configuration. While far enough away from the reference configuration this can be done by imposing an artificial reference state and using so called ALE (arbitrary Lagrangian-Eulerian) methods, we want our results to be extendable up to collisions. So instead we adopt the framework of zero-extension convergence (see \cite{evseevZeroextensionConvergenceSobolev2024} and a further upcoming work by the same authors).

For the reader's convenience we will summarize the basic concepts in this section.
%\noteMalte[inline]{I still want to shorten this section a bit, but probably not by much.}
Assume now \(\Omega(t)\) is a moving domain.

We now rigorously describe the spaces \[\begin{aligned}
L^2((0,T);W^{1,2}(\Omega(t);\mathbb R^d)),\\
L^2((0,T);W_n^{1,2}(\Omega(t);\mathbb R^d)),\\
L^2((0,T);W_{\operatorname{div},n}^{1,2}(\Omega(t);\mathbb R^d)).
\end{aligned}\]
First, classically we can identify the space
\(L^2((0,T);L^2(\Omega(t);\mathbb R^d))\) with
\(L^2(\Omega_T;\mathbb R^d)\). That is, we consider measurable functions
\(u\colon \Omega_T\to \mathbb R^d\) which are square integrable:
\[\int_{\Omega_T}|u|^2\,dx\,dt<\infty\] so that by Fubini theorem, we
have for a.e. \(t\in(0,T)\) that \(u(t)\in L^2(\Omega(t);\mathbb R^d)\)
and it holds
\[ \| u\|_{L^2((0,T);L^2(\Omega(t);\mathbb R^d))} :=\int_0^T \int_{\Omega(t)}|u|^2 \,dx\,dt<\infty\]
Let us henceforth assume \(\Omega(t)\subset \Omega\), \(t\in(0,T)\), for
some fixed Lipschitz domain \(\Omega\subset\mathcal R^d\). Clearly we can consider the zero
extension to \(u_0\colon (0,T)\times \Omega\to \mathbb R^d\) by
\[u_0(t,x)= \begin{cases}u(t,x), &x\in \Omega(t), \\ 0, & x\in \Omega\setminus\Omega(t)\end{cases}\]
so that then \(u_0\in L^2((0,T);L^2(\Omega;\mathbb R^d))\). In this
sense, we can understand the embedding
\[L^2((0,T);L^2(\Omega(t);\mathbb R^d))\hookrightarrow L^2((0,T);L^2(\Omega;\mathbb R^d)).\]
Now we say that \(u\in L^2((0,T);W^{1,2}(\Omega(t);\mathbb R^d))\) if it
holds

\begin{itemize}
\item
  \(u\in  L^2((0,T);L^2(\Omega(t);\mathbb R^d))\)
\item
  for a.e. \(t\in (0,T)\) we have
  \(u(t)\in W^{1,2}(\Omega(t);\mathbb R^d)\)
\item
  \(\nabla u\in L^2((0,T);L^2((\Omega(t);\mathbb R^{d\times d}))\),
  where the function
  \(\nabla u\colon \Omega_T\to \mathbb R^{d\times d}\) is a.e.\ defined
  by the previous point.
\end{itemize}

The norm on this space is defined as
\[\|u\|_{L^2((0,T);W^{1,2}(\Omega(t);\mathbb R^d))}:=\| u\|_{L^2((0,T);L^2(\Omega(t);\mathbb R^d))} +\| \nabla u\|_{L^2((0,T);L^2(\Omega(t);\mathbb R^{d\times d}))} \]
We now can consider the mapping \(u\mapsto (u_0, (\nabla u)_0)\) to be
an embedding
\[L^2((0,T);W^{1,2}(\Omega(t);\mathbb R^d))\hookrightarrow L^2((0,T);L^2(\Omega;\mathbb R^d))\times L^2((0,T);L^2(\Omega;\mathbb R^{d\times d})).\]
We now define the weak convergence in
\(L^2((0,T);W^{1,2}(\Omega(t);\mathbb R^d))\) by weak convergence in the
space on the right. Explicitly, we say that
\[ u_n\rightharpoonup u \quad \text{in }L^2((0,T);W^{1,2}(\Omega(t);\mathbb R^d))\]
if it holds \[\begin{aligned}
(u_n)_0 \rightharpoonup u_0 \quad \text{in }L^2((0,T);L^2(\Omega;\mathbb R^d)), \\
(\nabla u_n)_0 \rightharpoonup (\nabla u)_0 \quad \text{in }L^2((0,T);L^2(\Omega;\mathbb R^{d\times d})).
\end{aligned}\] It is easy to see that the Banach-Alaoglu theorem
continues to hold: If
\[\|u_n\|_{L^2((0,T);W^{1,2}(\Omega(t);\mathbb R^d))} \leq C\] then we
can choose subsequence (not relabeled) such that 
\[\begin{aligned}
(u_n)_0 \rightharpoonup w \quad \text{in }L^2((0,T);L^2(\Omega;\mathbb R^d)), \\
(\nabla u_n)_0 \rightharpoonup A \quad \text{in }L^2((0,T);L^2(\Omega;\mathbb R^{d\times d})).
\end{aligned}\]
Clearly, \(\operatorname{supp} w\subset \Omega_T\), so
that \(w=u_0\) for some \(u\in L^2((0,T);L^2(\Omega(t);\mathbb R^d))\).
By choosing \(\Phi\in C^\infty_c(\Omega_T;\mathbb R^{d\times d})\) we
find that
\[\int_0^T \int_{\Omega(t)} u_n\cdot \operatorname{div}\Phi \,dx\,dt \to \int_0^T\int_{\Omega(t)} u\cdot\operatorname{div} \Phi\,dx\,dt, \quad \text{as }n\to\infty,
\]
where by integration by parts on the left this is equal to 
\[\int_0^T \int_{\Omega(t)} u_n\cdot \operatorname{div}\Phi \,dx\,dt= -\int_0^T \int_{\Omega(t)}\nabla u_n :\Phi \,dx \to -\int_0^T \int_{\Omega(t)} A:\Phi \,dx\,dt,\]
showing that \(A=(\nabla u)_0\).

We define the subspace of functions that have zero normal trace as
\[\begin{aligned}
L^2((0,T);&W^{1,2}_n(\Omega(t);\mathbb R^d)) \\ := &\{u\in L^2((0,T);W^{1,2}(\Omega(t);\mathbb R^d)): u(t)|_{\partial \Omega(t)} \cdot n(t)=0 \text{ on }\partial\Omega(t)\text{ for a.e. }t\in(0,T)\} 
\end{aligned}\]where
\(u(t)|_{\partial\Omega(t)}\in W^{1/2,2}(\partial\Omega(t);\mathbb R^d)\)
is the trace of \(u(t)\). It is clear that this subspace is closed under
the weak convergence in \(L^2((0,T);W^{1,2}(\Omega(t);\mathbb R^d))\),
by compactness of the trace operator.

Finally, we consider also the space of divergence-free functions:
\[\begin{aligned}
L^2((0,T);W^{1,2}_{\operatorname{div}}(\Omega(t);\mathbb R^d)):= \{u\in L^2((0,T);W^{1,2}(\Omega(t);\mathbb R^d)): \operatorname{div} u=0\}, \\
L^2((0,T);W^{1,2}_{\operatorname{div},n}(\Omega(t);\mathbb R^d)):= \{u\in L^2((0,T);W_n^{1,2}(\Omega(t);\mathbb R^d)): \operatorname{div} u=0\}.
\end{aligned}\]

We call the moving domain admissible if there exists
\(w\in L^2((0,T);W^{1,2}_{\operatorname{div}}(\Omega(t);\mathbb R^d))\)
with \(w\cdot n =\tilde n_t\).

To define spaces with time derivative
\[W^{1,2}((0,T);W^{1,2}(\Omega(t);\mathbb R^d)),\] we proceed as
follows. We say that
\[u\in W^{1,2}((0,T);W^{1,2}(\Omega(t);\mathbb R^d)),\] if there exists
\(\partial_t u\in L^2((0,T);W^{1,2}(\Omega(t);\mathbb R^d)),\) which is
a time derivative of \(u\) in the sense that for any
\(\varphi\in C^1_c(\Omega_T)\) it holds
\[\int_0^T\int_{\Omega(t)} \partial_t u \cdot \varphi \,dx\,dt = -\int_0^T\int_{\Omega(t)} u \cdot \partial_t\varphi \,dx\,dt.\]
Spaces with higher derivatives or different integrability shall be
defined analogously.

\subsubsection{Convergence of domains}\label{sec:convergence-of-domains}

Now consider a sequence of moving domains. That is, for 
\(i\in\mathbb N\) each \(\Omega_T^{(i)}\) is a time dependent domain, as
well as \(\Omega_T\). We say that \(\Omega_T^{(i)}\to \Omega_T\) if both the domains and their boundaries
converge in the Hausdorff metric in space, uniformly in time, that is
\[\sup_{t\in(0,T)} \mathcal H(\Omega_T^{(i)}, \Omega_T) \to 0 \quad \text{ and } \quad \sup_{t\in(0,T)} \mathcal H(\partial \Omega_T^{(i)}, \partial \Omega_T) \to 0\]
\[\text{where }\mathcal H(\Omega_T^{(i)}, \Omega_T) = \max\left(\sup_{x\in\Omega_T^{(i)}}\operatorname{dist}(x, \Omega_T), \sup_{y\in\Omega_T}\operatorname{dist}(y, \Omega_T^{(i)})\right)\]
as \(i\to\infty\). It is easy to see that it implies the convergence of
their characteristic functions in \(L^1\):
\[\chi_{\Omega_T^{(i)}}\to\chi_{\Omega_T}\quad \text{in } L^1((0,T)\times\Omega)\]
and in fact for any any \(L^q\), \(1\leq q<\infty\), as can be seen by
\[\left\|\chi_{\Omega_T^{(i)}}-\chi_{\Omega_T}\right\|_{L^q((0,T)\times \Omega)} = \left| \chi_{\Omega_T} \vartriangle  \chi_{\Omega_T^{(i)}}\right|^{1/q} = \left\|\chi_{\Omega_T^{(i)}}-\chi_{\Omega_T}\right\|_{L^1((0,T)\times \Omega)}^{1/q},\]
where \(\vartriangle\) denotes the symmetric difference. Observe also
that whenever \(K\subset \Omega_T\) is a compact set, then
\(K\subset \Omega_T^{(i)}\) for all \(i\) large enough. This will later
be in particular important for us when \(K\) is the support of our
chosen test function.

\subsubsection{Convergence of functions on different domains}\label{convergence-of-functions-on-different-domains}

Let
\(u^{(i)} \in L^2((0,T);W^{1,2}(\Omega^{(i)}(t);\mathbb R^d))\) and
\(u\in L^2((0,T);W^{1,2}(\Omega(t);\mathbb R^d))\) where we have a
sequence of converging time dependent domains
\(\Omega_T^{(i)}\to \Omega_T\). The space normals to these domains are
denoted by \(n\) the normal to \(\Omega(t)\), \(n^{(i)}\) the normal to
\(\Omega^{(i)}\).

In this context we can define the weak convergences by zero extension,
that is we say that
\[u^{(i)}\overset\eta\rightharpoonup u\quad\text{in }L^2((0,T);W^{1,2}(\Omega(t);\mathbb R^d))\]
if it holds 
\[\begin{aligned}
u^{(i)}_0 \rightharpoonup u_0 &\quad \text{in }L^2((0,T);L^2(\Omega;\mathbb R^d)), \\
(\nabla u^{(i)})_0 \rightharpoonup (\nabla u)_0 &\quad \text{in }L^2((0,T);L^2(\Omega;\mathbb R^{d\times d})).
\end{aligned}\]
We can see that the Banach-Alaoglu theorem continues to
hold: If
\[\|u^{(i)}\|_{L^2((0,T);W^{1,2}(\Omega^{(i)}(t);\mathbb R^d))} \leq C\]
then we can choose subsequence (not relabeled) such that
\[\begin{aligned}
(u^{(i)})_0 \rightharpoonup w \quad \text{in }L^2((0,T);L^2(\Omega;\mathbb R^d)) \\
(\nabla u^{(i)})_0 \rightharpoonup A \quad \text{in }L^2((0,T);L^2(\Omega;\mathbb R^{d\times d}))
\end{aligned}\] By the convergence \(\Omega_T^{(i)}\to \Omega_T\) we can
easily see \(\operatorname{supp} w\subset \Omega_T\), so that \(w=u_0\)
for some \(u\in L^2((0,T);L^2(\Omega(t);\mathbb R^d))\). By choosing
\(\Phi\in C^\infty_c(\Omega_T;\mathbb R^{d\times d})\) we find that for
any \(i\) large enough enough it holds
\(\Phi\in C^\infty_c(\Omega_T^{(i)};\mathbb R^{d\times d})\), so that
\[\int_0^T \int_{\Omega^{(i)}(t)} u^{(i)}\cdot \operatorname{div}\Phi \,dx\,dt\to \int_0^T\int_{\Omega(t)} u\cdot\operatorname{div} \Phi\,dx\,dt \]
is equal to
\[= -\int_0^T \int_{\Omega^{(i)}(t)}\nabla u^{(i)} :\Phi \,dx \to -\int_0^T \int_{\Omega(t)} A:\Phi \,dx\,dt\]
(remember that \(\chi_{\Omega^{(i)}}(t)\) converges strongly) showing
that \(A=(\nabla u)_0\).

\begin{proposition}
If
\(u^{(i)} \in L^2((0,T);W_n^{1,2}(\Omega^{(i)}(t);\mathbb R^d))\) and
\(u^{(i)}\overset\eta\rightharpoonup u\), then the limit satisfies
\(u\in L^2((0,T);W_n^{1,2}(\Omega(t);\mathbb R^d))\). In other words,
zero normal trace is preserved under weak convergence.
\end{proposition}

\begin{proof}
Let us take a test function
\(\xi\in C^1_c((0,T)\times \mathbb R^d)\). Then we compute for a.e.
\(t\in(0,T)\) by the divergence theorem
\[\begin{aligned}0=\int_0^T\int_{\partial\Omega^{(i)}(t)} \xi u^{(i)} \cdot dn^{(i)}\,dt= \int_0^T\int_{\Omega^{(i)(t)}} \nabla\xi \cdot u^{(i)} + \xi\operatorname{div} u^{(i)} \,dx\,dt \\ \to \int_0^T \int_{\Omega(t)}\nabla\xi \cdot u+\xi \operatorname{div} u\,dx\,dt = \int_0^T \int_{\partial\Omega(t)}\xi u \cdot dn\,dt\end{aligned}\]
and since \(\xi\) was arbitrary, this shows \(u\cdot n=0\) on
\(\partial_x\Omega_T\). The convergence for \(i\to\infty\) goes through,
as it holds \[\begin{aligned}
u^{(i)}_0 \rightharpoonup u_0 \quad \text{in }L^2((0,T);L^2(\Omega;\mathbb R^d)), \\
(\operatorname{div} u^{(i)})_0 \rightharpoonup (\operatorname{div} u)_0 \quad \text{in }L^2((0,T);L^2(\Omega;\mathbb R^d)), \\
\chi_{\Omega_T^{(i)}}\to\chi_{\Omega_T}\quad \text{in } L^2((0,T);L^2(\Omega;\mathbb R^d)).
\end{aligned}\]
\end{proof}
In fact, we can also treat converging normal boundary values.
\begin{lemma}[Convergence of normal trace] Let
\(\Omega^{(i)} \to \Omega(t)\), \(u^{(i)}\overset\eta\rightharpoonup u\)
in \(L^2((0,T);W^{1,2}(\Omega(t);\mathbb R^d))\) and
\(u^{(i)}\cdot n^{(i)}|_{\partial \Omega^{(i)}}= \phi^{(i)}\) with
\(\phi^{(i)}\colon \partial\Omega^{(i)}\to \mathbb R\) given such that
\(\phi^{(i)}\rightharpoonup \phi\) in the sense that
\(\phi^{(i)}\mathcal H^{d-1}|_{\partial \Omega^{(i)}} \overset\ast\rightharpoonup \phi \mathcal H^{d-1}|_{\partial\Omega(t)}\)
as measures in \(M([0,T]\times \mathbb R^d)\), i.e. for all
\(\psi\in C_c([0,T]\times \mathbb R^d)\) it holds
\[\int_0^T\int_{\partial\Omega^{(i)}} \psi\phi^{(i)}\, dS\,dt \to \int_0^T\int_{\partial\Omega(t)}\psi\phi\,dS\,dt.\]
Then it holds \(u\cdot n|_{\partial\Omega(t)}=\phi\). In other words
\[u^{(i)}\overset\eta\rightharpoonup u \implies u^{(i)}\cdot n^{(i)}\mathcal H^{d-1}|_{\partial \Omega^{(i)}} \overset\ast\rightharpoonup u\cdot n \mathcal H^{d-1}|_{\partial\Omega(t)}\]
\end{lemma}

\begin{proof}
 It is sufficient to show that for all
\(\psi\in C([0,T]\times \mathbb R^d)\) it holds that
\[\int_0^T\int_{\partial\Omega(t)} \psi u\cdot dn = \int_0^T \int_{\partial\Omega(t)}\psi \phi\, dS.\]
We know that
\[\int_0^T \int_{\partial \Omega^{(i)}(t)} \psi u^{(i)}\cdot dn^{(i)}\,dt = \int_0^T\int _{\partial\Omega^{(i)}(t)}  \psi\phi^{(i)} \,dS\,dt\to \int_0^T\int_{\partial\Omega(t)}\psi\phi\,dS\,dt\]
Rewrite the left side by the divergence theorem
\[\begin{aligned}\int_0^T\int_{\partial \Omega^{(i)}(t)} \psi u^{(i)}\cdot dn^{(i)}\,dt = \int_0^T \int_{\Omega^{(i)}(t)}\nabla \psi \cdot u^{(i)} + \psi \operatorname{div} u^{(i)} \,dx\,dt \\ \to \int_0^T \int_{\Omega(t)}\nabla\psi \cdot u+\psi \operatorname{div} u\,dx\,dt = \int_0^T \int_{\partial \Omega(t)}\psi u\cdot dn\,dt\end{aligned}\]
and we are finished.
\end{proof}

\subsection{Approximation of test functions}\label{approximation-of-test-functions}

\subsubsection{Eulerian to Lagrangian and back}\label{eulerian-to-lagrangian-and-back}

We will have the deformation \(\eta\) mapping the Lagrangian solid \(Q\)
to the deformed configuration \(\eta(Q)\). Here we shall see that under
the \(W^{2,q}\) regularity, one can switch between Lagrangian and
Eulerian domains at no cost. This statement is made precise below.

\begin{lemma}\label{lem:eulerian-lagrangian-w2q}
Let
\(\eta\in W^{1,2}((0,T);W^{1,2}(Q;\mathbb R^d)\cap L^\infty((0,T); W^{2,q}(Q;\mathbb R^d))\)
with \(\det \nabla \eta\geq \varepsilon_0\) and \(\eta(t,\cdot)\)
injective for all \(t\). The mapping \(\xi\mapsto\xi\circ\eta\) is
linear bounded operator between the spaces
\[\begin{aligned}
L^1((0,T);W^{2,q}(\eta(Q);\mathbb R^d))&\cap W^{1,1}_0((0,T);L^2(\eta(Q);\mathbb R^d))\\ &\to L^1((0,T);W^{2,q}(Q;\mathbb R^d)) \cap W^{1,1}_0((0,T);L^2(Q;\mathbb R^d))\end{aligned}\]
with the bound \(C_\eta\) depending only on \(\varepsilon_0\),
\(\|\eta \|_{L^\infty((0,T);W^{2,q}(Q;\mathbb R^d))}\) and
\(\|\partial_t\eta\|_{L^\infty((0,T);L^2(Q;\mathbb R^d))}\).
\end{lemma}

\begin{proof}
 The linearity is clear. As in
\cite[A.4]{benesovaVariationalApproachHyperbolic2023}, we calculate
\[\begin{aligned}
\|\nabla^2 (\xi\circ\eta)\|_{L^q(Q)}= \|(\nabla^2 \xi \circ\eta\cdot \nabla\eta)\cdot \nabla\eta + \nabla\xi\circ\eta\cdot \nabla^2\eta\|_{L^q(Q)}\\ \leq \|\nabla\eta\|^2_{L^\infty(Q)}\|\nabla^2 \xi\circ\eta\|_{L^q(Q)}+\|\nabla\xi\|_{L^\infty(\eta(Q))}\|\nabla^2\eta\|_{L^q(Q)}
\end{aligned}\]
and use
\[\|\nabla^2\xi\circ\eta\|_{L^q(Q)}^q\leq \int_Q|\nabla^2\xi\circ\eta|^q\tfrac{\det \nabla\eta}{\varepsilon_0}\,dx = \frac{1}{\varepsilon_0} \|\nabla^2\xi\|_{L^q(\eta(Q))}^q\]
which then proves
\[\|\xi\circ\eta\|_{L^1((0,T);W^{2,q}(Q;\mathbb R^d))} \leq C_\eta\|\xi\|_{L^1((0,T);W^{2,q}(\eta(t,Q);\mathbb R^d))}.\]
Now for the time derivative, compute
\[\partial_t(\xi\circ\eta)=\partial_t\xi\circ\eta+(\nabla\xi\circ\eta)\partial_t\eta.\]
In the first term, we compute
\[\|\partial_t\xi\circ\eta\|_{L^2(Q;\mathbb R^d)}^2 \leq \int_Q|\partial_t\xi\circ \eta|^2 \tfrac{\det\nabla\eta}{\varepsilon_0}\,dx =\frac{1}{\varepsilon_0} \|\partial_t\xi\|_{L^2(\eta(Q);\mathbb R^d)}^2.\]
In the second term, with \(C\) coming from the Morrey inequality, we have
\[\begin{aligned}
\|(\nabla\xi\circ\eta)\partial_t\eta\|_{L^2(Q;\mathbb R^d)}\leq \|\nabla\xi\circ\eta\|_{L^\infty(Q;\mathbb R^d)}\|\partial_t\eta\|_{L^2(Q;\mathbb R^d)}&=\|\nabla\xi\|_{L^\infty(\eta(Q);\mathbb R^d)}\|\partial_t\eta\|_{L^2(Q;\mathbb R^d)}\\&\leq C\|\nabla^2\xi\|_{L^q(\eta(Q);\mathbb R^d)}\|\partial_t\eta\|_{L^2(Q;\mathbb R^d)}.
\end{aligned}\]
In total, this shows that
\[\|\partial_t(\xi\circ\eta)\|_{L^1((0,T);L^2(Q))} \leq \frac{1}{\varepsilon_0}\|\partial_t\xi\|_{L^1((0,T);L^2(Q))} + C\|\partial_t\eta\|_{L^\infty((0,T);L^2(Q))} \|\xi\|_{L^1((0,T);W^{2,q}(\eta(Q)))}\]
which concludes the proof.
\end{proof}

\begin{lemma}\label{lem:eulerian-lagrangian-wk02}
Let
\(\eta\in W^{2,q}(Q;\mathbb R^d)\cap W^{k_0,2}(Q;\mathbb R^d)\) with
\(\det\nabla\eta\geq \varepsilon_0\) be globally injective. Then the
mapping \[\xi\mapsto \xi\circ\eta\] is an isomorphism
\(W^{k_0,2}(\eta(Q);\mathbb R^d) \to W^{k_0,2}(Q;\mathbb R^d)\) with
norm depending only on \(\|\eta\|_{W^{2,q}(Q;\mathbb R^d)}\) and
\(\varepsilon_0\).
\end{lemma}

\begin{proof}
The proof follows from application of the chain rule to
\(\nabla^{k_0}(\xi\circ\eta)\), the embedding
\(W^{2,q}\hookrightarrow W^{1,\infty}\) and interpolation in the product
which results from this chain rule; see
\cite[Lemma A.3]{benesovaVariationalApproachHyperbolic2023}.
\end{proof}

\subsubsection{Extending the divergence-free domain}\label{extending-the-divergence-free-domain}

We recall here \cite[Proposition 2.22]{benesovaVariationalApproachHyperbolic2023}, which will be useful for convergences in the
coupled equation.

\begin{proposition}[Approximation of coupled test functions]\label{prop:approximation-coupled-test-functions} Fix a function \[
\eta \in L^{\infty}([0, T] ; \mathcal{E}) \cap W^{1,2}([0, T] ; W^{1,2}(Q ; \mathbb{R}^d)) \quad \text { with } \sup _{t \in T} E(\eta(t))<\infty
\] such that \(\eta(t) \notin \partial \mathcal{E}\), i.e.\ away from (self-)collision, for all
\(t \in[0, T]\). As before, set
\(\Omega(t)=\Omega \backslash \eta(t, Q)\). Let \(\mathcal{T}_\eta\) be
the set of admissible coupled test functions, defined as \[
\begin{aligned}
& \mathcal{T}_\eta:=\left\{(\phi, \xi) \in W^{1,2}\left([0, T] ; W^{1,2}\left(Q ; \mathbb{R}^d\right)\right) \times L^2\left([0, T] ; W_0^{1,2}\left(\Omega ; \mathbb{R}^d\right)\right) :\right. \\
& \phi=\xi \circ \eta \text { on }[0, T] \times Q \text { and } \operatorname{div} \xi(t)=0 \text { in } \Omega(t)\} .
\end{aligned}
\]

Then the set \[
\tilde{\mathcal{T}}_\eta:=\left\{(\phi, \xi) \in \mathcal{T}_\eta \mid \xi \in C^{\infty}\left([0, T] ; C_0^{\infty}\left(\Omega ; \mathbb{R}^d\right)\right), \phi = \xi \circ \eta, \operatorname{supp} \operatorname{div} \xi(t, y) \subset \subset \eta(t,Q)\right\}
\] for all \(t \in[0, T]\) and all \(y\) with
\(\operatorname{dist}(y, \Omega(t))<\delta\) for some
\(\delta>0\) is dense in \(\mathcal{T}_\eta\) in the following
sense: For every \(\delta\) sufficiently small there exists a
linear map
\((\phi, \xi) \mapsto\left(\phi_{\delta}, \xi_{\delta}\right) \in \tilde{\mathcal{T}}_\eta\)
such that \[
\operatorname{div}\left(\xi_{\delta}(t, y)\right)=0 \quad \text { for all } y \in \Omega \text { with } \operatorname{dist}(y, \Omega(t)) \leq \delta \text {. }
\] Moreover, if \(\xi\in L^2((0,T);W^{k_0,2}(\Omega(t);\mathbb R^d)\)
then
\[\xi_\delta\to\xi \quad\text{in }L^2((0,T);W^{k_0,2}(\Omega(t);\mathbb R^d)\quad \text{as }\delta\to 0\]and
if \(\eta\in L^\infty((0,T);W^{k_0+2,2}(Q;\mathbb R^d))\), then
\[\phi_\delta\to\phi \quad\text{in }L^\infty((0,T);W^{k_0+2,2}(Q;\mathbb R^d))\cap W^{1,2}((0,T);W^{1,2}(Q;\mathbb R^d))\]
Moreover the following bounds hold \[
\begin{aligned}
\left\|\xi_{\delta}(t)\right\|_{W^{1,2}(\Omega(t);\mathbb R^d)} & \leq c\|\xi(t)\|_{W^{1,2}\left(\Omega (t); \mathbb{R}^d\right)} \\
\left\|\xi_{\delta}(t)-\xi(t)\right\|_{L^2(\Omega(t))} & \leq c \delta^{\frac{2}{d+2}}\|\xi(t)\|_{W^{1,2}(\Omega(t))} \\
\left\|\xi_{\delta}(t)\right\|_{W^{k, a}(\Omega(t))} & \leq c(\delta)\|\xi(t)\|_{L^2\left(\Omega(t) ; \mathbb{R}^n\right)} \\
\left\|\phi_{\delta}(t)\right\|_{W^{k, a}(Q)} & \leq c\|\xi(t)\|_{C^k(\Omega)}\|\eta(t)\|_{W^{k, a}(Q)} \leq c(\delta)\|\xi(t)\|_{L^2(\Omega)}\|\eta(t)\|_{W^{k, a}(Q)}
\end{aligned}
\]
\end{proposition}
For a proof, see
\cite[Proposition 2.22]{benesovaVariationalApproachHyperbolic2023}.

\subsubsection{Universal Bogovskii operator}\label{sec:universal-bogovskii}

We state here a version of the universal Bogovskii operator. It is
universal in the sense that we have the very same operator for domains
which are similar enough, in a suitable sense. Classically, similar
enough means star shaped with respect to the same ball
\cite{geissertEquationDivBogovskii2006}, and graph domains in the
same coordinates
\cite{kampschulteUnrestrictedDeformationsThin2023}. We shall
further extend this to Lipschitz domains which correspond to close
deformations.

\begin{theorem}[Universal Bogovskii]\label{thm:bogovskii-unviersal} Let \(\Omega\) be a bounded
\(L\)-Lipschitz domain. Fix a finite covering \(\mathcal G\) of its
boundary by \(L\)-Lipschitz graphs; by this we mean that \(\mathcal G\)
consists of open rectangles such that for each \(G\in\mathcal G\),
\(G\cap \Omega\) is a subgraph of an \(L\)-Lipschitz function (in some
direction). Let \(b\in  C^\infty_c([0,T]\times\Omega)\) with
\(\int_\Omega  b(t)=1\) for all \(t\in[0,T]\). Then there exists a
Bogovskii operator
\(\mathcal B \colon C^\infty_c (\Omega)\to C^\infty_c(\Omega;\mathbb R^d)\)
which satisfies
\[\operatorname{div} \mathcal B f = f- b\int_{\Omega} f\]and which is universal in the sense that
for any \(L\)-Lipschitz domain \(\widetilde\Omega\) such that \(\mathcal G\)
is also a covering of its boundary by \(L\)-Lipschitz graphs (in the
same sense as above) it also satisfies
\(\mathcal B \colon C^\infty_c (\widetilde\Omega)\to C^\infty_c(\widetilde\Omega;\mathbb R^d)\)
and extends to a bounded operator
\[\mathcal B\colon W^{k,p}_0(\widetilde\Omega)\to W^{k+1,p}_0(\widetilde\Omega;\mathbb R^d)\]
with the norm of this operator depending on $\mathcal G$ and $\Omega$ but independent of \(\widetilde\Omega\).
\end{theorem}

 Restricted to one rectangular subdomain \(G\in \mathcal G\), this
is shown in \cite[Corollary 3.4]{kampschulteUnrestrictedDeformationsThin2023}. While there the 2d-case is studied, the same idea applies verbatim to higher dimensions. In the proof therein, the rectangle is covered by slices
such that two neighboring slices are star-shaped with respect with the
same cube and through a subordinate partition of unity the operator is
extended from one slice to the next. In our case, the slices at the
edges of the neighboring rectangles overlap, and in this way the
construction can be extended from one rectangle to the next. For more
details of the construction see
\cite{kampschulteUnrestrictedDeformationsThin2023}, proof of Theorem
3.3 and Corollary 3.4. and~\cite{BreMenSchSu23} Theorem~2.11.

\begin{theorem}[Bogovskii for close deformations]\label{thm:bogovskii-close-deformations}
 Let
\(\eta\in L^\infty((0,T);\mathcal E)\cap W^{1,2}((0,T);W^{1,2}(Q;\mathbb R^d))\).
Then for \(\delta>0\) small enough the following holds: Fix
\(b\in  C^\infty_c((0,T)\times\Omega_\eta)\) with $\int_{\Omega_\eta}b=1$ and
\(\operatorname{dist}(\operatorname{supp}b,\partial\Omega_\eta)\geq \delta\).
Then exists a universal Bogovskii operator \(\mathcal B\) for domains
\(\Omega_\eta(t)\), that is
\[\mathcal B \colon C^{\infty}_c((0,T)\times\Omega_\eta)\to C^{\infty}_c((0,T)\times\Omega_\eta;\mathbb R^d)\]
which satisfies
\[\operatorname{div} \mathcal B f = f- b\int_{\Omega_\eta} f\] and for
every
\(\tilde\eta\in L^\infty((0,T);\mathcal E)\cap W^{1,2}((0,T);W^{1,2}(Q;\mathbb R^d))\)
with
\begin{equation}\label{eqn:close-defromation-l2}
\|\eta(t)-\tilde\eta(t)\|_{W^{2,q}(Q;\mathbb R^d)}\leq \gamma(\delta)
\end{equation}
the operator \(\mathcal B\) maps
\(\mathcal B \colon C^{\infty}_c((0,T)\times\Omega_{\tilde\eta})\to C^{\infty}_c((0,T)\times\Omega_{\tilde\eta};\mathbb R^d)\)
and is, moreover, extended to a bounded operator for $1\leq r\leq \infty$ and $1<p<\infty$
\[\mathcal B\colon L^r((0,T);W^{k,p}_0(\Omega_{\tilde\eta}))\to L^r((0,T); W^{k+1,p}_0(\Omega_{\tilde\eta};\mathbb R^d))\]
with norm independent of \(\tilde\eta\).
\end{theorem}

\begin{proof}
 By the \(C([0,T];C^{1,\alpha}(Q;\mathbb R^d))\) regularity
of \(\eta\), we can find a partition of the time interval
\(0=t_1<t_2<\dots<t_M=T\), such that on each interval
\((t_{i},t_{i+2})\) the assumptions of Theorem \ref{thm:bogovskii-unviersal} are satisfied.
Namely that there is a covering \(\mathcal G_i\) which covers the
boundary of \(\Omega_\eta(t)\), \(t\in(t_{i-1},t_i)\) by \(L\)-Lipschitz
graphs and for any given \(\tilde \eta\) as in the statement,
\(\mathcal G_i\) also covers the boundary of \(\Omega_{\tilde\eta}(t)\),
\(t\in(t_{i-1},t_i)\) by \(L\)-Lipschitz graphs (because \(W^{2,q}\)
embeds into Lipschitz functions and the closeness \eqref{eqn:close-defromation-l2}) so that Theorem
\ref{thm:bogovskii-unviersal} above gives us the Bogovskii operator \(\mathcal B_i\). Consider a
partition of unity \(\{\psi_i\}\) on \((0,T)\) subordinate to the
covering \((t_i,t_{i+2})\) we define in total the sought Bogovskii
operator by
\[\mathcal B(f)(t) = \sum_{i=1}^{M-2}\psi_i(t)\mathcal B_i(f(t)).\] All
the desired properties of \(\mathcal B\) follow.
\end{proof}

\subsubsection{Approximating fluid-only test functions}\label{approximating-fluid-only-test-functions}

Here, analogously to above, we state an approximation result for the
test functions, this time to be used in the fluid-only equation.
This will go by extending the normals and tangents to the deformed
domains, by means of extending the deformation.

Throughout this section, assume that we have the deformation
\[\eta\in L^\infty((0,T);W^{k_0+2,2}(Q;\mathbb R^d))\cap W^{1,2}((0,T);W^{k_0+2,2}(Q;\mathbb R^d))\text{ with }E(\eta(t))\leq E_0.\]
We now consider for a.e.\ \(t\in(0,T)\) an extension of \(\eta(t)\colon Q\to\mathcal R^d\) to
\(\overline \eta(t)\colon \mathbb R^d\to\mathbb R^d\) such that
\[\|\overline \eta(t)\|_{W^{k_0+2,2}(\mathbb R^d)}\leq C\|\eta\|_{W^{k_0+2,2},(Q)}\]
which can be done using a bounded linear extension operator
\(\overline {\cdot}\colon W^{k_0+2,2}(Q;\mathbb R^d)\to W^{k_0+2,2}(\mathbb R^d;\mathbb R^d)\). Then we have 
\[\begin{aligned}
\|\overline \eta\|_{L^\infty((0,T);W^{k_0+2,2}(\mathbb R^d;\mathbb R^d))} &\leq C\|\eta\|_{L^\infty((0,T);W^{k_0+2,2}(Q;\mathbb R^d))},\\
\|\partial_t \overline \eta\|_{L^2((0,T);W^{k_0+2,2}(\mathbb R^d;\mathbb R^d))} &\leq C\|\partial_t \eta\|_{L^2((0,T);W^{k_0+2,2}(Q;\mathbb R^d))}.
\end{aligned}\]
Further, for some \(\delta>0\) depending only on these norms, we have on \(Q_\delta\) (the
\(\delta\)-neighborhood of \(Q\)) that
\[\det\nabla\overline \eta\geq \varepsilon_0/2\quad \text{in }(0,T)\times Q_\delta,\]
so that \(\overline \eta\) is locally injective on \(Q_\delta\). This means it is also globally injective for small enough \(\delta\), as \(\eta\)
is away from self-contact. Moreover, $\delta$ depends only on $E_0$.

We now treat the normals and tangets to the deformed configuration $\eta(Q)$.
Assume the reference domain \(Q\) has \(C^\infty\) boundary we
have the reference unit normal \(n_Q\in C^\infty(\partial Q;\mathbb R^d)\),
and then the deformed unit normal \(n\colon \partial \eta(Q)\to\mathbb R^d\)
to the deformed configuration \(\eta(Q)\) is
\[n(y)=\frac{[\operatorname{cof}\nabla\eta(\eta^{-1}(y))]n_Q(\eta^{-1}(y))}{|[\operatorname{cof}\nabla\eta(\eta^{-1}(y))]n_Q(\eta^{-1}(y))|}\]
Similarly, for a vector \(v\in \mathbb R^d\), its tangential component is (as $I-n\otimes n$ is the projection to the tangent plane)
\[[v]_\tau=[I-n\otimes n] v = v- (v\cdot n)n.\]

\begin{lemma}[Extension of normal and tangent]\label{lem:ext-normal-tangent}
The unit normal \(n\colon \partial \eta(Q)\to\mathbb R^d\) to the deformed configuration \(\eta(Q)\) admits an extension 
\(\overline n \in W^{k_0,2}(U;\mathbb R^d)\), where $\tilde Q$ is some $\delta$-neighborhood of $\eta(Q)$, with $\delta$ depending on $E_0$. Similarly, the tangent $\tau$ admits an extension of the same regularity.
\end{lemma}
\begin{proof}
We define the extension of the deformed normal by setting
\[\overline n(y)=\frac{[\operatorname{cof}\nabla\overline\eta(\overline\eta^{-1}(y))]\overline n_Q(\overline \eta^{-1}(y))}{|[\operatorname{cof}\nabla\overline\eta(\overline\eta^{-1}(y))]\overline n_Q(\overline \eta^{-1}(y))|}\]
for \(y\) in \(\delta'\)-neighborhood of \(\eta(Q)\), where
\(\delta'\) is given by \(\delta\) and the energy bound \(E_0\).

By Lemma \ref{lem:eulerian-lagrangian-wk02} we see that \(\overline n\) inherits all the
regularity of \(\overline n_Q\) up to \(W^{k_0,2}\). In particular that
\(\overline n \in W^{k_0,2}(\overline\eta(Q_\delta);\mathbb R^d)\).

For the tangent, we consider its extended tangential component of a vector field $\xi$ by
\[[\xi]_{\overline\tau}(y)= [\xi(y)]_{\overline \tau}=[I-\overline n(y)\otimes \overline n(y)] v(y) = v(y)- (v(y)\cdot \overline n(y))\overline n(y).\]
Here it is apparent that \([\xi]_{\overline\tau}\) again inherits the
regularity of \(\overline n\) up to \(W^{k_0,2}\) due to Lemma \ref{lem:eulerian-lagrangian-wk02}.
\end{proof}

\medskip
\noindent\textit{The approximations.}

\begin{proposition}[Approximation of fluid-only test functions]\label{prop:approximation-fluid-only}
Let $\eta^{(i)}\in L^\infty((0,T);\mathcal E)\cap W^{1,2}((0,T);W^{1,2}(Q;\mathbb R^d))$ with $E(\eta^{(i)}(t))\leq E_0$ and $$\begin{aligned}\eta^{(i)}&\overset\ast\rightharpoonup \eta &\quad\text{in }&L^\infty((0,T);W^{2,q}(Q;\mathbb R^d)), \\
\partial_t\eta^{(i)}&\rightharpoonup \partial_t\eta &\quad\text{in }&L^{2}((0,T);W^{1,2}(Q;\mathbb R^d))  \end{aligned},$$
so that also  $\Omega\setminus\eta^{(i)}(t,Q)=\Omega^{(i)}(t)\to\Omega(t) = \Omega\setminus\eta(t,Q)$. Let $\xi\in C^\infty(\Omega(t);\mathbb R^d)$, $\xi\cdot n =0$ on $\partial\Omega(t)$, and $\operatorname{div}\xi=0$ in $\Omega(t)$.

\begin{enumerate}
\item[(i)]
 Then for every $\delta>0$ there is $\xi_\delta^{(i)}$ such that
$$\begin{aligned}
\xi^{(i)}_\delta &\overset\eta\to \xi_\delta &\quad \text{in }&L^2((0,T);W^{1,2}(\Omega;\mathbb R^d)),\\
\partial_t\xi^{(i)}_\delta &\overset\eta\to \partial_t\xi_\delta &\quad \text{in }&L^2((0,T);L^2(\Omega;\mathbb R^d))
\end{aligned}$$
with $\xi_\delta^{(i)}(t)\in W^{1,2}(\Omega^{(i)}(t);\mathbb R^d)$, $\operatorname{div} \xi_\delta^{(i)}=0$ in $\Omega^{(i)}(t)$, $\xi_\delta^{(i)} \cdot \overline n^{(i)} =0$ on a $\delta$-neighborhood of $\partial\Omega(t)$ (which includes $\partial\Omega^{(i)}(t)$ for $i$ large enough) and  $\xi_\delta(t)\in W^{1,2}(\Omega(t);\mathbb R^d)$, $\xi_\delta\cdot \overline n =0$ on a $\delta$-neighborhood of $\partial\Omega(t)$ with
$$\begin{aligned}
\xi_\delta &\overset\eta\to \xi&\quad \text{in }&L^2((0,T);W^{1,2}(\Omega;\mathbb R^d)),\\
\partial_t\xi_\delta &\overset\eta\to \partial_t\xi &\quad \text{in }&L^2((0,T);L^2(\Omega;\mathbb R^d)).
\end{aligned}$$ Moreover, the following estimates hold: $$\begin{aligned}
%\|\xi_\delta^{(i)}(t)\|_{W^{1,2}}&\leq c\|\xi (t)\|_{W^{1,2}}\\
%\|\xi_\delta^{(i)}(t)\|_{W^{k_0,2}}&\leq c(\delta)\|\xi (t)\|_{W^{1,2}}\\
\|\xi_\delta^{(i)}(t) - \xi_\delta(t)\|_{W^{1,2}}&\leq c\delta^{-2q(q-2)}\|\eta (t)-\eta^{(i)}(t)\|_{W^{2,q}} ,\\
\|\xi_\delta(t)-\xi(t)\|_{W^{1,2}}&\leq c\delta^\frac12\|\xi (t)\|_{W^{1,\infty}}\\
\end{aligned}$$
with $c$ depending only on $E_0$.

%\noteMalte[inline]{We seem to be jumping between $\delta$ and $\varepsilon$-neighborhoods in this paper. Can we settle down on one of them?}

\item[(ii)] If additionally $$\eta^{(i)}\rightharpoonup \eta \quad\text{in }W^{1,2}((0,T);W^{k_0+2,2}(Q;\mathbb R^d)),$$ 
then it also holds $$\xi^{(i)}_\delta \overset\eta\to \xi_\delta\quad \text{in }L^2((0,T);W^{k_0,2}(\Omega;\mathbb R^d))$$ 
with the estimates $$\begin{aligned}
%\|\xi_\delta^{(i)}(t)\|_{W^{k_0,2}}&\leq c\|\xi (t)\|_{W^{k_0,2}}\\
\|\xi_\delta^{(i)}(t) - \xi_\delta(t)\|_{W^{k_0,2}}&\leq c\delta^{-k_0}\|\xi (t)\|_{W^{k_0,2}},
\end{aligned}$$
where $c$ depends only on $E_0$ and $\|\eta\|_{W^{k_0,2}}$.

If, moreover, $ \nabla^\ell(\xi\cdot n)=0$ for $\ell=1,\dots,k_0$, then
$$\xi_\delta \overset\eta\to \xi\quad \text{in }L^2((0,T);W^{k_0,2}(\Omega;\mathbb R^d))$$
with the estimate
$$\begin{aligned}
\|\xi_\delta(t)-\xi(t)\|_{W^{k_0,2}}&\leq c\delta\|\xi (t)\|_{W^{k_0+1,2}},
\end{aligned}$$
where $c$ depends only on $E_0$ and $\|\eta\|_{W^{k_0,2}}$.
\end{enumerate}
\end{proposition}
\begin{proof}
 Let now \(\delta>0\) be fixed. Consider a smooth
cutoff \(\psi_\delta\in C^\infty(\mathbb R^d;\mathbb R^d)\) such that
\(\psi_\delta(y)=1\) if
\(\operatorname{dist}(y,\partial \eta(Q))\leq \delta\) and
\(\psi_\delta(y)=0\) if
\(\operatorname{dist}(y,\partial\eta(Q))\geq 2\delta\).

Now we define $\tilde\xi^{(i)}_\delta$ and $\tilde\xi_\delta$ by
\begin{align*}
 \tilde\xi^{(i)}_\delta&= \psi_\delta\left( (\xi\cdot\overline\tau)\overline\tau^{(i)} \right) + (1-\psi_\delta)\xi\quad &\text{in }&\Omega^{(i)}(t),\\
\qquad\tilde\xi_\delta&= \psi_\delta\left( (\xi\cdot\overline\tau)\overline\tau\right) + (1-\psi_\delta)\xi\quad &\text{in }&\Omega(t).
\end{align*}
These satisfy all the conditions except for being divergence free.

Using the universal Bogovskii operator $\mathcal B$ from Theorem \ref{thm:bogovskii-close-deformations} we put
\[\xi_\delta^{(i)}=\tilde\xi_\delta^{(i)} - \mathcal B (\operatorname{div} \tilde\xi_\delta^{(i)}),\]
so that we get
\[\operatorname{div}\xi_\delta^{(i)} = \varphi\int_{\Omega^{(i)}}\operatorname{div}\tilde\xi_\delta^{(i)}\,dx = \varphi \int_{\partial \Omega^{(i)}}\tilde\xi_\delta^{(i)}\cdot dn^{(i)}=0,\]
where the last integral is zero due to
\(\tilde\xi_\delta^{(i)}\cdot n^{(i)}=0\) on
\(\partial\Omega^{(i)}\). Define also
\[\xi_\delta=\tilde\xi_\delta - \mathcal B (\operatorname{div} \tilde\xi_\delta),\]
so that then
\[\xi_\delta-\xi_\delta^{(i)} = \tilde\xi_\delta-\tilde\xi_\delta^{(i)} - \mathcal B(\operatorname{div}\tilde\xi_\delta) +\mathcal B (\operatorname{div}\tilde\xi^{(i)}_\delta). \]
Now \(\xi_\delta\) and \(\xi_\delta^{(i)}\) satisfy both the given
boundary conditions and the divergence free condition, as required. It is left to check the convergences and estimates.

\noindent\textit{(i)}: 
Since it holds $$\tilde \xi_\delta-\tilde\xi_\delta^{(i)} =\psi_\delta (\xi\cdot \overline\tau)(\overline\tau - \overline\tau^{(i)}) ,$$
we compute
$$\begin{aligned}\|\nabla(\tilde\xi_\delta-\tilde\xi_\delta^{(i)})\|_{L^2}= \left\| \nabla(\psi_\delta(\xi\cdot \overline\tau)) (\overline\tau-\overline\tau^{(i)})+\psi_\delta(\xi\cdot \overline\tau) \nabla (\overline\tau-\overline\tau^{(i)})\right\|_{L^2}
\\ \leq \|\nabla(\psi_\delta ( \xi\cdot \overline\tau))\|_{L^{2q/(q-2)}} \|\overline\tau-\overline\tau^{(i)}\|_{L^q}+ \|\psi_\delta ( \xi\cdot \overline\tau)\|_{L^{2q/(q-2)}} \|\nabla(\overline\tau-\overline\tau^{(i)})\|_{L^q}
\\ \leq C\delta^{-2q/(q-2)} \|\eta-\eta^{(i)}\|_{W^{2,q}}.
\end{aligned}$$
Consequently, by the continuity of the Bogovskii operator from Theorem \ref{thm:bogovskii-close-deformations} we have the same estimates for $\xi_\delta$ and $\xi_\delta^{(\tau)}$$$\begin{aligned}\|\nabla(\xi_\delta-\xi_\delta^{(i)})\|_{L^2} \leq C\delta^{-2q/(q-2)} \|\eta-\eta^{(i)}\|_{W^{2,q}},
\end{aligned}$$
which is the desired estimate.

For the limit passage \(\delta \to 0\), we compute
\[\xi-\xi_\delta = \xi -\tilde\xi_\delta+\mathcal B(\operatorname{div}\tilde\xi_\delta) = \psi_\delta(\xi\cdot \overline n)\overline n + \mathcal B(\operatorname{div}\tilde\xi_\delta).\]
Thus
\[\nabla(\tilde \xi -\tilde\xi_\delta)=\nabla(\psi_\delta(\xi\cdot \overline n)\overline n)= \nabla\psi_\delta (\xi\cdot \overline n)  \overline n +\psi_\delta \nabla(\xi\cdot \overline n)  \overline n +\psi_\delta (\xi\cdot \overline n) \nabla \overline n .\]
To estimate this in \(L^2(\Omega(t);\mathbb R^d)\), we note that the cutoff
\(\psi_\delta\) satisfies \[
\|\nabla\psi_\delta\|_{L^\infty}\leq \frac{c}{\delta}, \]
the function \(\xi\cdot \overline n\) has zero trace on
\(\partial\eta(Q)\), so that
\[\|\xi\cdot\overline n\|_{L^\infty}\leq c\delta,\]
and finally \(\|\nabla \overline n\|_{L^q}\leq C\|\eta\|_{W^{2,q}}\). Altogether
with the fact that
\(\left|\operatorname{supp}\psi_\delta\right|\leq c\delta\) we obtain
\[\|\nabla(\tilde \xi -\tilde\xi_\delta)\|_{L^2}= \|\nabla\psi_\delta (\xi\cdot \overline n)  \overline n +\psi_\delta \nabla(\xi\cdot \overline n)  \overline n +\psi_\delta (\xi\cdot \overline n) \nabla \overline n\|_{L^2}\leq C\delta^\frac12 \|\eta\|_{W^{2,q}}\|\xi\|_{W^{1,\infty}}\]
and thus 
$$\|\xi_\delta-\xi_\delta^{(\tau)}\|_{W^{1,2}}\leq c\delta^\frac12.$$
Further, since \(\operatorname{div}\tilde\xi_\delta \to 0\) in
\(L^2((0,T);L^2(\Omega(t)))\) we also have, by the continuity of the operator $\mathcal B$, that
\(\mathcal B(\operatorname{div}\tilde\xi_\delta) \to 0\) in
\(L^2((0,T);L^2(\Omega(t);\mathbb R^d))\) and the same inequality holds for $\xi$ and $\xi_\delta$.

For the convergence of the time derivative with $i\to\infty$, we write 
$$\partial_t(\tilde\xi_\delta^{(i)}-\tilde\xi_\delta) = \partial_t (\psi_\delta(\xi\cdot \overline\tau)(\overline\tau -\overline\tau^{(i)}))=\partial_t\psi_\delta(\xi\cdot \overline\tau)(\overline\tau -\overline\tau^{(i)})+\psi_\delta\partial_t(\xi\cdot \overline\tau)(\overline\tau -\overline\tau^{(i)})+\psi_\delta(\xi\cdot \overline\tau)\partial_t(\overline\tau -\overline\tau^{(i)})$$and use the estimates $$\|\partial_t\psi_\delta\|_{L^\infty} \leq \frac c \delta, \quad \|\overline\tau -\overline \tau ^{(i)}\|_{L^\infty}\leq c \|\eta-\eta^{(i)}\|_{W^{2,q}},\quad \|\partial_t(\overline\tau -\overline\tau^{(i)})\|_{L^2}\leq c \|\partial_t\eta-\partial_t\eta^{(i)}\|_{L^2}$$ along with $|\operatorname{supp}\psi_\delta|\leq c \delta$ to obtain the estimate $$\|\partial_t(\tilde\xi_\delta^{(i)}-\tilde\xi_\delta)\|_{L^2} \leq \frac C \delta \|\eta\|_{W^{2,q}} \|\eta-\eta^{(i)}\|_{W^{2,q}} + \|\partial_t\xi\|_{L^2}\|\eta-\eta^{(i)}\|_{W^{2,q}}+\delta\|\xi\|_{L^\infty}\|\eta\|_{W^{2,q}}\|\partial_t\eta-\partial_t\eta^{(i)}\|.$$
Finally, for the time derivative with $\delta\to 0$ we write $$\partial_t(\tilde\xi_\delta-\tilde\xi) = \partial_t(\psi_\delta(\xi\cdot \overline n) \overline n)=\partial_t \psi_\delta(\xi\cdot \overline n) \overline n+\psi_\delta\partial_t(\xi\cdot \overline n) \overline n+\psi_\delta(\xi\cdot \overline n) \partial_t \overline n,$$ use the estimates (the second one following from $\xi\cdot n=0$ on $\partial\Omega(t)$) $$\|\partial_t\psi_\delta\|_{L^\infty} \leq \frac c \delta, \quad \|\xi\cdot \overline n\|_{L^\infty}\leq c\delta, \quad \|\overline n\|_{L^\infty}\leq c \|\eta\|_{W^{2,q}},\quad \|\partial_t \overline n\|_{L^2}\leq c \|\partial_t\eta\|_{L^2},$$ 
as well as $|\operatorname{supp}\psi_\delta|\leq c \delta$, to obtain the estimate $$\|\partial_t(\tilde\xi_\delta-\tilde\xi)\|_{L^2} \leq C \delta (\|\eta\|_{W^{2,q}} + \|\partial_t\xi\|_{L^2}\|\eta\|_{W^{2,q}} + \|\partial_t\eta\|_{L^2}).$$

\noindent\textit{(ii)}:
For \(i\to\infty\), compute
\[\tilde \xi_\delta-\tilde\xi_\delta^{(i)} =\psi_\delta (\xi\cdot \overline\tau)(\overline\tau - \overline\tau^{(i)}). \]

 Since $\psi_\delta(\xi\cdot \overline\tau)\in C^{k_0}(\Omega;\mathbb R^d)$, we can calculate 
 $$\|\nabla^\ell(\tilde\xi_\delta-\tilde\xi_\delta^{(i)})\|_{L^2}= \left\| \sum_{j=0}^\ell \nabla^{\ell-j} (\psi_\delta(\xi\cdot \overline\tau)) \nabla^j(\overline\tau-\overline\tau^{(i)})\right\|_{L^2}\leq\|\nabla^j \psi_\delta\|_{L^\infty}\|\nabla^{\ell-j} \xi\cdot \overline\tau\|_{C^{k_0}} \|\nabla(\overline\tau-\overline\tau^{(i)})\|_{L^2}$$ $$\leq C\delta^{-\ell}\|\xi\cdot \overline\tau\|_{C^{k_0}}\|\eta-\eta^{(i)}\|_{W^{k_0+1,2}}.$$
 As before, due to the continuity of the Bogovskii operator, the same estimates hold for $\xi_\delta$ and $\xi_\delta^{(\tau)}$.
 
For the limit passage $\delta \to 0$, we compute $$\xi-\xi_\delta = \xi -\tilde\xi_\delta+\mathcal B(\operatorname{div}\tilde\xi_\delta) = \psi_\delta(\xi\cdot \overline n)\overline n + \mathcal B(\operatorname{div}\tilde\xi_\delta)$$
%\malte{
$$\|\nabla^\ell(\psi_\delta(\xi\cdot \overline n)\overline n)\|_{L^2(\Omega(t);\mathbb R^d)} \leq C\left\|\sum_{i=0}^\ell\nabla^{\ell-i}\psi_\delta\sum_{j=0}^i \nabla^{i-j}(\xi\cdot \overline n) \nabla^j \overline n\right\|_{L^2(\Omega(t);\mathbb R^d)}.$$
%}
%\noteMalte[inline]{Without the norm there is all kinds of different cases here depending in which order the derivatives hit the functions, but I think now it is correct.}
To estimate this, the cutoff $\psi_\delta$ is chosen so that $$
\|\nabla^{\ell-i}\psi_\delta\|_{L^\infty}\leq \frac{c}{\delta^{\ell-i}}, $$ the function $\nabla^\ell(\xi\cdot \overline n)$ has zero traces on $\partial\eta(Q)$ up to order $k_0$, so that $$\|\nabla^{i-j}(\xi\cdot\overline n)\|_{L^\infty}\leq c\delta^{k_0-i+j},$$ and finally $\|\nabla^j \overline n\|_{L^\infty}\leq C$. Altogether with the fact that $|\operatorname{supp}\psi_\delta|\leq c\delta$ we obtain
$$\|\nabla^\ell(\psi_\delta(\xi\cdot \overline n)\overline n)\|_{L^2}\leq C\delta.$$
This shows $$\|\tilde\xi(t)-\tilde \xi_\delta(t)\|_{L^2}\leq C\delta \|\eta(t)\|_{W^{k_0,2}}$$
and, again by the continuity of the Bogovskii operator, the same holds for $\xi$ and $\xi_\delta$.

\end{proof}

\section{Strong and weak formulation}\label{sec:strong-and-weak-formulation}

In this section, we formulate the fluid-structure interaction in both the weak and strong formulations. We further verify the consistency of the weak and strong
solutions, namely that any sufficiently regular weak solution is also strong.

\subsection{Weak formulation}\label{sec:weak-formulation}

\begin{definition}[Weak solution]\label{def:weak-solution-full}
The deformation
\(\eta \in L^\infty((0,T),\mathcal E)\cap W^{1,2}((0,T);W^{1,2}(Q;\mathbb R^d))\)
and the velocity field \[
\begin{aligned}
v\in &L^2((0,T);W_{\operatorname{div},\eta}^{1,2}(\Omega(t);\mathbb R^d)) \\ &:= \{ u\in L^2((0,T);W^{1,2}(\Omega(t);\mathbb R^d)): \operatorname{div} u = 0 \text{ in }\Omega(t), u\cdot n = (\partial_t\eta \circ\eta^{-1}) \cdot n \text{ on }\partial \eta(t,Q) \}
\end{aligned}
\] where \(\Omega(t)=\Omega\setminus \eta(t,Q)\), are called a \emph{weak
solution to the fluid-structure interaction problem \eqref{eqn:solid-intro}-\eqref{eqn:init-cond-intro}} if the following two weak equations hold:

\subsubsection*{Fluid-only equation} For all
\(\xi\in C^\infty([0,T] \times \overline \Omega(t);\mathbb R^d)\),
\(\xi\cdot n=0\) on \(\partial \Omega(t)\), \(\operatorname{div}\xi=0\),
with \(\xi(T)=0\) it holds
\begin{equation}\label{eqn:weak-fluid-only}
\begin{aligned}
\int_0^T &-\rho_f\langle v,\partial_t \xi\rangle_{\Omega(t)} - \rho_f \langle v,v\cdot \nabla\xi \rangle_{\Omega(t)}
\, dt + \nu \langle \varepsilon v, \varepsilon \xi\rangle_{\Omega(t)} 
\\
&+ a \langle (v-\partial_t \eta\circ\eta^{-1})\cdot \tau , \xi\cdot \tau\rangle_{\partial\Omega(t)} \,dt = \int_0^T \rho_f\langle f,\xi\rangle_{\Omega(t)}\,dt+\rho_f \langle v_0, \xi(0)\rangle_{\Omega(0)}.
\end{aligned}
\end{equation}
\subsubsection*{Coupled equation}

For all
\(\phi\in L^\infty((0,T); W^{2,q}(Q;\mathbb R^d))\cap W^{1,2}((0,T);W^{1,2}(Q;\mathbb R^d))\),
\(\xi \in C^\infty ([0,T]\times \overline \Omega;\mathbb R^d)\) with
\(\phi = \xi\circ \eta\) in \(Q\),
\(\xi \cdot n = 0\) on
\(\partial \Omega\),
\(\operatorname{div} \xi=0\) in \(\Omega(t)\),
\(\xi(T)=0\),
\(\phi(T)=0\)
 it holds 
 \begin{equation}\label{eqn:weak-coupled}
 \begin{aligned} 
&-\int_0^T \rho_s\langle\partial_t\eta,\partial_t\phi\rangle\,dt+\int_0^T  DE(\eta)\langle\phi\rangle +D_2R(\eta,\partial_t\eta)\langle \phi\rangle dt
\\
&+\int_0^T -\rho_f\langle v,\partial_t \xi\rangle _{\Omega(t)} + \rho_f \langle v,v\cdot \nabla\xi \rangle_{\Omega(t)} + \nu \langle \varepsilon v, \varepsilon \xi\rangle_{\Omega(t)}  \,dt 
\\ 
&\quad=\int_0^T\rho_s\langle f,\phi\rangle dt + \int_0^T \rho_f\langle f,\xi\rangle_{\Omega(t)}\,dt+\rho_s\langle \eta_*,\phi(0)\rangle+\rho_f \langle v_0, \xi(0)\rangle_{\Omega(0)}. 
\end{aligned}
\end{equation}
\end{definition}
Observe that in the {\em coupled equation} above no friction term at the boundary appears, as here the test function is continuous over the interface, which means in particular that their tangential components coincide. 

We shall see now that from this weak formulation, the pressure can be reconstructed.

\begin{proposition}[Pressure reconstruction]\label{prop:pressure-reconstruction}
Let
\(\eta \in L^\infty((0,T);\mathcal E)\cap W^{1,2}((0,T);W^{1,2}(Q;\mathbb R^d))\),
\(v\in L^2((0,T);W_{\operatorname{div},\eta}^{1,2}(\Omega(t);\mathbb R^d))\)
 be a weak solution to the fluid-structure interaction problem as in Definition \ref{def:weak-solution-full}. Then there exists a pressure $p\in\mathcal D'((0,T)\times\Omega)$ with $\operatorname{supp}p\subset (0,T)\times\overline{\Omega(t)}$ such that 
\begin{equation}\label{eqn:weak-formulation-with-pressure}
\begin{aligned}
\int_0^T -\rho_s\langle \partial_t \eta, \partial_t (\xi\circ\eta)\rangle +\langle DE(\eta)+D_2 R(\eta,\partial_t\eta), \xi\circ\eta\rangle-\rho_s\langle f\circ\eta,\xi\circ\eta\rangle \\ 
-\rho_f\langle v, \partial_t\xi \rangle -\rho_f \langle v\otimes v,\nabla \xi\rangle+\nu\langle \varepsilon v,\varepsilon \xi\rangle -\rho_f\langle f,\xi\rangle\,dt-\langle p, \operatorname{div}\xi\rangle =0
\end{aligned}
\end{equation}
holds for all $\xi\in C^\infty_0((0,T)\times\Omega;\mathbb R^d)$. Moreover, $p$ has the regularity 
\begin{equation}\label{eqn:pressure-regularity}
p\in L^\infty((0,T);W^{-1,q'}(\Omega))+W^{-1,\infty}((0,T);W^{1,2}_0(\Omega))+L^a((0,T);L^{b}(\Omega))
\end{equation}
for $1<a<\infty$ and $b=\frac{ad}{ad-2}$.
\end{proposition}
\begin{proof}
 The following weak equation is satisfied \[\begin{aligned} 
-\int_0^T \rho_s\langle\partial_t\eta,\partial_t(\xi \circ \eta)\rangle\,dt+\int_0^T  DE(\eta)\langle\xi \circ \eta\rangle +D_2R(\eta,\partial_t\eta)\langle \xi \circ \eta\rangle dt
\\
+\int_0^T -\rho_f\langle v,\partial_t \xi\rangle _{\Omega(t)} + \rho_f \langle v,v\cdot \nabla\xi \rangle_{\Omega(t)} + \nu \langle \varepsilon v, \varepsilon \xi\rangle_{\Omega(t)}  \,dt 
\\ 
=\int_0^T\rho_s\langle f\circ\eta,\xi \circ \eta\rangle dt + \int_0^T \rho_f\langle f,\xi\rangle_{\Omega(t)}\,dt
\end{aligned}\] for
\(\xi\in C^\infty_{0}((0,T)\times \Omega;\mathbb R^d)\) with
\(\operatorname{div}\xi|_{\Omega(t)}=0\), as defined above.

Define the functional \(\Pi\) by  \[\begin{aligned}
\langle\Pi, \xi\rangle := \int_0^T -\rho_s\langle \partial_t \eta, \partial_t (\xi\circ\eta)\rangle +\langle DE(\eta)+D_2 R(\eta,\partial_t\eta), \xi\circ\eta\rangle-\rho_s\langle f\circ\eta,\xi\circ\eta\rangle \\ 
-\rho_f\langle v, \partial_t\xi \rangle -\rho_f \langle v\otimes v,\nabla \xi\rangle+\nu\langle \varepsilon v,\varepsilon \xi\rangle -\rho_f\langle f,\xi\rangle\,dt,\quad \forall \xi\in C^\infty_0((0,T)\times \Omega;\mathbb R^d).
\end{aligned}\] By the weak equation above, \(\Pi\) vanishes for
\(\xi\in C^\infty_{0}((0,T)\times \Omega;\mathbb R^d)\) with
\(\operatorname{div}\xi|_{\Omega(t)}=0\). In particular \(\Pi\) vanishes
for \(\xi\) with \(\operatorname{supp}\xi\subset \eta(\cdot, Q)\), which
means that \(\operatorname{supp}\Pi\subset \Omega(t)\) (as a support of
distribution \(\Pi\in\mathcal D'((0,T)\times\Omega)^d\)).

By interpolation of $|v|^2\in L^\infty((0,T);L^1(\Omega(t)))\cap L^2((0,T);L^{\frac{d}{d-2}}(\Omega(t)))$ we have a bound on $v\otimes v$ in $L^a((0,T);L^b(\Omega(t);\mathbb R^{d\times d}))$.
So from this we have for
\(\xi\in C^\infty_0([0,T]\times\overline \Omega;\mathbb R^d)\) the
estimate \[\begin{aligned}
|\langle\Pi, \xi\rangle| &\leq \rho_s\| \partial_t \eta\|_{L^\infty((0,T);L^2(Q))}\| \partial_t (\xi\circ\eta)\|_{L^1((0,T);L^2(\eta(Q)))} \\&+\| DE(\eta)\|_{L^\infty((0,T);W^{-2,q'}(Q))}\|\xi\circ\eta \|_{L^1((0,T);W^{2,q}(Q))} \\&+\|D_2 R(\eta,\partial_t\eta)\|_{L^2((0,T);W^{-1,2}(Q))} \|\xi\circ\eta \|_{L^2((0,T);W^{1,2}(\eta(Q)))} \\&-\rho_s\| f\circ\eta\|_{L^2((0,T);L^2(Q;\mathbb R^d))} \|\xi\circ\eta\|_{L^2((0,T);L^2(Q;\mathbb R^d))} \\ 
-\rho_f\|& v\|_{L^\infty((0,T);L^2(\Omega(t))} \| \partial_t\xi \|_{L^1((0,T);L^2(\Omega(t)))} -\rho_f \| v\otimes v\|_{L^a((0,T);L^b(\Omega(t)))}\|\nabla \xi\|_{L^{a'}((0,T);L^{b'}(\Omega(t)))} \\+\nu\| \varepsilon v&\|_{L^2((0,T);L^2(\Omega(t)))}\|\varepsilon \xi\|_{L^2((0,T);L^2(\Omega(t)))} -\rho_f\| f\|_{L^2((0,T);L^2(\Omega(t);\mathbb R^d))} \|\xi\|_{L^2((0,T);L^2(\Omega(t);\mathbb R^d))}\,dt.
\end{aligned}\]

Recall that by Lemma \ref{lem:eulerian-lagrangian-w2q} we have that \(\xi\mapsto\xi\circ\eta\) is an
isomorphism in the spaces above. This shows that indeed \(\Pi\) is
defined (can be extended) for
\[\xi\in X:=L^1((0,T);W^{2,q}_0(\Omega;\mathbb R^d))\cap W^{1,1}_0((0,T);L^2(\Omega;\mathbb R^d))\cap L^{a'}((0,T);W^{1,b'}(\Omega;\mathbb R^d))\]
and is bounded on this space. In other words, \(\Pi\) satisfies
\[\Pi\in X^*=L^\infty((0,T);W^{-2,q'}(\Omega;\mathbb R^d))+W^{-1,\infty}((0,T);L^2(\Omega;\mathbb R^d))+L^a((0,T);W^{-1,b'}(\Omega;\mathbb R^d)).\]

Consider now the divergence operator in \(\Omega\) as a mapping
\[\operatorname{div}\colon X \to Y=L^1((0,T);W^{1,q}(\Omega))\cap W^{1,1}((0,T);W^{-1,2}(\Omega))\cap L^{a'}((0,T); L^{b'}(\Omega)).\]

We construct the fluid-structure Bogovskii operator $${\mathcal B}\colon C^\infty_0((0,T)\times\Omega)\to C^\infty_0((0,T)\times \Omega;\mathbb R^d)$$
as follows. Fix $b\in C^\infty_0((0,T)\times \eta(\cdot,Q))$ with $\int_{\eta(t,Q)} b(t)\,dx=1$, $t\in(0,T)$. Then define $${\mathcal B} (\psi)(t)=\mathcal B_{\Omega}\left(\psi(t)-b(t)\int_\Omega\psi(t)\,dx\right), \quad \psi\in C^\infty_0([0,T]\times \Omega)$$ where $\mathcal B_\Omega$ is the Bogovskii operator for the fixed domain $\Omega$. From this we see that $$\operatorname{div}( {\mathcal B}\psi)= b \int_\Omega\psi \,dx,\quad \psi\in C^\infty_0([0,T]\times \Omega).$$
In fact we can see that it holds that it can be extended as a bounded operator ${\mathcal B}\colon Y \to X$.

Define \(p\in Y^*\) by setting
\[\langle p,\psi\rangle:= \langle \Pi,\mathcal B\psi\rangle, \quad \psi\in Y.\]It holds by the Bogovskii estimates
that \(\mathcal B \colon Y\to X\) is bounded, and we already know that
\(\Pi\in X^*\), verifying that indeed $p\in Y^*$. It remains to show that \(\Pi=\nabla p\). Note that \(\nabla\) is a dual operator to \(\operatorname{div}\). For this, we
compute
\[\langle p, \operatorname{div}\psi\rangle = \langle \Pi, \mathcal B (\operatorname{div}\psi)\rangle= \langle\Pi,\psi\rangle + \langle\Pi, \mathcal B(\operatorname{div}\psi)-\psi\rangle = \langle\Pi,\psi\rangle,\quad \psi\in Y\]
where the last inequality holds since
\(\operatorname{div}(\mathcal B(\operatorname{div}\psi)-\psi)=0\) in
\(\Omega(t)\) and thus it is annihilated by \(\Pi\). This verifies that
\(\Pi=\nabla p\). To conclude the regularity, we have
\[p\in Y^*=L^\infty((0,T);W^{-1,q'}(\Omega))+W^{-1,\infty}((0,T);W^{1,2}_0(\Omega))+L^a((0,T);L^{b}(\Omega)).\]
By the definition of \(p\), we directly verified the weak equation
with pressure, namely 
\[\begin{aligned}
\int_0^T -\rho_s\langle \partial_t \eta, \partial_t (\xi\circ\eta)\rangle +\langle DE(\eta)+D_2 R(\eta,\partial_t\eta), \xi\circ\eta\rangle-\rho_s\langle f\circ\eta,\xi\circ\eta\rangle \\ 
-\rho_f\langle v, \partial_t\xi \rangle -\rho_f \langle v\otimes v,\nabla \xi\rangle+\nu\langle \varepsilon v,\varepsilon \xi\rangle -\rho_f\langle f,\xi\rangle\,dt-\langle p, \operatorname{div}\xi\rangle =0,
\end{aligned}\]
holds for all \(\xi\in C^\infty_0((0,T)\times\Omega;\mathbb R^d)\).
\end{proof}

\subsection{Strong formulation}

Here we state the strong formulation of our problem. For this assume the
solid energy and dissipation have a density, that is they can be written
as
\begin{equation}\label{eqn:ER-density}
E(\eta)=\int_Q e\left(\nabla \eta, \nabla^2 \eta\right) \,d x, \quad R\left(\eta, \partial_t \eta\right)=\int_Q r\left(\nabla \eta, \partial_t \nabla \eta\right) \,d x,
\end{equation}
for some
\(e\in C^3(\mathbb R^{d\times d}_{\det>0}\times \mathbb R^{d^3})\) and 
\(r\in C^2(\mathbb R^{d\times d}_{\det>0}\times \mathbb R^{d\times d})\).
Further, assume the boundary \(\partial Q\) is \(C^{1,1}\), so that the
normal to the boundary \(n\in C^{0,1}(\partial Q;\mathbb R^d)\) has
bounded mean curvature
\(\operatorname{div}_S n\in L^\infty(\partial Q)\).

\begin{definition}[Strong solution]\label{def:strong-solution-full}
We say that
\(\eta\in C^4([0,T]\times \overline Q;\mathbb R^d)\cap L^\infty((0,T);\mathcal E)\),
\(v\in C^2([0,T]\times \overline \Omega(t);\mathbb R^d)\) and
\(p\in C^1([0,T]\times \overline \Omega(t))\) is a \emph{strong solution} if it
satisfies the following equations:
\\ \textit{Bulk solid equation:} 
\begin{equation}\label{eqn:strong-bulk-solid}
\begin{aligned}
  \rho_s \partial_{tt}\eta-\operatorname{div}_x\nabla_\xi e(\nabla\eta,\nabla^2\eta) +\operatorname{div}_x^2 \nabla_w e(\nabla\eta,\nabla^2\eta)&-\operatorname{div}_x \nabla_\xi r(\nabla\eta,\partial_t\nabla\eta)\\ & -\operatorname{div}_x\nabla_z r(\nabla\eta,\partial_t\nabla\eta) &=\rho_s f \quad \text{in }(0,T)\times Q
  \end{aligned}
\end{equation}
\\
  \textit{Bulk fluid equation:} 
\begin{equation}\label{eqn:strong-bulk-fluid}
\begin{aligned}
  \rho_f(\partial_t v +v\cdot \nabla v)&= \nu \Delta v -\nabla p + \rho_f f\quad \text{in }(0,T)\times \Omega(t) \\
  \operatorname{div}v&=0 \quad \text{in } (0,T)\times \Omega(t)
  \end{aligned}
\end{equation}
\\
  \textit{Initial conditions:} 
  \begin{equation}\label{eqn:strong-initial-cond}
\begin{aligned}
  v(0,\cdot)&=v_0 \quad\text{in }\Omega(0)
  \\
  \eta(0)&=\eta_0 
  \\
  \partial_t\eta(0)&=\eta_*
  \end{aligned}
  \end{equation}
\\
  \textit{Kinematic coupling on the interface:}
  \begin{equation}\label{eqn:strong-coupling}
  v\cdot n = (\partial_t\eta\circ\eta^{-1})\cdot n\quad\text{on } (0,T)\times\partial \eta(t,Q) 
  \end{equation}
\\
  \textit{Normal-stress equality the interface:} 
  \begin{equation}\label{eqn:strong-normal-stress}
\begin{aligned}
\nabla_\xi e(\nabla\eta,\nabla^2\eta) - \operatorname{div}_x\nabla_w e(\nabla \eta,\nabla^2\eta) +\nabla_\xi r(\nabla\eta,\partial_t\nabla\eta)+& \\\nabla_z r(\nabla\eta,\partial_t\nabla\eta) -\operatorname{div}_S \nabla_w e(\nabla\eta,\nabla^2\eta) - \operatorname{div}_S n \nabla_w e(\nabla\eta,\nabla^2\eta) n\otimes n]_n & 
   = [[\varepsilon v]n + pn]_n\circ\eta \\ &\quad \text{on }(0,T)\times \partial Q
  \end{aligned}
  \end{equation}
 \textit{ Slipping:} 
  \begin{equation}
  \label{eq:slip} 
  \begin{aligned}
&\nabla_\xi e(\nabla\eta,\nabla^2\eta) - \operatorname{div}_x\nabla_w e(\nabla \eta,\nabla^2\eta) +\nabla_\xi r(\nabla\eta,\partial_t\nabla\eta) \\
&\quad +\nabla_z r(\nabla\eta,\partial_t\nabla\eta) -\operatorname{div}_S \nabla_w e(\nabla\eta,\nabla^2\eta) - (\operatorname{div}_S n) \nabla_w e(\nabla\eta,\nabla^2\eta) n\otimes n]_\tau 
\\
&= -a( \partial_t \eta -v\circ\eta)\cdot \tau |\operatorname{cof}(\nabla \eta)n| \quad \text{on }(0,T)\times \partial Q
  \\
  ([\varepsilon v]n)_\tau  &= -a( v-\partial_t \eta \circ\eta^{-1})\cdot \tau  \quad\text{on } (0,T)\times\partial \Omega(t)
  \end{aligned}
  \end{equation} %\noteMalte[inline]{Please check in particular the signs of $a$}\todo{fix, should be jump}
  
  \textit{Solid hyperstress:}
  \begin{equation}\label{eqn:strong-hyperstress-bc}
   \nabla_w e(\nabla\eta,\nabla^2\eta): n\otimes  n =0 \quad \text{on }(0,T)\times \partial Q
   \end{equation}
\end{definition}

\begin{theorem}[Weak-strong compatibility]\label{thm:weak-strong-compatibility}
Assume that we have a weak solution as in Definition \ref{def:weak-solution-full} with
\[\eta\in C^4([0,T]\times \overline Q;\mathbb R^d)\cap L^\infty((0,T);\mathcal E),\quad v\in C^2([0,T]\times \overline \Omega(t);\mathbb R^d),\quad p\in C^1([0,T]\times \overline \Omega(t)).\] Under
this regularity, it is also a strong solution as in Definition \ref{def:strong-solution-full}.
\end{theorem}
\begin{proof}

%\noteMalte[inline]{I tried to remove duplicate calculations in the following proof (everything was first formal, then in coordintes). I left the removed parts as comments in the file for now. I also tried to include the Navier-Slip term. Please check.}

%Throughout this proof, to aid reading, we include the calculations in
%vector notation, as well as written in (spatial) components.
%Component-wise calculations always use the Einstein summation
%convention, that is any repeating index (labelled \(i,j,k,l\)) is summed
%over \(1,\dots,d\).

\emph{Fluid-only test function.}
By the weak formulation for the fluid-only test function, we have that
for \(\xi\in C^\infty([0,T] \times \overline \Omega(t);\mathbb R^d)\),
\(\xi\cdot n=0\) on \(\partial \Omega(t)\), with \(\xi(T)=0\) it holds by \eqref{eqn:weak-formulation-with-pressure}
\[\begin{aligned}&\int_0^T -\rho_f\langle v,\partial_t \xi\rangle_{\Omega(t)}  + \rho_f \langle v,v\cdot \nabla\xi \rangle_{\Omega(t)}+ \nu \langle \varepsilon v, \varepsilon \xi\rangle_{\Omega(t)} \,dt -\langle p,\operatorname{div}\xi\rangle \\
&= \int_0^T a \langle \partial_t \eta \circ \eta^{-1} -v, \xi\rangle_{\partial \Omega(t)} + \rho_f\langle f,\xi\rangle_{\Omega(t)}\,dt+\rho_f \langle v_0, \xi(0)\rangle_{\Omega(0)}.\end{aligned}\]
First, we focus on the time derivative. By the transport Theorem \ref{thm:transport-theorem}, we compute
\[\frac{d}{dt}\int_{\Omega(t)}v\cdot\xi\,dx = \int_{\Omega(t)}\partial_t v\cdot \xi \,dx+\int_{\Omega(t)}v\cdot\partial_t \xi \,dx +\int_{\partial\Omega(t)} (v\cdot \xi) \tilde n_t \,dS.\]
Integrating this in time (recall \(\xi(T)=0\)) we obtain
\[- \langle v(0),\xi(0)\rangle_{\Omega(0)}= \int_0^T\left( \langle\partial_t v, \xi \rangle_{\Omega(t)}+\langle v,\partial_t \xi \rangle _{\Omega(t)} +\int_{\partial\Omega(t)} (v\cdot \xi)\tilde n_t \,dS\right) dt.\]
For the convective term, integrating by parts gives, thanks to
\(\operatorname{div} v=0\) and
\(v\cdot n=\tilde n_t\) on
\(\partial\Omega(t)\), 
\[
\langle v,v\cdot \nabla\xi\rangle_{\Omega(t)} =  -\langle v\otimes v,\nabla \xi\rangle_{\Omega(t)} = \int_{\partial\Omega(t)}(v\cdot\xi) \underbrace{v\cdot dn}_{=\tilde n_t\,dS}-\int_{\Omega(t)} v\cdot \nabla v\cdot \xi\, dx
% \int_{\Omega(t)}v_i v_j\partial_j\xi_i\,dx &= \int_{\partial\Omega(t)}v_i\xi_i \underbrace{v_j\,dn_j}_{=\tilde n_t\,dS} - \int_{\Omega(t)}\partial_j v_i v_j\xi_i\,dx
\]
For the viscosity term, we can again use \(\operatorname{div}v=0\) \[\begin{aligned}
\langle \varepsilon v,\varepsilon \xi\rangle_{\Omega(t)} &= \int_{\partial\Omega(t)} [\varepsilon v]\xi\cdot dn - \int_{\Omega(t)} \frac12  \Delta v \xi \,dx
\\ 
% \int_{\Omega(t)}\frac14 (\partial_j v_i+\partial_i v_j)(\partial_j\xi_i+\partial_i\xi_j)\,dx&= \int_{\partial \Omega(t)} \frac14(\partial_j v_i +\partial_i v_j)(\xi_i n_j+\xi_j n_i)\,dS \\ &\quad\quad-\int_{\Omega(t)}\frac14 \partial_{jj} v_i\xi_i + \frac14\partial_{ii} v_j\xi_j\,dx
\end{aligned}\]
Finally, for the pressure we have
\[\begin{aligned}-\int_0^T \int_{\Omega(t)}  p\operatorname{div}\xi \,dx\,dt=  \int_0^T -\int_{\partial\Omega(t)} p \xi \cdot  dn  +\int_{\Omega(t)} \nabla p \cdot \xi \,dx\,dt=  \int_0^T \int_{\Omega(t)} \nabla p \cdot \xi \,dx\,dt, \\
% -\int_0^T \int_{\Omega(t)}  p \partial_i \xi_i \,dx\,dt=  \int_0^T - \int_{\partial\Omega(t)} p \xi_i   dn_i  +\int_{\Omega(t)} \partial_i p  \xi_i \,dx\,dt=  \int_0^T \int_{\Omega(t)} \partial_i p  \xi_i \,dx\,dt.
\end{aligned}\]
where the boundary term vanishes since \(\xi\cdot n=0\).
Altogether, we obtain
\[
\begin{aligned}\langle v(0),\xi(0)\rangle_{\Omega(0)}+\int_0^T \langle \partial_t v,\xi\rangle_{\Omega(t)} + \rho_f \int_{\Omega(t)} v\cdot \nabla v\cdot \xi\, dx + \nu \left(\int_{\partial\Omega(t)}[\varepsilon v]\xi\cdot dn - \frac12\int_{\Omega(t)}  \Delta v \xi \,dx\right)\\+\int_{\Omega(t)} \nabla p \cdot \xi \,dx \,dt =  \langle v_0,\xi(0)\rangle _{\Omega(0)}+ \int_0^Ta \langle \partial_t \eta \circ \eta^{-1} -v, \xi\rangle_{\partial \Omega(t)} + \rho_f\langle f,\xi\rangle_{\Omega(t)}\,dt\end{aligned}
\]
for all \(\xi\in C^\infty([0,T]\times \overline\Omega(t);\mathbb R^d)\),
\(\xi\cdot n=0\) on \(\partial\Omega(t)\), with \(\xi(T)=0\). In the
boundary term, we use that
\([\varepsilon v]\xi\cdot n=\xi \cdot [\varepsilon v]n\) and that
\(\xi\) is an arbitrary tangential field. Thus in particular, we see
that the tangential normal stress is the slip term, i.e.
\[([\varepsilon v]n)_\tau  = a (\partial_t  \eta \circ \eta^{-1}-v) \cdot \tau \quad\text{on } (0,T)\times\partial \Omega(t).\]

\noindent\emph{Continuous test function}
We have for 
\(\phi\in L^\infty((0,T); W^{2,q}(Q;\mathbb R^d))\cap W^{1,2}((0,T);W^{1,2}(Q;\mathbb R^d))\),
\(\xi \in C^\infty ([0,T]\times \overline \Omega)\) with
\(\phi = \xi\circ \eta\) in \(Q\), \(\xi \cdot n = 0\) on
\(\partial \Omega\),
\(\xi(T)=0\), \(\phi(T)=0\) that it holds by \eqref{eqn:weak-formulation-with-pressure} 

\begin{equation}\label{eqn:pressure-coupled-eqn-test}
\begin{aligned} 
-\int_0^T \rho_s\langle\partial_t\eta,\partial_t\phi\rangle\,dt+\int_0^T  DE(\eta)\langle\phi\rangle +D_2R(\eta,\partial_t\eta)\langle \phi\rangle dt
\\
+\int_0^T -\rho_f\langle v,\partial_t \xi\rangle _{\Omega(t)} + \rho_f \langle v,v\cdot \nabla\xi \rangle_{\Omega(t)} + \nu \langle \varepsilon v, \varepsilon \xi\rangle_{\Omega(t)}   \,dt  -\langle p,\operatorname{div}\xi\rangle
\\ 
=\int_0^T\rho_s\langle f,\phi\rangle dt + \int_0^T \rho_f\langle f,\xi\rangle_{\Omega(t)}\,dt+\rho_s\langle \eta_*,\phi(0)\rangle+\rho_f \langle v_0, \xi(0)\rangle_{\Omega(0)} 
\end{aligned}
\end{equation}
 So let us take such \(\xi\) and \(\phi\).

For the fluid terms, perform the manipulations as above to obtain \[\begin{aligned}
\int_0^T &-\rho_f\langle v,\partial_t \xi\rangle_{\Omega(t)}  + \rho_f \langle v,v\cdot \nabla\xi \rangle_{\Omega(t)}+ \nu \langle \varepsilon v, \varepsilon \xi\rangle_{\Omega(t)} \,dt -\langle p,\operatorname{div}\xi\rangle  \\ &- \int_0^T \rho_f\langle f,\xi\rangle_{\Omega(t)}\,dt - \rho_f \langle v_0, \xi(0)\rangle_{\Omega(0)}
=\langle v(0),\xi(0)\rangle_{\Omega(0)} \\&+\int_0^T \langle \partial_t v,\xi\rangle_{\Omega(t)} + \rho_f \int_{\Omega(t)} v\cdot \nabla v\cdot \xi\, dx + \nu \left(\int_{\partial\Omega(t)}[\varepsilon v]\xi\cdot dn - \frac12\int_{\Omega(t)}  \Delta v \xi \,dx\right) \,dt \\ &+ \int_{\Omega(t)}\nabla p\cdot \xi \,dx\,dt -\int_{\partial\Omega(t)} p \xi \cdot  dn- \langle v_0,\xi(0)\rangle _{\Omega(0)}- \int_0^T \rho_f\langle f,\xi\rangle_{\Omega(t)}\,dt\end{aligned}.\]

Now consider the solid terms from \eqref{eqn:pressure-coupled-eqn-test} , namely \[\begin{aligned} 
-\int_0^T \rho_s\langle\partial_t\eta,\partial_t\phi\rangle\,dt+\int_0^T  DE(\eta)\langle\phi\rangle +D_2R(\eta,\partial_t\eta)\langle \phi\rangle dt
-\int_0^T\rho_s\langle f,\phi\rangle dt - \rho_s\langle \eta_*,\phi(0)\rangle.
\end{aligned}\]

We calculate (writing \(e=e(\xi,w)\), \(r=r(\xi,z)\)) for a test
function \(\phi\in C^\infty([0,T]\times \overline Q;\mathbb R^d)\) with
\(\phi(T)=0\), from \eqref{eqn:ER-density} 
\[\begin{aligned}
DE(\eta)\langle \phi\rangle &=\int_Q \nabla_\xi e(\nabla\eta,\nabla^2\eta): \nabla \phi + \nabla_w e(\nabla\eta,\nabla^2\eta):\nabla^2\phi\,dx, \\
D_2R(\eta,\partial_t\eta)\langle \phi\rangle &=\int_Q \nabla_\xi r(\nabla\eta,\partial_t\nabla\eta):\nabla\phi+\nabla_z r(\nabla\eta,\partial_t\nabla\eta):\nabla\phi\,dx.
\end{aligned}\]
% in components: \[\begin{aligned}
% DE(\eta)\langle \phi\rangle &=\int_Q \partial_{\xi_{ij}} e(\nabla\eta,\nabla^2\eta)\partial_j \phi_i + \partial_{w_{ijk}} e(\nabla\eta,\nabla^2\eta)\partial_{jk}\phi_i\,dx \\
% D_2R(\eta,\partial_t\eta)\langle \phi\rangle &=\int_Q \partial_{\xi_{ij}} r(\nabla\eta,\partial_t\nabla\eta)\partial_j\phi_i+\partial_{z_{ij}} r(\nabla\phi,\partial_t\nabla\eta)\partial_j\phi_i\,dx
% \end{aligned}\]

We integrate by parts and obtain \[\begin{aligned}
DE(\eta)\langle\phi\rangle&=\int_{\partial Q}\nabla_\xi e(\nabla\eta,\nabla^2\eta)\cdot \phi\cdot dn - \int_Q \operatorname{div}_x\nabla_\xi e(\nabla\eta,\nabla^2\eta)\cdot\phi\,dx \\ &+ \int_{\partial Q} \nabla_w e(\nabla\eta,\nabla^2\eta):\nabla \phi\cdot dn-\int_Q \operatorname{div}_x\nabla_w e(\nabla\eta,\nabla^2\eta):\nabla\phi\,dx.
\\
% DE(\eta)\langle\phi\rangle&=\int_{\partial Q}\partial_{\xi_{ij}} e(\nabla\eta,\nabla^2\eta) \phi_i dn_j - \int_Q \partial_j\partial_{\xi_{ij}} e(\nabla\eta,\nabla^2\eta)\phi_i\,dx \\ &+ \int_{\partial Q} \partial_{w_{ijk}} e(\nabla\eta,\nabla^2\eta)\partial_j \phi_i  dn_k-\int_Q \partial_k\nabla_{w_{ijk}} e(\nabla\eta,\nabla^2\eta)\partial_j\phi_i\,dx
\end{aligned}\] 
Integration by parts in the last term yields
\[\begin{aligned}
-\int_Q \operatorname{div}_x\nabla_w e(\nabla\eta,\nabla^2\eta):\nabla\phi\,dx&= -\int_{\partial Q} \!\! \operatorname{div}_x\nabla_w e(\nabla \eta,\nabla^2\eta)\cdot \phi\cdot dn + \int_Q \operatorname{div}_x^2 \nabla_w e(\nabla\eta,\nabla^2\eta)\cdot \phi\,dx. \\
% -\int_Q \partial_k\nabla_{w_{ijk}} e(\nabla\eta,\nabla^2\eta)\partial_j\phi_i\,dx&=-\int_{\partial Q} \partial_k \partial_{w_{ijk}} e(\nabla \eta,\nabla^2\eta)\phi_i\, dn_j + \int_Q \partial_j\partial_k \partial_{w_{ijk}} e(\nabla\eta,\nabla^2\eta)\phi_i \,dx
\end{aligned}\]
Now, the second-to-last term (namely
\(\int_{\partial Q} \nabla_w e(\nabla\eta,\nabla^2\eta):\nabla \phi\cdot dn\))
can be further rewritten as follows. We
denote \[\begin{aligned}
\nabla &= (n\otimes n + (I-n\otimes n))\nabla=: n\otimes\partial_n  +\nabla_S. \\
% \partial_i &= n_in_j\partial_j + (\partial_i-n_in_j\partial_j)
\end{aligned}\] 
In other words, for \(u\colon \mathbb R^d\to \mathbb R\) we have
\[\begin{aligned}
\nabla u &= (n\cdot \nabla u) n + (\nabla u - (n\cdot \nabla u)n) = \partial_n u n +\nabla_S u, \\
% \partial_i u &=  n_j\partial_ju n_i + (\partial_i u-n_j\partial_ju n_i)
\end{aligned}\]
and for \(\eta\colon \mathbb R^d\to \mathbb R^d\)
\[\begin{aligned}
\nabla \eta &= (n\cdot \nabla \eta) n + (\nabla \eta - (n\cdot \nabla \eta)n) = \partial_n \eta \otimes n +\nabla_S \eta, \\
% \partial_j \eta_i &=  n_k\partial_k\eta_i n_j + (\partial_j \eta_i-n_k\partial_k \eta_i n_j) \\
\operatorname{div}_S \eta &= \operatorname{Tr} \nabla_S\eta .%= \partial_i\eta_i-n_k\partial_k\eta_i n_i
\end{aligned}\]
Rewriting the $\nabla\phi$ in this form,
we obtain \[\begin{aligned}
 \int_{\partial Q} \nabla_w e(\nabla\eta,\nabla^2\eta):\nabla \phi\cdot dn &= \int_{\partial Q} \nabla_w e(\nabla\eta,\nabla^2\eta): n\otimes \partial_n \phi \cdot  dn +\int_{\partial Q} \nabla_w e(\nabla\eta,\nabla^2\eta):\nabla_S \phi\cdot dn  \\
%  \int_{\partial Q} \partial_{w_{ijk}} e(\nabla\eta,\nabla^2\eta)\partial_j \phi_i  dn_k &= \int_{\partial Q} \partial_{w_{ijk}} e(\nabla\eta,\nabla^2\eta) n_l \partial_l \phi_i n_j  dn_k  \\ &\quad +\int_{\partial Q} \partial_{w_{ijk}} e(\nabla\eta,\nabla^2\eta) (\partial_j\phi_i- n_l \partial_l \phi_i n_j)  dn_k
\end{aligned}\] Using the Gauss-Green formula on the surface
\(\partial Q\) (which has empty relative boundary)
gives us
\[\begin{aligned}\int_{\partial Q} \nabla_w e(\nabla\eta,\nabla^2\eta):\nabla_S \phi\cdot dn = &-\int_{\partial Q} \operatorname{div}_S \nabla_w e(\nabla\eta,\nabla^2\eta)\cdot  \phi\cdot dn \\&-\int_{\partial Q} (\operatorname{div}_S n) \nabla_w e(\nabla\eta,\nabla^2\eta): n \otimes  \phi\cdot dn. \\
% \int_{\partial Q} \partial_{w_{ijk}} e(\nabla\eta,\nabla^2\eta) (\partial_j\phi_i- n_l \partial_l \phi_i n_j)  dn_k =&-\int_{\partial Q} \!\! ( \partial_i \nabla_{w_{ijk}} e(\nabla\eta,\nabla^2\eta)-n_l\partial_l\nabla_{w_{ijk}} e(\nabla\eta,\nabla^2\eta) n_i )  \phi_j\, dn_k \\&-\int_{\partial Q} (\partial_i n_i-n_k\partial_k n_i n_i) \nabla_{w_{ijk}} e(\nabla\eta,\nabla^2\eta)) n_i   \phi_k\, dn_j
\end{aligned}\]
Thus in total we get for the energy term 
\[\begin{aligned}
DE(\eta)\langle\phi\rangle &=\int_{\partial Q}\hspace{-1em}\nabla_\xi e(\nabla\eta,\nabla^2\eta)\cdot \phi\cdot dn - \int_Q\, \operatorname{div}_x\nabla_\xi e(\nabla\eta,\nabla^2\eta)\cdot\phi\,dx + \int_{\partial Q} \hspace{-1em} \nabla_w e(\nabla\eta,\nabla^2\eta): n\otimes \partial_n \phi \cdot  dn \\
&-\int_{\partial Q} \operatorname{div}_S \nabla_w e(\nabla\eta,\nabla^2\eta)\cdot  \phi\cdot dn -\int_{\partial Q} (\operatorname{div}_S n) \nabla_w e(\nabla\eta,\nabla^2\eta): n\otimes n \cdot  \phi\cdot dn \\
&-\int_{\partial Q} \operatorname{div}_x\nabla_w e(\nabla \eta,\nabla^2\eta)\cdot \phi\cdot dn + \int_Q \operatorname{div}_x^2 \nabla_w e(\nabla\eta,\nabla^2\eta)\cdot \phi\,dx.
%\\
%DE(\eta)\langle\phi\rangle&=\int_{\partial Q}\partial_{\xi_{ij}} e(\nabla\eta,\nabla^2\eta) \phi_i dn_j - \int_Q \partial_j\partial_{\xi_{ij}} e(\nabla\eta,\nabla^2\eta)\phi_i\,dx + \int_{\partial Q} \partial_{w_{ijk}} e(\nabla\eta,\nabla^2\eta)\partial_j \phi_i  dn_k\\
%&-\int_{\partial Q} \partial_k \partial_{w_{ijk}} e(\nabla \eta,\nabla^2\eta)\phi_i\, dn_j + \int_Q \partial_j\partial_k \partial_{w_{ijk}} e(\nabla\eta,\nabla^2\eta)\phi_i \,dx
\end{aligned}\] 
For the dissipation, integrating by parts in space we get
\[\begin{aligned}
D_2R(\eta,\partial_t\eta)\langle\phi\rangle&= \int_{\partial Q}\nabla_\xi r(\nabla\eta,\partial_t\nabla\eta)\cdot\phi\cdot dn-\int_Q \operatorname{div}_x \nabla_\xi r(\nabla\eta,\partial_t\nabla\eta)\cdot\phi\,dx \\ 
&+ \int_{\partial Q} \nabla_z r(\nabla\eta,\partial_t\nabla\eta)\cdot \phi \cdot dn-\int_Q \operatorname{div}_x\nabla_z r(\nabla\eta,\partial_t\nabla\eta)\cdot \phi \,dx.
\\
% D_2R(\eta,\partial_t\eta)\langle\phi\rangle&=\int_{\partial Q}\partial_{\xi_{ij}} r(\nabla\eta,\partial_t\nabla\eta)\phi_i\, dn_j-\int_Q \partial_j \partial_{\xi_{ij}} r(\nabla\eta,\partial_t\nabla\eta)\phi_i\,dx \\
% &+ \int_{\partial Q} \partial_{z_{ij}} r(\nabla\eta,\partial_t\nabla\eta)\phi_i \, dn_j-\int_Q \partial_j\partial_{z_{ij}} r(\nabla\eta,\partial_t\nabla\eta) \phi_i \,dx
\end{aligned}\]
For the inertial term, integrate by parts in time to obtain (recall
that \(\phi(T)=0\)) 
\[
-\int_0^T \rho_s\langle\partial_t\eta,\partial_t\phi\rangle\,dt = \rho_s\langle \partial_t\eta(0),\phi(0)\rangle+ \int_0^T \rho_s\langle\partial_{tt}\eta ,\phi\rangle\,dt.
\]
Therefore, in total we get that 
\[\begin{aligned}
&-\int_0^T \rho_s\langle\partial_t\eta,\partial_t\phi\rangle\,dt+\int_0^T  DE(\eta)\langle\phi\rangle +D_2R(\eta,\partial_t\eta)\langle \phi\rangle dt
-\int_0^T\rho_s\langle f,\phi\rangle dt - \rho_s\langle \eta_*,\phi(0)\rangle=
\\
&\int_0^T \rho_s\langle\partial_{tt}\eta,\phi\rangle+ 
\int_{\partial Q}\nabla_\xi e(\nabla\eta,\nabla^2\eta)\cdot \phi\cdot dn \\&- \int_Q \operatorname{div}_x\nabla_\xi e(\nabla\eta,\nabla^2\eta)\cdot\phi\,dx +\int_{\partial Q} \nabla_w e(\nabla\eta,\nabla^2\eta): n\otimes \partial_n \phi \cdot  dn \\
&-\int_{\partial Q} \operatorname{div}_S \nabla_w e(\nabla\eta,\nabla^2\eta)\cdot  \phi\cdot dn -\int_{\partial Q} (\operatorname{div}_S n) \nabla_w e(\nabla\eta,\nabla^2\eta): n\otimes n \cdot  \phi\cdot dn  \\&
-\int_{\partial Q} \operatorname{div}_x\nabla_w e(\nabla \eta,\nabla^2\eta)\cdot \phi\cdot dn  + \int_Q \operatorname{div}_x^2 \nabla_w e(\nabla\eta,\nabla^2\eta)\cdot \phi\,dx \int_{\partial Q}\nabla_\xi r(\nabla\eta,\partial_t\nabla\eta)\cdot\phi\cdot dn \\&
-\int_Q \operatorname{div}_x \nabla_\xi r(\nabla\eta,\partial_t\nabla\eta)\cdot\phi\,dx \, dt + \int_{\partial Q} \nabla_z r(\nabla\eta,\partial_t\nabla\eta)\cdot \phi \cdot dn-\int_Q \operatorname{div}_x\nabla_z r(\nabla\eta,\partial_t\nabla\eta)\cdot \phi \,dx\\&
-\int_0^T\rho_s\langle f,\xi\rangle dt+\rho_s\langle \partial_t\eta(0),\phi(0)\rangle -\rho_s\langle\partial_t\eta(0),\phi(0)\rangle.
\end{aligned}\]

Altogether, with the solid and fluid terms, it holds
\[\begin{aligned}
\langle v(0),\xi(0)\rangle_{\Omega(0)}+\int_0^T \langle \partial_t v,\xi\rangle_{\Omega(t)} + \rho_f \int_{\Omega(t)} v\cdot \nabla v\cdot \xi\, dx + \nu \left(\int_{\partial\Omega(t)}[\varepsilon v]\xi\cdot dn - \frac12\int_{\Omega(t)}  \Delta v \xi \,dx\right) \,dt 
\\
+ \int_{\Omega(t)}\nabla p\cdot \xi \,dx\,dt-\int_{\partial\Omega(t)} p \xi \cdot  dn - \langle v_0,\xi(0)\rangle _{\Omega(0)}- \int_0^T \rho_f\langle f,\xi\rangle_{\Omega(t)}\,dt +
\int_0^T \rho_s\langle\partial_{tt}\eta,\phi\rangle+ 
\\
\int_{\partial Q}\nabla_\xi e(\nabla\eta,\nabla^2\eta)\cdot \phi\cdot dn - \int_Q \operatorname{div}_x\nabla_\xi e(\nabla\eta,\nabla^2\eta)\cdot\phi\,dx +\int_{\partial Q} \nabla_w e(\nabla\eta,\nabla^2\eta): n\otimes \partial_n \phi \cdot  dn 
\\
-\int_{\partial Q} \operatorname{div}_S \nabla_w e(\nabla\eta,\nabla^2\eta)\cdot  \phi\cdot dn -\int_{\partial Q} (\operatorname{div}_S) n \nabla_w e(\nabla\eta,\nabla^2\eta): n\otimes   \phi\cdot dn  
\\
-\int_{\partial Q} \operatorname{div}_x\nabla_w e(\nabla \eta,\nabla^2\eta)\cdot \phi\cdot dn  + \int_Q \operatorname{div}_x^2 \nabla_w e(\nabla\eta,\nabla^2\eta)\cdot \phi\,dx \int_{\partial Q}\nabla_\xi r(\nabla\eta,\partial_t\nabla\eta)\cdot\phi\cdot dn 
\\
-\int_Q \operatorname{div}_x \nabla_\xi r(\nabla\eta,\partial_t\nabla\eta)\cdot\phi\,dx \, dt + \int_{\partial Q} \nabla_z r(\nabla\eta,\partial_t\nabla\eta)\cdot \phi \cdot dn-\int_Q \operatorname{div}_x\nabla_z r(\nabla\eta,\partial_t\nabla\eta)\cdot \phi \,dx
\\
-\int_0^T\rho_s\langle f,\xi\rangle dt+\rho_s\langle \partial_t\eta(0),\phi(0)\rangle -\rho_s\langle\partial_t\eta(0),\phi(0)\rangle=0
\end{aligned}\] 
for all
\(\phi\in L^\infty((0,T); W^{2,q}(Q;\mathbb R^d))\cap W^{1,2}((0,T);W^{1,2}(Q;\mathbb R^d))\),
\(\xi \in C^\infty ([0,T]\times \overline \Omega;\mathbb R^d)\) with
\(\phi = \xi\circ \eta\) in \(Q\), \(\xi \cdot n = 0\) on
\(\partial \Omega\),
\(\xi(T)=0\), \(\phi(T)=0\).

From this, we can see that the strong formulation as in Definition \ref{def:strong-solution-full} holds. The only non-trivial part is the term of the slip-friction of the solid. For this we note that by the above tells us that it needs to be equal and opposite to the respective stress of the fluid, which we already derived using the fluid-only equation.
\end{proof}

\section{Variational existence scheme}\label{sec:variational-existence-scheme}

Let us fix the regularization parameter $\kappa>0$. We use the regularized energy and dissipation, defined as 
\[E_\kappa(\eta) := E(\eta)+\kappa^{a_0}\|\nabla^{k_0+2}\eta\|^2,\quad R_\kappa(\eta,b):=R(\eta,b) + \kappa\|\nabla^{k_0+2}b\|^2\]
where the exponent \(a_0>0\) will be chosen later, and $k_0$ is chosen so large that $W^{k_0,2}(Q)$ embeds into $W^{2,q}(Q)$.

Now we fix \(h>0\) and solve first the time delayed problem on
\((0,h)\).

\begin{definition}[Time-delayed problem]\label{def:time-delayed-solution}
Let \(w_f\in L^2((0,h)\times \Omega_0;\mathbb R^d)\),
\(w_s\in L^2((0,h)\times Q;\mathbb R^d)\) be given, where $\Omega_0=\Omega\setminus \eta_0(\overline Q)$. Then
\(\eta,v\) is called a \emph{weak solution to the time-delayed problem} if
\((\eta\circ\eta^{-1})\cdot n= v\cdot n\) on \(\partial\eta(\cdot,Q)\),
and satisfies the following equations.

\subsubsection*{Fluid-only equation}
\[\nu \langle\varepsilon v, \varepsilon \xi\rangle_{\Omega(t)} + \kappa\langle\nabla^{k_0} v, \nabla^{k_0}\xi\rangle_{\Omega(t)}+\rho_f \left\langle \tfrac{v \circ \Phi-w_{f}}{h}, \xi \circ\Phi \right\rangle_{\Omega_{0}} - a \langle \partial_t \eta \circ \eta^{-1}-v,\xi\rangle_{\partial \Omega(t)} -\rho_f\langle f,\xi\rangle_{\Omega(t)}=0\]
for all test functions
\(\xi\in C^\infty((0,h)\times \overline\Omega(t);\mathbb R^d)\) with
\(\operatorname{div}\xi =0\) in \(\Omega(t)\), \(\xi\cdot n = 0\) on
\(\partial\Omega(t)\).

\subsubsection*{Coupled equation} \[\begin{aligned} 
DE(\eta)\langle\phi\rangle+D_2 R\left(\eta, \partial_t\eta\right) \langle \phi \rangle  + 2\kappa^{a_0}\langle \nabla^{k_0+2}\eta, \nabla^{k_0+2}\phi \rangle + 2\kappa\left\langle \nabla^{k_0+2}\partial_t\eta , \nabla^{k_0+2}\phi \right\rangle
\\
+\rho_s\left\langle\tfrac{\partial_t\eta-w_s}{h}, \phi\right\rangle- \rho_s\left\langle f\circ\eta,\phi \right\rangle
\\
+\nu \langle\varepsilon v, \varepsilon \xi\rangle_{\Omega(t)} + \kappa\langle\nabla^{k_0} v, \nabla^{k_0}\xi\rangle_{\Omega(t)}+\rho_f \left\langle \tfrac{v \circ \Phi-w_{f}}{h}, \xi \circ\Phi \right\rangle_{\Omega_{0}} - \rho_f\langle f,\xi\rangle_{\Omega}=0
\end{aligned}\] for all test functions
\(\xi\in C^\infty((0,h)\times \Omega;\mathbb R^d)\),
\(\phi\in C([0,h];W^{k_0+2,2}(Q;\mathbb R^d))\) with
\(\phi=\xi\circ \eta\) and \(\operatorname{div}\xi =0\) in
\(\Omega(t)\), \(\nabla^\ell(\xi\cdot n) = 0\) on \(\partial\Omega\), $\ell\in \{0,\dots,k_0\}$.

The flow map \(\Phi\colon (0,T)\times\Omega_0\to\mathbb R^d\) is such that
\(\Phi_t:=\Phi(t,\cdot)\) solves \(\partial_t \Phi_t = v(t)\circ\Phi_t\)
with \(\Phi_0=\operatorname{id}_{\Omega_0}\), here we denote
\(\Omega(t)=\Omega\setminus \eta(t,Q)\).
\end{definition}

The velocities \(w_f\) resp.\ \(w_s\) will later represent the fluid
resp.\ solid velocity in the previous \(h\) step. To solve this problem, we perform the following minimization scheme.

\subsection{Minimization}\label{sec:minimization}

For now, let \(\tau>0\) be a fixed discretization step. To avoid technicalities, assume that $T$ is an integer multiple of $\tau$. Throughout this section, we use the time-step index $k=1,\dots, T/\tau$

To ease the notation, we write for the fluid terms
\(\|\cdot\|_{\Omega_k}\) for the \(L^2(\Omega_k;\mathbb R^d)\) norm, and
in the solid terms \(\|\cdot\|\) for the \(L^2(Q;\mathbb R^d)\) norm,
similarly for scalar products on these spaces.

We discretize the right hand side $f$ as
\[f_k^{(\tau)}=\int_{(k-1)\tau}^{k\tau}f(t)\,dt\] and the velocities
\[\begin{aligned}
w_{s,k}^{(\tau)} = \frac 1\tau \int_{(k-1)\tau}^{k\tau} w_s\,dt, \quad w_{f,k}^{(\tau)} = \frac 1\tau \int_{(k-1)\tau}^{k\tau} w_f\,dt. 
\end{aligned}\] Assume we have \(\eta_{k-1}^{(\tau)}\),
\(v_{k-1}^{(\tau)}\), where
\(\eta_0^{(\tau)}\in \mathcal E \cap W^{k_0+2,2}(Q;\mathbb R^d)\) and
\(v_0^{(\tau)}\in W^{k_0,2}(\Omega_0;\mathbb R^d)\) is the given initial
condition and the domain is
\(\Omega_{k-1}^{(\tau)}=\Omega\setminus\eta_{k-1}^{(\tau)}(\overline Q)\).

The next step is found as
\begin{equation}\label{eqn:minimization-Jktau}
(\eta_k^{(\tau)},v_k^{(\tau)})\in \operatorname{arg\,min}  \mathcal J_k^{(\tau)}(\eta,v)
\end{equation}
where the minimum is over
\((\eta,v)\in\mathcal E \cap W^{k_0+2,2}(Q;\mathbb R^d)\times W^{k_0,2}_{\operatorname{div}}(\Omega_{k-1}^{(\tau)};\mathbb R^d)\)
satisfying
\[\text{ with } \\ \tfrac{\eta-\eta_{k-1}^{(\tau)}}{\tau}\cdot n_{k-1}^{(\tau)} = v \circ\eta \cdot n_{k-1}^{(\tau)} ,\text{ on }\partial Q, \text{ and } v\cdot n_{k-1}^{(\tau)} = 0 \text{ on } \partial\Omega,  
\]
where  the functional
\(\mathcal J_k^{(\tau)}\colon W^{k_0+2,2}(Q;\mathbb R^d)\times W^{k_0,2}_{\operatorname{div}}(\Omega_{k-1}^{(\tau)};\mathbb R^d) \to \mathbb R\)
is defined as \[\begin{aligned}
\mathcal J_k^{(\tau)}(\eta,v):=E(\eta)+\tau R\left(\eta_{k-1}^{(\tau)}, \tfrac{\eta-\eta_{k-1}^{(\tau)}}{\tau}\right) + \kappa^{a_0}\|\nabla^{k_0+2}\eta\|^2 + \kappa\tau\left\| \nabla^{k_0+2}\tfrac{\eta-\eta_{k-1}^{(\tau)}}{\tau} \right\|^2
\\
 +\frac{a\tau}{2}
 \int_{\partial\Omega_{k-1}}\left| \tfrac{\eta-\eta_{k-1}^{(\tau)}\circ\eta_{k-1}^{-1}}{\tau}-v \right |^2\, dx
 +\rho_s\frac{\tau h}{2}\left\|\tfrac{\tfrac{\eta-\eta_{k-1}^{(\tau)}}{\tau}-w_{s,k-1}^{(\tau)} }{h}\right\|^2-\tau \rho_s \left\langle f_k^{(\tau)}\circ\eta_{k-1}^{(\tau)},\tfrac{\eta-\eta_{k-1}^{(\tau)}}{\tau} \right\rangle
\\
+\frac{\tau\nu}{2} \|\varepsilon v\|_{\Omega_{k-1}^{(\tau)}}^2 + \kappa\frac{\tau}{2}\|\nabla^{k_0} v\|_{\Omega_{k-1}^{(\tau)}}^2+\rho_f\frac{\tau h}{2} \left\| \tfrac{v\circ \Phi_{k-1}^{(\tau)} -w_{f,k-1}^{(\tau)}}{h} \right\|_{\Omega_0}^2 - \tau \rho_f\langle f_k^{(\tau)},v\rangle_{\Omega_{k-1}^{(\tau)}}
.\end{aligned}\]
The flow map begins as \(\Phi_0^{(\tau)}:=\operatorname{id}\), and for
the next step is defined as
\begin{equation}\label{eqn:Phiktau-def}
\Phi_k^{(\tau)}:=(\operatorname{id}+\tau v_k^{(\tau)})\circ\Phi_{k-1}^{(\tau)}.
\end{equation}
We will assume that it holds
\(\Omega_k^{(\tau)}:=\Phi_k^{(\tau)}(\Omega_0),\) and moreover
\(\Phi_k^{(\tau)}\colon \Omega_0\to\Omega_k^{(\tau)}\) is a
diffeomorphism with bounded Jacobian. This is certainly true for \(k=0\),
and for subsequent steps it will be shown below in Proposition \ref{prop:estimates-of-the-flow-map}.

\begin{lemma}[Existence of a minimum]
The minimization problem \eqref{eqn:minimization-Jktau} admits a minimum $(\eta_k^{(\tau)},v_k^{(\tau)})$.
\end{lemma}

\begin{proof}

We show that the minimum exists by the direct method of the calculus of
variations. Firstly, the set over which we minimize is nonempty, because
\((\eta_{k-1}^{(\tau)},0)\in \mathcal E\cap W^{k_0+2,2}(Q;\mathbb R^d)\times W^{k_0,2}_{\operatorname{div}}(\Omega_{k-1}^{(\tau)};\mathbb R^d)\)
satisfies the coupling condition and is thus admissible. To see that the
functional \(\mathcal J_k^{(\tau)}\) is bounded from below, use the
inequality \((a-b)^2\geq a^2/2-b^2\) in the terms
\[
\begin{aligned}
\rho_f\frac{\tau h}{2} \left\| \tfrac{v\circ \Phi_{k-1}^{(\tau)} -w_{f,k-1}^{(\tau)}}{h} \right\|_{\Omega_{0}}^2 &\geq \rho_f \frac{\tau}{4h}\|v\circ\Phi_{k-1}^{(\tau)}\|_{\Omega_{0}}^2 - \rho_f\frac{\tau}{2h}\|w_{f,k-1}^{(\tau)}\|^2,
\\
\rho_s\frac{\tau h}{2}\left\|\tfrac{\tfrac{\eta-\eta_{k-1}}{\tau}-w_{s,k-1}^{(\tau)} }{h}\right\|^2&\geq \rho_s \frac{\tau}{4h}\left\| \tfrac{\eta-\eta_{k-1}^{(\tau)}}{\tau}\right\|^2-\rho_s\frac{\tau}{2h}\|w_{s,k-1}^{(\tau)}\|^2,
\end{aligned}
\]
and the Young inequality in the \(f\) terms, namely
\[
\begin{aligned}
-\tau \rho_f\langle f,v\rangle _{\Omega_{k-1}}
&\geq -\rho_f\tau\|f\|_{\Omega_{k-1}^{(\tau)}}^2 -\rho_f \frac{\tau}{4}\|v\|_{\Omega_{k-1}^{(\tau)}}^2
\\
-\tau \rho_s \left\langle f\circ\eta,\tfrac{\eta-\eta_{k-1}^{(\tau)}}{\tau} \right\rangle 
&\geq - \rho_s\tau \|f\circ\eta\|^2 - \rho_s\frac{\tau}{4}\left\|\tfrac{\eta-\eta_{k-1}^{(\tau)}}{\tau}\right\|^2.
\end{aligned}
\]
Omitting the other non-negative terms, we get
\[\mathcal J_k^{(\tau)}(\eta,v) \geq  E_{\min}  - \rho_f\frac{\tau}{2h}\|w_{f,k-1}^{(\tau)}\|^2 -\rho_s\frac{\tau}{2h}\|w_{s,k-1}^{(\tau)}\|^2 -\rho_f\tau\|f\|_{\Omega_{k-1}^{(\tau)}}^2 - \rho_s\tau \|f\circ\eta\|^2.  \]
Since the mapping \(f\mapsto f\circ\eta\) is bounded in \(L^2\) by
Lemma \ref{lem:eulerian-lagrangian-w2q} (since \(E(\eta)\) is bounded) this gives a uniform bound
and thus \(\mathcal J\) is bounded from below.

Further \(\mathcal J_k^{(\tau)}\) is coercive on
\(W^{k_0+2,2}(Q;\mathbb R^d)\times W^{k_0,2}(\Omega_{k-1}^{(\tau)};\mathbb R^d)\),
as can be easily seen (the \(\kappa\)-terms are coercive and the \(f\)
terms can be absorbed as above). Further, it is weakly lower
semicontinuous due to the assumptions on \(E\) and \(R\) and the
convexity of the other terms. So there exists a minimizing sequence with
a weakly convergent subsequence.

Moreover the set over which we minimize in \eqref{eqn:minimization-Jktau} is weakly closed by the
following argument. The Ciarlet-Ne\v{c}as condition is weakly closed,
because weak convergence in \(W^{k_0+2,2}(Q;\mathbb R^d)\) implies
uniform convergence of \(\eta\) and \(\nabla\eta\). Further, we stay strictly 
away from \(\det\nabla\eta=0\) by the assumption \eqref{as:E-lower-bound-det}.
Moreover, the normal coupling is also a weakly closed condition, since
continuous functions up to the boundary are compact in \(W^{k_0+2,2}\)
resp.\ \(W^{k_0,2}\), we get that the trace and in particular the
coupling is preserved in the weak limit.
\end{proof}

\subsubsection{The discrete weak Euler-Lagrange equation}\label{the-discrete-weak-el-equation}
Here we derive the Euler-Lagrange equation for the minimizer \((\eta_k^{(\tau)},v_k^{(\tau)})\) of
\(\mathcal J_k^{(\tau)}\).

\emph{Fluid-only equation}.

As \((\eta_k^{(\tau)},v_k^{(\tau)})\) is a minimizer of
\(\mathcal J_k^{(\tau)}\), we get in particular that \(v_k^{(\tau)}\) is
a minimizer of \(\mathcal J_k^{(\tau)}(\eta_k^{(\tau)},\cdot )\) over all
\(v\) satisfying
\(v\circ\eta_k^{(\tau)} \cdot n_{k-1}^{(\tau)}=\tfrac{\eta_k^{(\tau)}-\eta_{k-1}^{(\tau)}}{\tau} \cdot n_{k-1}^{(\tau)}\).
Since the functional \(\mathcal J_k^{(\tau)}(\eta_k^{(\tau)},\cdot )\)
is convex on \(W^{k_0,2}(\Omega_{k-1}^{(\tau)};\mathbb R^d)\), we get
that the following Euler-Lagrange equation for the minimizer
\(v_k^{(\tau)}\) is satisfied:
\begin{align}\label{eqn:minimizer-fluid-only-eqn}
0&=\nu \langle\varepsilon v_k^{(\tau)}, \varepsilon \xi\rangle_{\Omega_{k-1}^{(\tau)}} + \kappa\langle\nabla^{k_0} v_k^{(\tau)}, \nabla^{k_0}\xi\rangle_{\Omega_{k-1}^{(\tau)}}+\rho_f \left\langle \tfrac{v_k^{(\tau)} \circ \Phi_{k-1}^{(\tau)}-w_{f,k-1}^{(\tau)}}{h}, \xi \circ\Phi_{k-1} \right\rangle_{\Omega_{0}} \\
&- \langle f_k^{(\tau)},\xi\rangle_{\Omega_{k-1}^{(\tau)}} +a\left\langle \tfrac{\eta_{k}^{(\tau)}-\eta_{k-1}^{(\tau)}}{\tau} -v_k^{(\tau)},\xi\right\rangle_{\partial \Omega_{k-1}^{(\tau)}} \nonumber
\end{align}
holds for all
\(\xi\in C^\infty(\overline \Omega_k^{(\tau)}; \mathbb R^d)\) with
\(\xi\cdot n_{k-1}^{(\tau)}=0\) and \(\operatorname{div}\xi=0\) in
\(\Omega_{k-1}^{(\tau)}\).

\noindent\emph{Coupled equation}

Let us now derive the discrete coupled equation. So we take
functions \(\xi\in C^\infty(Q;\mathbb R^d)\) and
\(\phi\in W^{k_0,2}(\Omega;\mathbb R^d)\) with
\(\operatorname{div}\xi=0\) in \(\Omega_{k-1}^{(\tau)}\) and
\(\phi=\xi\circ\eta_{k-1}^{(\tau)}\) in \(Q\). Now let us take the
perturbation with the scaling \((\phi, \xi/\tau)\) (which is the one that preserves the coupling condition on the interface.) That is, we
differentiate
\[t\mapsto \mathcal J_k^{(\tau)}(\eta_k^{(\tau)}+t\phi,v_k^{(\tau)}+t\xi/\tau)\]
at \(t=0\). The differentiation is possible, as the perturbation
\((\eta_k^{(\tau)}+t\phi,v_k^{(\tau)} + t\xi/\tau)\) is admissible for
small enough \(t\), see \cite[Lemma 2.7]{benesovaVariationalApproachHyperbolic2023}.
%\noteMalte[inline]{This is unclear to me as written. We can use the distance to collision lemma here, or argue directly by mapping degree/reference the similar argument in the old paper, but it should be done.}

The resulting equation is 
\begin{equation}\label{eqn:minimizer-coupled-eqn}
\begin{aligned} 
DE(\eta_k^{(\tau)})\langle\phi\rangle+D_2 R\left(\eta_{k-1}^{(\tau)}, \tfrac{\eta_k^{(\tau)}-\eta_{k-1}^{(\tau)}}{\tau}\right) \langle \phi \rangle  + 2\kappa^{a_0}\langle \nabla^{k_0+2}\eta_k^{(\tau)}, \nabla^{k_0+2}\phi \rangle +
\\ 2\kappa\left\langle \nabla^{k_0+2}\tfrac{\eta_k^{(\tau)}-\eta_{k-1}^{(\tau)}}{\tau} , \nabla^{k_0+2}\phi \right\rangle
+\rho_s\left\langle\tfrac{\tfrac{\eta_k^{(\tau)}-\eta_{k-1}^{(\tau)}}{\tau}-w_{s,k-1}^{(\tau)}}{h}, \phi\right\rangle- \rho_s\left\langle f_k^{(\tau)}\circ\eta_{k-1}^{(\tau)},\phi \right\rangle
\\+\nu \langle\varepsilon v_k^{(\tau)}, \varepsilon \xi\rangle_{\Omega_{k-1}^{(\tau)}}
 + \kappa\langle\nabla^{k_0} v_k^{(\tau)}, \nabla^{k_0}\xi\rangle_{\Omega_{k-1}^{(\tau)}}+\rho_f \left\langle \tfrac{v_k^{(\tau)} \circ \Phi_{k-1}^{(\tau)}-w_{f,k-1}^{(\tau)}}{h}, \xi \circ\Phi_{k-1}^{(\tau)} \right\rangle_{\Omega_{0}} \\- \rho_f\langle f_k^{(\tau)},\xi\rangle_{\Omega_{k-1}^{(\tau)}}=0
\end{aligned}
\end{equation} for all
\(\xi\in C^\infty(\overline \Omega;\mathbb R^d)\),
\(\operatorname {div} \xi =0\) in \(\Omega_{k-1}^{(\tau)}\) and
\(\phi\in W^{k_0,2}(Q;\mathbb R^d)\) with
\(\phi = \xi\circ\eta_{k-1}^{(\tau)}\) in \(Q\).

\begin{lemma}[Discrete energy estimates]\label{lem:discrete-energy-estimates}
For the constructed $\eta_k^{(\tau)},v_k^{(\tau)}$ the following energy estimate holds for all $1\leq k\leq T/\tau$:
\[\begin{aligned} 
E(\eta_k^{(\tau)}) &+ \kappa^{a_0}\|\nabla^{k_0+2}\eta_k^{(\tau)}\|^2 + \sum_{j=1}^k \Bigg[ \kappa\tau\left\| \nabla^{k_0+2}\tfrac{\eta_j^{(\tau)}-\eta_{j-1}^{(\tau)}}{\tau} \right\|^2  +\tau R\left(\eta_{j-1}^{(\tau)}, \tfrac{\eta_j^{(\tau)}-\eta_{j-1}^{(\tau)}}{\tau}\right)
\\
&+\rho_s \frac{\tau}{8h}\left\| \tfrac{\eta_j^{(\tau)}-\eta_{j-1}^{(\tau)}}{\tau}\right\|^2
+\frac{\tau\nu}{2} \|\varepsilon v_j^{(\tau)}\|_{\Omega_{j-1}^{(\tau)}}^2 + \kappa\frac{\tau}{2}\|\nabla^{k_0} v_j^{(\tau)}\|_{\Omega_{j-1}^{(\tau)}}^2+  \frac{a\tau}{2} \left\|\tfrac{\eta_j^{(\tau)}-\eta_{j-1}^{(\tau)}}{\tau} - v\right\|_{\partial \Omega_{j-1}^{(\tau)}}^2\\
&+
\rho_f \frac{\tau}{8h}\|v_j^{(\tau)}\circ\Phi_{j-1}^{(\tau)}\|_{\Omega_0}^2\Bigg]
\\
&\leq
E(\eta_{0})+ \kappa^{a_0}\|\nabla^{k_0+2}\eta_{0}\|^2
\\
&\quad+ \sum_{j=1}^k\left[\rho_s\frac{\tau}{h}\left\|w_{s,j-1}^{(\tau)}\right\|^2
+\rho_f\frac{\tau }{h} \left\| w_{f,j-1}^{(\tau)} \right\|_{\Omega_{0}}^2 
+ \rho_s 2\tau h\|f_j^{(\tau)}\circ\eta_{j-1}^{(\tau)}\|^2 
+\rho_f 2\tau h \|f_j^{(\tau)}\|_{\Omega_{j-1}^{(\tau)}}^2\right]. 
\end{aligned}\]
\end{lemma}
\begin{proof}
For the discrete energy estimates, we compare
\((\eta_k^{(\tau)},v_k^{(\tau)})\) with \((\eta_{k-1}^{(\tau)},0)\) in
the minimization \eqref{eqn:minimization-Jktau}. That is,
\[\mathcal J_k^{(\tau)}(\eta_k^{(\tau)},v_k^{(\tau)})\leq\mathcal J_k^{(\tau)}(\eta_{k-1}^{(\tau)},0).\]
Writing this out gives 
\[\begin{aligned} 
E(\eta_k^{(\tau)})&+\tau R\left(\eta_{k-1}^{(\tau)}, \tfrac{\eta_k^{(\tau)}-\eta_{k-1}^{(\tau)}}{\tau}\right) + \kappa^{a_0}\|\nabla^{k_0+2}\eta_k^{(\tau)}\|^2 + \kappa\tau\left\| \nabla^{k_0+2}\tfrac{\eta_k^{(\tau)}-\eta_{k-1}^{(\tau)}}{\tau} \right\|^2
\\
+\frac{a\tau}{2}
 &\int_{\partial\Omega_{k-1}}\left| \tfrac{\eta_k^{(\tau)}-\eta_{k-1}^{(\tau)}\circ\eta_{k-1}^{-1}}{\tau}-v_k \right |^2\, dx +\rho_s\frac{\tau h}{2}\left\|\tfrac{\tfrac{\eta_k^{(\tau)}-\eta_{k-1}^{(\tau)}}{\tau}-w_{s,k-1}^{(\tau)}}{h}\right\|^2-\tau \rho_s \left\langle f_k^{(\tau)}\circ\eta_{k-1}^{(\tau)},\tfrac{\eta_k^{(\tau)}-\eta_{k-1}^{(\tau)}}{\tau} \right\rangle
\\
&+\frac{\tau\nu}{2} \|\varepsilon v_k^{(\tau)}\|_{\Omega_{k-1}^{(\tau)}}^2 + \kappa\frac{\tau}{2}\|\nabla^{k_0} v_k^{(\tau)}\|_{\Omega_{k-1}^{(\tau)}}^2+\rho_f\frac{\tau h}{2} \left\| \tfrac{v_k^{(\tau)} \circ \Phi_{k-1}^{(\tau)}-w_{f,k-1}^{(\tau)}}{h} \right\|_{\Omega_{0}}^2 - \tau \rho_f\langle f_k^{(\tau)},v_k^{(\tau)}\rangle_{\Omega_{k-1}^{(\tau)}}
\\
&\leq
E(\eta_{k-1}^{(\tau)})+ \kappa^{a_0}\|\nabla^{k_0+2}\eta_{k-1}^{(\tau)}\|^2
+\rho_s\frac{\tau h}{2}\left\|\tfrac{w_{s,k-1}^{(\tau)}}{h}\right\|^2
+\rho_f\frac{\tau h}{2} \left\| \tfrac{w_{f,k-1}^{(\tau)} }{h} \right\|_{\Omega_{0}}^2.
\end{aligned}\]
Now we sum this over \(j=1,\dots,k\), so that we obtain
the energy estimates 
\[\begin{aligned} 
E(\eta_k^{(\tau)})&+ \kappa^{a_0}\|\nabla^{k_0+2}\eta_k^{(\tau)}\|^2 + \frac{1}{2} \sum_{j=1}^k\Bigg[ \kappa\tau\left\| \nabla^{k_0+2}\tfrac{\eta_j^{(\tau)}-\eta_{j-1}^{(\tau)}}{\tau} \right\|^2 +\tau R\left(\eta_{j-1}^{(\tau)}, \tfrac{\eta_j^{(\tau)}-\eta_{j-1}^{(\tau)}}{\tau}\right)
\\
&+\frac{a\tau}{2}
 \int_{\partial\Omega_{j-1}}\left| \tfrac{\eta_j^{(\tau)}-\eta_{j-1}^{(\tau)}\circ\eta_{j-1}^{-1}}{\tau}-v_j \right |^2\, dx +\rho_s\frac{\tau h}{2}\left\|\tfrac{\tfrac{\eta_j^{(\tau)}-\eta_{j-1}^{(\tau)}}{\tau}-w_{s,j-1}^{(\tau)} }{h}\right\|^2-\tau \rho_s \left\langle f_j^{(\tau)}\circ\eta_{j-1}^{(\tau)},\tfrac{\eta_j^{(\tau)}-\eta_{j-1}^{(\tau)}}{\tau} \right\rangle
\\
&+\frac{\tau\nu}{2} \|\varepsilon v_j^{(\tau)}\|_{\Omega_{j-1}^{(\tau)}}^2 + \kappa\frac{\tau}{2}\|\nabla^{k_0} v_j^{(\tau)}\|_{\Omega_{j-1}^{(\tau)}}^2+\rho_f\frac{\tau h}{2} \left\| \tfrac{v_j^{(\tau)} \circ \Phi_{j-1}^{(\tau)}-w_{f,j-1}^{(\tau)}}{h} \right\|_{\Omega_{0}}^2 - \tau \langle f_j^{(\tau)},v_j^{(\tau)}\rangle_{\Omega_{j-1}^{(\tau)}} \Bigg]
\\
&\leq
E(\eta_0)+ \kappa^{a_0}\|\nabla^{k_0+2}\eta_0\|^2
+C(h)\|w\|_{L^2(L^2)}^2 + C\|f\|_{L^2(L^\infty)} .
\end{aligned}\]

To get uniform in \(\tau\) energy estimates, using the Young
inequality we get in inertial and forcing term in the solid
\[\begin{aligned}
\rho_s \frac{\tau}{4h}\left\| \tfrac{\eta_k^{(\tau)}-\eta_{k-1}}{\tau}\right\|^2 -\rho_s\frac{\tau}{2h}\|w_{s,k-1}^{(\tau)}\|^2 - \rho_s\frac{\tau}{8h}^2 \left\| \tfrac{\eta_k^{(\tau)}-\eta_{k-1}^{(\tau)}}{\tau}\right\|^2 - \rho_s 2\tau h\|f_k^{(\tau)}\circ\eta_{k-1}^{(\tau)}\|  \\\leq
\rho_s\frac{\tau h}{2}\left\|\tfrac{\tfrac{\eta_k^{(\tau)}-\eta_{k-1}^{(\tau)}}{\tau}-w_{s,k-1}^{(\tau)}}{h}\right\|^2-\tau \rho_s\left\langle f_k^{(\tau)}\circ\eta_{k-1}^{(\tau)},\tfrac{\eta_k^{(\tau)}-\eta_{k-1}^{(\tau)}}{\tau} \right\rangle,
\end{aligned}\]
as well as in the fluid 
\[\begin{aligned}
\rho_f \frac{\tau}{4h}\|v_k^{(\tau)}\circ\Phi_{k-1}^{(\tau)}\|_{\Omega_0}^2 - \rho_f \frac{\tau}{2h}\|w_{f,k-1}^{(\tau)}\|_{\Omega_0}^2 -\rho_f \frac{\tau}{8h}\|v_k^{(\tau)}\|_{\Omega_{k-1}^{(\tau)}}^2-\rho_f2\tau h \|f_k^{(\tau)}\|_{\Omega_{k-1}^{(\tau)}}^2
\\
\leq \rho_f\frac{\tau h}{2} \left\| \tfrac{v_k^{(\tau)} \circ \Phi_{k-1}^{(\tau)}-w_{f,k-1}^{(\tau)}}{h} \right\|_{\Omega_{0}}^2 - \tau \rho_f \langle f_k^{(\tau)},v_k^{(\tau)}\rangle_{\Omega_{k-1}^{(\tau)}}.
\end{aligned}\]

As \(\Phi_j^{(\tau)}\) is a diffeomorphism with bounded change of volume by by Proposition \ref{prop:estimates-of-the-flow-map} below,
we have
\(\|v_k^{(\tau)}\circ\Phi_{k-1}^{(\tau)}\|_{\Omega_0}=\|v_k^{(\tau)}\|_{\Omega_{k-1}^{(\tau)}}\)
. So we get 
\[\begin{aligned} 
E(\eta_k^{(\tau)})&+\tau R\left(\eta_{k-1}^{(\tau)}, \tfrac{\eta_k^{(\tau)}-\eta_{k-1}^{(\tau)}}{\tau}\right) + \kappa^{a_0}\|\nabla^{k_0+2}\eta_k^{(\tau)}\|^2 + \kappa\tau\left\| \nabla^{k_0+2}\tfrac{\eta_k^{(\tau)}-\eta_{k-1}^{(\tau)}}{\tau} \right\|^2
\\
&+\rho_s \frac{\tau}{8h}\left\| \tfrac{\eta_k^{(\tau)}-\eta_{k-1}^{(\tau)}}{\tau}\right\|^2
+\frac{\tau\nu}{2} \|\varepsilon v_k^{(\tau)}\|_{\Omega_{k-1}^{(\tau)}}^2 + \kappa\frac{\tau}{2}\|\nabla^{k_0} v_k^{(\tau)}\|_{\Omega_{k-1}^{(\tau)}}^2+ 
\rho_f \frac{\tau}{8h}\|v_k^{(\tau)}\circ\Phi_{k-1}^{(\tau)}\|_{\Omega_0}^2
\\
&\leq
E(\eta_{k-1}^{(\tau)})+ \kappa^{a_0}\|\nabla^{k_0+2}\eta_{k-1}^{(\tau)}\|^2
+\rho_s\frac{\tau}{h}\left\|w_{s,k-1}^{(\tau)}\right\|^2
+\rho_f\frac{\tau }{h} \left\| w_{f,k-1}^{(\tau)} \right\|_{\Omega_{0}}^2 
+ \rho_s 2\tau h\|f_k^{(\tau)}\circ\eta_{k-1}^{(\tau)}\|^2 
\\
&\quad+\rho_f2\tau h \|f_k^{(\tau)}\|_{\Omega_{k-1}^{(\tau)}}^2.
\end{aligned}\] 
Now we sum this over \(j=1,\dots,k\), so that we obtain the energy estimates in the statement of this lemma.
\end{proof}

The following proposition is contained in \cite[Proposition 4.6, Corollary 4.7]{benesovaVariationalApproachHyperbolic2023}

\begin{proposition}[Estimates of the flow map]\label{prop:estimates-of-the-flow-map}
For every $\varepsilon>0$, the discrete flow map satisfies the estimate on its Jacobian \[\frac {1}{1+\varepsilon} \leq \det \nabla\Phi_k^{(\tau)}\leq 1+\varepsilon\]
for $\tau$ small enough, in dependence of $\varepsilon$ (i.e.\ for all $0<\tau<\tau_0(\varepsilon)$).
Consequently, it holds
\[\lim_{\tau \to 0} \det \nabla \Phi^{(\tau)}=1\] uniformly on
\([0,T]\). Moreover the maps \(\Phi^{(\tau)}\) are uniformly Lipschitz-continuous
in space, so that the Lipschitz constant
\(\operatorname{Lip}\Phi^{(\tau)}(t)\) is independent of \(\tau\) and \(t\).
\end{proposition}

\subsubsection{Weak limit $\tau\to 0$}\label{sec:weak-limit}

We define the piecewise constant and piecewise affine interpolations as
\begin{equation}\label{eqn:tau-interpolations}
\begin{aligned}
\eta^{(\tau)}(t, x)&=\eta_{k}^{(\tau)}(x) & \text { for } \tau (k-1) &\leq t<\tau k,\quad x\in Q, \\
\underline{\eta}^{(\tau)}(t, x)&=\eta_{k-1}^{(\tau)}(x) & \text { for } \tau (k-1) &\leq t<\tau k,\quad x\in Q, \\
\tilde{\eta}^{(\tau)}(t, x)&=\tfrac{\tau k-t}{\tau} \eta_{k-1}^{(\tau)}(x)+\tfrac{t-\tau (k-1)}{\tau} \eta_k^{(\tau)}(x) & \text { for } \tau (k-1) &\leq t<\tau k,\quad x\in Q, \\
v^{(\tau)}(y)&=v_{k-1}^{(\tau)}(y) & \text { for } \tau (k-1) &\leq t<\tau k,\quad y \in \Omega_{k-1}^{(\tau)}, \\
\Phi^{(\tau)}(t, y)&=\Phi_{k-1}^{(\tau)}(y) & \text { for } \tau (k-1) &\leq t<\tau k,\quad y \in \Omega_{k-1}^{(\tau)}, \\ \tilde{\Phi}^{(\tau)}(t, y)&=\tfrac{\tau k-t}{\tau} \Phi_{k-1}^{(\tau)}(y)+\tfrac{t-\tau (k-1)}{\tau} \Phi_k^{(\tau)}(y) & \text { for } \tau (k-1) &\leq t<\tau k,\quad y \in \Omega_{k-1}^{(\tau)},
\end{aligned}
\end{equation}
as well as \(\Omega^{(\tau)}(t)=\Omega_{k-1}^{(\tau)}\) for
\(\tau (k-1) \leq t<\tau k\).

By the estimate of Lemma \ref{lem:discrete-energy-estimates} we have (for a nonlabelled subsequence $\tau\to 0$) that
\begin{equation}\label{eqn:weak-limit-tau}
\begin{aligned}
\eta^{(\tau)} &\overset{\ast}\rightharpoonup \eta \quad & \text{in } &L^\infty((0,h);W^{k_0+2,2}(Q;\mathbb R^d)), \\
\partial_t \tilde \eta^{(\tau)} &\rightharpoonup \partial_t\eta \quad & \text{in } &L^2((0,h);W^{k_0+2,2}(Q;\mathbb R^d)), \\
v^{(\tau)}&\overset\eta\rightharpoonup v\quad & \text{in } &L^2((0,h);W^{k_0,2}(\Omega(t);\mathbb R^d)).
\end{aligned}
\end{equation} Since we have a compact embedding
\[W^{1,2}((0,h);W^{k_0,2}(Q;\mathbb R^d)) \hookrightarrow \hookrightarrow C^{(1/2)^-}([0,h];C^{1,\alpha^-}(Q;\mathbb R^d)), \]
we have
\[\eta^{(\tau)}\to \eta \quad \text{ in } C^{(1/2)^-}([0,h];C^{1,\alpha^-}(Q;\mathbb R^d)),\]
which in particular means that \(\eta^{(\tau)}\rightarrow \eta\) uniformly
on \([0,h]\times \partial Q\) and also
\(\det \nabla \eta^{(\tau)} \to \det \nabla \eta\) in
\(C([0,T]\times \overline Q)\).

Now we verify that the coupling condition holds in the limit.
\begin{lemma}
The limit $(\eta,v)$ satisfies the coupling condition
\begin{equation}
v \cdot n=\partial_t \eta \circ\eta^{-1} \cdot n \quad \text{on } (0,T)\times\partial \Omega(t).
\end{equation}
\end{lemma}
\begin{proof}
We know
that for the approximate solutions $(\eta^{(\tau)},v^{(\tau)})$ the following coupling condition
holds
\begin{equation*}\label{eqn:coupling-condition-tau}
v^{(\tau)}\circ \eta^{(\tau)}\cdot n^{(\tau)}=\partial_t\eta^{(\tau)}\cdot n^{(\tau)} \quad \text{on } [0,T]\times\partial Q.
\end{equation*}
 Let us now operate in the Eulerian
domain and fix an arbitrary test function
\(\psi\in C_0((0,T)\times\Omega)\). Now write the solid part using the
divergence theorem, change of variables, and the above convergences:
\[\begin{aligned}
&\int_0^T\int_{\partial\Omega^{(\tau)}} \psi \partial_t\eta^{(\tau)} \circ(\eta^{(\tau)})^{-1} \cdot dn^{(\tau)}\,dt = \int_0^T\int_{\eta^{(\tau)}(Q)}\operatorname{div}(\psi \partial_t\eta^{(\tau)} \circ(\eta^{(\tau)})^{-1})\,dx \,dt
\\
&= \int_0^T \int_{\eta^{(\tau)}(Q)} \nabla \psi\cdot \partial_t\eta^{(\tau)} \circ(\eta^{(\tau)})^{-1} + \psi\nabla \partial_t\eta^{(\tau)} \circ(\eta^{(\tau)})^{-1}:\nabla^T  (\eta^{(\tau)})^{-1
} \,dx\,dt 
\\
&= \int_0^T\int_Q \left(\nabla \psi \circ\eta^{(\tau)}\cdot \partial_t\eta^{(\tau)}+\psi\circ\eta^{(\tau)} \nabla\partial_t \eta^{(\tau)} :\nabla ^T (\eta^{(\tau)})^{-1}\circ \eta^{(\tau)} \right) \det \nabla \eta^{(\tau)}\,dx\,dt
\\
&\to   \int_0^T\int_Q \left(\nabla \psi \circ\eta \cdot \partial_t\eta +\psi\circ\eta \nabla\partial_t \eta  :\nabla ^T \eta^{-1}\circ \eta \right) \det \nabla \eta\,dx \,dt  \\
&= \int_0^T\int_{\eta(Q)} \nabla\psi \cdot \partial_t \eta\circ \eta^{-1} + \psi \nabla \partial_t\eta \circ \eta^{-1}:\nabla^T \eta^{-1} \,dt
\\
&= \int_0^T\int_{\eta(Q)}\operatorname{div}(\psi \partial_t\eta \circ\eta^{-1})\,dx \,dt= \int_0^T\int_{\partial\eta(Q)}\psi \partial_t\eta \circ\eta^{-1} \cdot dn\,dt.
\end{aligned}\] 
Similarly on the fluid domain, we can write by the
divergence theorem
\[\begin{aligned}\int_0^T\int_{\partial \Omega^{(\tau)}(t)} \psi v^{(\tau)}\cdot dn^{(\tau)}\,dt = \int_0^T \int_{\Omega^{(\tau)}(t)}\nabla \psi \cdot v^{(\tau)} + \psi \operatorname{div} v^{(\tau)} \,dx\,dt \\ \to \int_0^T \int_{\Omega(t)}\nabla\psi \cdot v+\psi \operatorname{div} v\,dx\,dt = \int_0^T \int_{\partial \Omega(t)}\psi v\cdot dn\,dt.\end{aligned}\]
Now recall that by \eqref{eqn:coupling-condition-tau} we have that the left hand sides of the last two equations are equal, so that this proves that the coupling condition holds.
\end{proof}

%\noteMalte[inline]{I would promote this to a technical lemma maybe?}

We also pass to the limit in the flow map. By the estimates of the
previous section, Proposition \ref{prop:estimates-of-the-flow-map}, there is
\(\Phi\in C([0,T];W^{1,\infty}(\Omega_0;\mathbb R^d))\) such that
\[\Phi^{(\tau)}\to\Phi \quad \text{in }C([0,T];C^{0,\alpha}(\Omega_0;\mathbb R^d).\]
By Proposition \ref{prop:estimates-of-the-flow-map} we know that \(\det\nabla\Phi=1\) a.e. Moreover, passing to the
limit by \ref{eqn:Phiktau-def} definition of \(\Phi^{(\tau)}\) we have 
\begin{equation}\label{eqn:flow-map-dt-equal-v}
\partial_t \Phi(t)=\lim _{\tau \rightarrow 0} \partial_t \tilde{\Phi}^{(\tau)}(t)=\lim _{\tau \rightarrow 0} u^{(\tau)}(t) \circ \Phi^{(\tau)}(t) =v(t) \circ \Phi(t).
\end{equation}
From here it also follows that \(\Phi_t\colon \Omega_0\to\Omega(t)\) is
a diffeomorphism, \(t\in[0,T]\).

\subsection{\texorpdfstring{Passing to the limit \(\tau\to 0\) in the Euler-Lagrange equation}{Passing to the limit \textbackslash tau\textbackslash to 0 in the Euler-Lagrange equation}}\label{passing-to-the-limit-tauto-0-in-el-eqn.}

\subsubsection{Fluid only equation}\label{sec:fluid-only-eqn}
\begin{proposition}\label{prop:fluid-only-eqn-taulimit}
The limit solution $v$ from \eqref{eqn:weak-limit-tau} satisfies the equation
\[\int_0^h \nu \langle\varepsilon v, \varepsilon \xi \rangle_{\Omega(t)} + \kappa\langle\nabla^{k_0} v, \nabla^{k_0}\xi\rangle_{\Omega(t)}+\rho_f \left\langle \tfrac{v \circ \Phi-w^{f}}{h}, \xi \circ\Phi \right\rangle_{\Omega_{0}} - \langle f,\xi \rangle_{\Omega(t)}\,dt- a \langle\partial_t \eta \circ \eta^{-1}-v,\xi\rangle_{\partial \Omega(t)}=0\]
holds for all 
\(\xi\in C^\infty([0,h] \times \overline \Omega(t);\mathbb R^d)\),
\(\nabla^\ell(\xi\cdot n)=0\) on \(\partial \Omega(t)\), \(\operatorname{div}\xi=0\)
in \(\Omega(t)\), with \(\xi(h)=0\) and \(\nabla^\ell( \xi\cdot n)=0\),
\(\ell=1,\dots,k_0\)
\end{proposition}

\begin{proof}
We have a fluid only equation as in \eqref{eqn:minimizer-fluid-only-eqn}, denoting this in the \((\tau)\)-notation of \eqref{eqn:tau-interpolations}, this reads as
\begin{equation}\label{eqn:fluid-only-eqn-tau}
\begin{aligned}
\int_0^h \nu \langle\varepsilon v^{(\tau)}, \varepsilon \xi^{(\tau)}\rangle_{\Omega^{(\tau)}} + \kappa\langle\nabla^{k_0} v^{(\tau)}, \nabla^{k_0}\xi^{(\tau)}\rangle_{\Omega^{(\tau)}}+\rho_f \left\langle \tfrac{v^{(\tau)} \circ \Phi^{{(\tau)}}-w^{(\tau),f}}{h}, \xi^{(\tau)} \circ\Phi^{(\tau)}\right\rangle_{\Omega_{0}}  \\ - \langle f^{(\tau)},\xi^{(\tau)}\rangle_{\Omega^{(\tau)}} -a \langle\partial_t \tilde\eta^{(\tau)} \circ (\underline{\eta}^{(\tau)})^{-1}-v^{(\tau)},\xi\rangle_{\partial \Omega(t)}\,dt=0
\end{aligned}
\end{equation}
where
\(\xi^{(\tau)}\equiv \xi_k^{(\tau)} \in C^\infty(\overline \Omega_{k-1}^{(\tau)}; \mathbb R^d)\)
on \(((k-1)\tau,k\tau]\) with \(\xi_k^{(\tau)}\cdot n_k^{(\tau)}=0\) on
\(\partial \Omega_k^{(\tau)}\) and
\(\operatorname{div}\xi_k^{(\tau)}=0\) in \(\Omega_k^{(\tau)}\).

Approximation of
\(\xi\):
Let us have a fixed test function for the limit equation, that is
\(\xi\in C^\infty([0,h] \times \overline \Omega(t);\mathbb R^d)\),
\(\xi\cdot n=0\) on \(\partial \Omega(t)\), \(\operatorname{div}\xi=0\)
in \(\Omega(t)\), with \(\xi(h)=0\) and \(\nabla^\ell (\xi\cdot n)=0\),
\(\ell=1,\dots,k_0\).

Let now \(\delta>0\) be fixed. Consider a smooth cutoff
\(\psi_\delta\) which is \(1\) on \(\delta\)-neighborhood of
\(\partial\eta(Q)\) and vanishes outside its
\(2\delta\)-neighborhood. That is,
\(\psi_\delta\in C^\infty(\mathbb R^d;\mathbb R^d)\) such that
\(\psi_\delta(y)=1\) if
\(\operatorname{dist}(y,\partial \eta(Q))\leq \delta\) and
\(\psi_\delta(y)=0\) if
\(\operatorname{dist}(y,\partial\eta(Q))\geq 2\delta\).

We now take the approximations \(\xi^{(\tau)}_\delta\) and
\(\xi_\delta\) given by Proposition \ref{prop:approximation-fluid-only} (ii). By this we see that
\(\xi_\delta^{(\tau)}\) is an admissible test function for
equation \eqref{eqn:fluid-only-eqn-tau}. By the same proposition we see that for fixed
\(\delta>0\) we can pass to the limit in the equation as
\(\tau \to 0\) and see that the limit equation is satisfied with
\(\xi_\delta\).

At this point we will comment in a bit more detail about the convergence
\(\tau\to 0\). For any \(\xi\in C^\infty([0,T] \times \overline \Omega(t);\mathbb R^d)\),
\(\xi\cdot n=0\) on \(\partial \Omega(t)\), \(\operatorname{div}\xi=0\),
with \(\xi(T)=0\)), construct
\(\xi^{(\tau)}(t) \in C^\infty(\overline \Omega_{k-1}^{(\tau)}; \mathbb R^d)\)
on \(t\in((k-1)\tau,k\tau]\) with
\(\xi_k^{(\tau)}\cdot n_{k-1}^{(\tau)}=0\) on
\(\partial \Omega_k^{(\tau)}\) and
\(\operatorname{div}\xi_k^{(\tau)}=0\) in \(\Omega_{k-1}^{(\tau)}\) as above. By
the lemma we have that \[\begin{aligned}
\xi^{(\tau)}_\delta&\overset\eta\to\xi_\delta\quad \text{in }L^2((0,T);W^{k_0,2}(\Omega;\mathbb R^d))
\end{aligned}\] holds in the sense that
\[(\nabla^\ell\xi^{(\tau)}_\delta)_0 \to(\nabla^\ell\xi_\delta)_0\quad \text{in }L^2((0,T);L^2(\Omega;\mathbb R^d))\]
for \(\ell=0,\dots,k_0\), where \((\cdot)_0\) is the extension by \(0\)
to \(\Omega\). In particular this means
\[\int_0^T \|(\nabla^\ell\xi^{(\tau)}_\delta)_0 -(\nabla^\ell\xi_\delta)_0\|_{L^2(\Omega;\mathbb R^d)}^2\,dt \to 0\]
for \(\ell=0,\dots,k_0\).

Then we show in a standard way, that for every
\(u^{(\tau)} \in L^2((0,T);L^2(\Omega^{(\tau)}(t);\mathbb R^d))\) and
\(u\in L^2((0,T);L^2(\Omega(t);\mathbb R^d))\) with
\(u^{(\tau)}\rightharpoonup u\) it holds that
\[\int_0^T \langle u^{(\tau)},\nabla^\ell\xi_\delta^{(\tau)}\rangle_{\Omega^{(\tau)}(t)}\,dt \to \int _0^T \langle u, \nabla^\ell \xi_\delta\rangle_{\Omega(t)} \,dt.\]
Indeed, we can write \[\begin{aligned}
&\left| \int_0^T \langle u^{(\tau)},\nabla^\ell\xi_\delta^{(\tau)}\rangle_{\Omega^{(\tau)}(t)}- \langle u, \nabla^\ell \xi_\delta\rangle_{\Omega(t)} \,dt \right| \\
&=\left|\int_0^T \langle u^{(\tau)}_0,(\nabla^\ell\xi_\delta^{(\tau)})_0 - (\nabla^\ell \xi_\delta)_0\rangle_{\Omega}- \langle u_0-u^{(\tau)}_0, (\nabla^\ell \xi_\delta)_0\rangle_{\Omega} \,dt  \right| \\ \leq&\|u_0^{(\tau)}\|_{L^2((0,T);L^2(\Omega))} \|(\nabla^\ell\xi_\delta^{(\tau)})_0 -(\nabla^\ell\xi_\delta)_0\|_{L^2((0,T);L^2(\Omega;\mathbb R^d))}^2 + \left|\int_0^T  \langle u_0-u^{(\tau)}_0, (\nabla^\ell \xi_\delta)_0\rangle_{\Omega} \,dt  \right| \to 0
\end{aligned}\] where the convergence is respectively by boundedness of
\(u_0^{(\tau)}\), strong convergence of
\((\nabla^\ell\xi_\delta^{(\tau)})_0\) and weak convergence of
\(u^{(\tau)}_0\), see Section \ref{sec:spaces-on-a-moving-domain}. %\todo{backref to convergence \section}

\[\int_0^h \nu \langle\delta v, \delta \xi_\delta \rangle_{\Omega(t)} + \kappa\langle\nabla^{k_0} v, \nabla^{k_0}\xi_\delta\rangle_{\Omega(t)}+\rho_f \left\langle \tfrac{v \circ \Phi-w^{f}}{h}, \xi_\delta \circ\Phi \right\rangle_{\Omega_{0}} - \langle f,\xi_\delta \rangle_{\Omega(t)}- a \langle\partial_t \eta \circ \eta^{-1}-v,\xi_\delta\rangle_{\partial \Omega(t)}\,dt=0\]

By the same lemma we can now pass here with
\(\delta \to 0\) with
\(\xi_\delta \to \xi\) in \(L^2((0,T);W^{k_0,2}(\Omega(t);\mathbb R^d))\)
to obtain the desired limit equation.
\end{proof}

\subsubsection{Coupled equation}\label{sec:coupled-eqn}
\begin{proposition}\label{prop:coupled-equation-taulimit}
For the limit $\eta,v$ from \eqref{eqn:weak-limit-tau} the following equation is satisfied
\begin{equation}\label{eqn:coupled-equation-taulimit}
\begin{aligned} 
\int_0^h DE(\eta)\langle\phi\rangle&+D_2 R\left(\eta, \partial_t\eta \right) \langle \phi \rangle  + 2\kappa^{a_0}\langle \nabla^{k_0+2}\eta, \nabla^{k_0+2}\phi \rangle + 2\kappa\left\langle \nabla^{k_0+2}\partial_t\eta , \nabla^{k_0+2}\phi \right\rangle
\\
&+\rho_s\left\langle\tfrac{\partial_t\eta- w \circ \eta_0}{h}, \phi\right\rangle^2-\rho_s\left\langle f\circ\eta,\phi \right\rangle
\\
&+\nu \langle\varepsilon v, \varepsilon \xi\rangle_{\Omega(t)} + \kappa\langle\nabla^{k_0} v, \nabla^{k_0}\xi\rangle_{\Omega(t)}+\rho_f \left\langle \tfrac{v \circ \Phi_t -w}{h}, \xi\circ\Phi \right\rangle_{\Omega_{0}} - \rho_f\langle f,\xi\rangle_{\Omega(t)}\,dt=0
\end{aligned}
\end{equation} for all
\(\xi\in C^\infty([0,h]\times \overline \Omega(t);\mathbb R^d)\),
\(\operatorname {div} \xi =0\) in \(\Omega(t)\) and
\(\phi\in W^{k_0+2,2}(Q;\mathbb R^d)\) with \(\phi = \xi\circ\eta\) on
\(Q\).

\end{proposition}

\begin{proof}
As for now, assume the minimizer \(\eta_k^{(\tau)}\) is always in the
interior of \(\mathcal E^{k_0}\). Thus as above \eqref{eqn:minimizer-coupled-eqn} it holds
\[\begin{aligned} 
DE( \eta_k^{(\tau)})\langle\phi\rangle&+D_2 R\left(\eta_{k-1}^{(\tau)}, \tfrac{\eta_k^{(\tau)}-\eta_{k-1}^{(\tau)}}{\tau}\right) \langle \phi \rangle  + 2\kappa^{a_0}\langle \nabla^{k_0+2}\eta_k^{(\tau)}, \nabla^{k_0+2}\phi \rangle
\\ &+ 2\kappa\left\langle \nabla^{k_0+2}\tfrac{\eta_k^{(\tau)}-\eta_{k-1}^{(\tau)}}{\tau} , \nabla^{k_0+2}\phi \right\rangle
+\rho_s\left\langle\tfrac{\tfrac{\eta_k^{(\tau)}-\eta_{k-1}^{(\tau)}}{\tau}-w_{s,k-1}^{(\tau)}}{h}, \phi\right\rangle- \rho_s\left\langle f_k^{(\tau)}\circ\eta_{k-1}^{(\tau)},\phi \right\rangle
\\&+\nu \langle\varepsilon v_k^{(\tau)}, \varepsilon \xi\rangle_{\Omega_{k-1}^{(\tau)}}
 + \kappa\langle\nabla^{k_0} v_k^{(\tau)}, \nabla^{k_0}\xi\rangle_{\Omega_{k-1}^{(\tau)}}+\rho_f \left\langle \tfrac{v_k^{(\tau)} \circ \Phi_{k-1}^{(\tau)}-w_{f,k-1}^{(\tau)}}{h}, \xi \circ\Phi_{k-1}^{(\tau)}\right\rangle_{\Omega_{0}} 
 \\&- \rho_f\langle f_k^{(\tau)},\xi\rangle_{\Omega_{k-1}^{(\tau)}}=0
\end{aligned}\] for all
\(\xi\in C^\infty(\overline \Omega;\mathbb R^d)\),
\(\operatorname {div} \xi =0\) in \(\Omega_{k-1}^{(\tau)}\) and
\(\phi\in W^{k_0+2,2}(Q;\mathbb R^d)\) with
\(\phi = \xi\circ\eta_{k-1}^{(\tau)}\).

Now we pass to the limit. We use as in Proposition \ref{prop:approximation-coupled-test-functions} that the
test functions can be approximated.

Let \(\xi\) be a test function for the limit equation, that is
\(\xi\in C^\infty([0,h]\times \overline \Omega(t); \mathbb R^d)\) with
\(\xi\cdot n=0\) on \(\partial\Omega\) and \(\operatorname{div}\xi=0\)
in \(\Omega(t)\). The corresponding solid test function \(\phi\) is
defined by \(\phi:=\xi\circ\eta\).

Pick \(\delta >0\). Then for this \(\delta\), find
\(\xi_\delta\) and \(\phi_\delta\) by Proposition \ref{prop:approximation-coupled-test-functions}.
As \(\Omega^{(\tau)}(t)\to \Omega(t)\) in the Hausdorff distance we have
that \(\xi_\delta\) is divergence free on \(\Omega^{(\tau)}\).
Necessarily it holds \(\phi_\delta = \xi_\delta\circ\eta\), so
we will use this notation.

For \(\tau\) small enough, \(\xi_\delta\), \(\phi_\delta\) is
an admissible test function for the approximate equation in \eqref{eqn:coupled-eqn-tau} .

Thus we get, after integrating over each \(\tau\) intervals and summing
this up over \(k\), that 
\begin{equation}\label{eqn:coupled-eqn-tau}
\begin{aligned} 
\int_0^h DE(\eta^{(\tau)})\langle\xi_\delta\circ\underline\eta^{(\tau)}\rangle+D_2 R\left(\underline\eta^{(\tau)}, \partial_t\tilde\eta^{(\tau)}\right) \langle \xi_\delta\circ\underline\eta^{(\tau)}\rangle  + 2\kappa^{a_0}\langle \nabla^{k_0+2}\overline\eta^{(\tau)}, \nabla^{k_0+2}(\xi_\delta\circ\underline \eta^{(\tau)}) \rangle\\ + 2\kappa\left\langle \nabla^{k_0} \partial_t\tilde\eta^{(\tau)}, \nabla^{k_0}(\xi_\delta\circ\underline\eta^{(\tau)}) \right\rangle
+\rho_s\left\langle\tfrac{\partial_t\tilde\eta^{(\tau)}-w_{s,k-1}^{(\tau)}}{h}, \xi_\delta\circ\underline\eta^{(\tau)} \right\rangle- \rho_s\left\langle \overline f^{(\tau)}\circ\underline\eta^{(\tau)},\xi_\delta\circ\underline\eta^{(\tau)} \right\rangle
\\
+\nu \langle\varepsilon v^{(\tau)}, \varepsilon \xi_\delta\rangle_{\Omega^{(\tau)}} + \kappa\langle\nabla^{k_0} v^{(\tau)}, \nabla^{k_0}\xi_\delta\rangle_{\Omega^{(\tau)}(t)}+\rho_f \left\langle \tfrac{v^{(\tau)} \circ \Phi^{(\tau)}-w_{f,k-1}^{(\tau)}}{h}, \xi\circ\Phi^{(\tau)} \right\rangle_{\Omega_{0}} \\ - \rho_f\langle \overline f^{(\tau)},\xi_\delta\rangle_{\Omega^{(\tau)}(t)}\,dt =0.
\end{aligned}
\end{equation}
Now we pass to the limit in this equation. It is possible since we have
the following convergences
\begin{align*}
\eta^{(\tau)} &\overset\ast\rightharpoonup \eta &\quad  \text{in } &L^\infty((0,h);W^{k_0+2,2}(Q;\mathbb R^d)), \\
\partial_t\tilde\eta^{(\tau)} &\rightharpoonup \partial_t\eta &\quad  \text{in } &L^2((0,h);W^{k_0+2,2}(Q;\mathbb R^d)), \\
\eta^{(\tau)}&\to \eta &\quad \text{in } &C^{(1/2)^-}([0,h];C^{1,\alpha^-}(Q;\mathbb R^d)),\\
v^{(\tau)}&\overset\eta\rightharpoonup v &\quad \text{in } &L^2((0,h);W^{k_0,2}(\Omega(t);\mathbb R^d)).
\intertext{The passage to the limit in the non linearity is not a
problem, as thanks to the compact embedding
\(W^{k_0+2,2}\hookrightarrow\hookrightarrow W^{2,q}\) we have}
\eta^{(\tau)}&\to\eta \quad &\text{in }&L^2((0,h);W^{2,q}(Q;\mathbb R^d)),
\intertext{and thus by \eqref{as:E-differentiable}}
DE(\eta^{(\tau)})&\rightharpoonup DE(\eta) \quad&\text{in }&L^2((0,h);W^{2,q}(Q;\mathbb R^d)^*),
\intertext{and also by \eqref{as:R-differentiable}}
D_2R(\underline\eta^{(\tau)},\partial_t\tilde\eta^{(\tau)}) &\rightharpoonup D_2R(\eta,\partial_t\eta) \quad&\text{in }&L^2((0,h);W^{1,2}(Q;\mathbb R^d)).
\end{align*}

Thus one can pass to the limit \(\tau\to0\) and the limit equation in the statement
is satisfied with $\xi_\delta$. Since the functions \(\xi_\delta \to \xi\) in
\(L^2((0,T);W^{k_0,2}(\Omega(t);\mathbb R^d)\) from Proposition \ref{prop:approximation-coupled-test-functions}, we
obtain equation in the statement for $\xi$.
\end{proof}

\begin{proposition}[Energy inequality]
The solution satisfies the energy inequality
\begin{equation}\label{eqn:energy-inequality-taulimit}
\begin{aligned} 
&E(\eta(h))+2\kappa^{a_0}\|\nabla^{k_0+2}\eta(h)\|^2+ \frac{1}{2h} \int_0^h \rho_f\|v\|^2_{\Omega(t)}+\rho_s\| \partial_t\eta \|^2 \,dt 
\\ 
&+ \int_0^h 2R(\eta,\partial_t\eta)  + 2\kappa\|\nabla^{k_0+2}\partial_t\eta\|^2 + \kappa\|\nabla^{k_0} v\|^2 + a \|\partial_t \eta \circ \eta^{-1} -v\|^2_{\partial \Omega(t)} \,dt
\\
& \leq E(\eta(0))+2\kappa^{a_0}\|\nabla^{k_0+2}\eta(0)\|^2+ \frac{1}{2h} \int_0^h\rho_f\|w_f\|^2_{\Omega_0}+\rho_s\| w_s \|^2 \,dt +\int_0^t \rho_s\left\langle f\circ\eta,\partial_t\eta \right\rangle
+ \rho_f\langle f,v\rangle_{\Omega(t)}\,dt.
\end{aligned}
\end{equation}
\end{proposition}

\begin{proof}
 Note that (in contrast to \cite[Lemma 4.8]{benesovaVariationalApproachHyperbolic2023}) we derive the energy inequality already here on the discrete $\tau$ level. Here we rely on our findings from \cite{cesikStabilityConvergenceTime2023}. 
Let $\tilde v_k^{(\tau)}\in W^{k_0,2}(\Omega_{k-1}^{(\tau)})$ be an extension of the boundary jump $\left.\left(v_k^{(\tau)}-\tfrac{\eta_k^{(\tau)}-\eta_{k-1}^{(\tau)}}{\tau}\circ(\eta_k^{(\tau)})^{-1} \right)\right|_{\partial\Omega_{k-1}^{(\tau)}}$ into the fluid domain. Then we use the test function pair $\left( \tfrac{\eta_k^{(\tau)}-\eta_{k-1}^{(\tau)}}{\tau}, v_k^{(\tau)}-\tilde v_k^{(\tau)}\right)$ which satisfies the required coupling in \eqref{eqn:coupled-eqn-tau} in \eqref{eqn:coupled-eqn-tau}; and use the test function $\tilde v_k^{(\tau)}$ in \eqref{eqn:fluid-only-eqn-tau}. Adding these two gives $$\begin{aligned}
\left(DE( \eta_k^{(\tau)})+D_2 R\left(\eta_{k-1}^{(\tau)}, \tfrac{\eta_k^{(\tau)}-\eta_{k-1}^{(\tau)}}{\tau}\right) \right)\left\langle \tfrac{\eta_k^{(\tau)}-\eta_{k-1}^{(\tau)}}{\tau} \right\rangle  + 2\kappa^{a_0}\left\langle \nabla^{k_0+2}\eta_k^{(\tau)}, \nabla^{k_0+2}\tfrac{\eta_k^{(\tau)}-\eta_{k-1}^{(\tau)}}{\tau} \right\rangle
\\+ 2\kappa\left\| \nabla^{k_0+2}\tfrac{\eta_k^{(\tau)}-\eta_{k-1}^{(\tau)}}{\tau} \right\|^2
+\rho_s\left\langle\tfrac{\tfrac{\eta_k^{(\tau)}-\eta_{k-1}^{(\tau)}}{\tau}-w_{s,k-1}^{(\tau)}}{h}, \tfrac{\eta_k^{(\tau)}-\eta_{k-1}^{(\tau)}}{\tau}\right\rangle- \rho_s\left\langle f_k^{(\tau)}\circ\eta_{k-1}^{(\tau)},\tfrac{\eta_k^{(\tau)}-\eta_{k-1}^{(\tau)}}{\tau} \right\rangle
\\
+\nu \|\varepsilon v_k^{(\tau)}\|_{\Omega_{k-1}^{(\tau)}}^2 \!+ \kappa\|\nabla^{k_0} v_k^{(\tau)}\|_{\Omega_{k-1}^{(\tau)}}^2\!+\rho_f \left\langle \tfrac{v_k^{(\tau)} \circ \Phi_{k-1}^{(\tau)}-w_{f,k-1}^{(\tau)}}{h}, v_k^{(\tau)} \circ\Phi_{k-1}^{(\tau)}\right\rangle_{\Omega_{0}} \hspace{-1em} 
\\
+a \left\|\tfrac{\eta_k^{(\tau)}-\eta_{k-1}^{(\tau)}}{\tau}\circ(\eta_k^{(\tau)})^{-1}  -v_k^{(\tau)}\right\|^2_{\partial\Omega_{k-1}^{(\tau)}} - \rho_f\langle f_k^{(\tau)},v_k^{(\tau)}\rangle_{\Omega_{k-1}^{(\tau)}}=0.
\end{aligned}$$
Here we use the non convexity estimate  \eqref{as:E-nonconvexity-estimate} $$DE( \eta_k^{(\tau)})\left\langle\tfrac{\eta_k^{(\tau)}-\eta_{k-1}^{(\tau)}}{\tau}\right\rangle\geq\frac1\tau(E(\eta_k^{(\tau)})-E(\eta_{k-1}^{(\tau)})-C_1\|\nabla(\eta_k^{(\tau)}-\eta_{k-1}^{(\tau)})\|^2),$$
two-homogeneity of R \eqref{as:R-two-homogeneous} $$D_2 R\left(\eta_{k-1}^{(\tau)}, \tfrac{\eta_k^{(\tau)}-\eta_{k-1}^{(\tau)}}{\tau}\right) \left\langle \tfrac{\eta_k^{(\tau)}-\eta_{k-1}^{(\tau)}}{\tau} \right\rangle =2 R\left(\eta_{k-1}^{(\tau)}, \tfrac{\eta_k^{(\tau)}-\eta_{k-1}^{(\tau)}}{\tau}\right), $$
we expand the term
$$2\kappa^{a_0}\left\langle \nabla^{k_0+2}\eta_k^{(\tau)}, \nabla^{k_0+2}\tfrac{\eta_k^{(\tau)}-\eta_{k-1}^{(\tau)}}{\tau} \right\rangle = \frac{\kappa^{a_0}}{\tau}\left(  \|\nabla^{k_0+2}\eta_k^{(\tau)}\|^2 + \|\nabla^{k_0+2}(\eta_k^{(\tau)}-\eta_{k-1}^{(\tau)})\|^2 -\|\nabla^{k_0+2}\eta_{k-1}^{(\tau)}\|^2\right)$$
and use Young's inequality in the inertial terms $$\begin{aligned}
\rho_s\left\langle\tfrac{\tfrac{\eta_k^{(\tau)}-\eta_{k-1}^{(\tau)}}{\tau}-w_{s,k-1}^{(\tau)}}{h}, \tfrac{\eta_k^{(\tau)}-\eta_{k-1}^{(\tau)}}{\tau}\right\rangle &\geq  \frac{\rho_s}{2h} \left\|\tfrac{\eta_k^{(\tau)}-\eta_{k-1}^{(\tau)}}{\tau}\right\|^2 - \frac{\rho_s}{2h}\|w_{s,k-1}^{(\tau)}\|^2,\\
\rho_f \left\langle \tfrac{v_k^{(\tau)} \circ \Phi_{k-1}^{(\tau)}-w_{f,k-1}^{(\tau)}}{h}, v_k^{(\tau)} \circ\Phi_{k-1}^{(\tau)}\right\rangle_{\Omega_{0}} &\geq \frac{\rho_f}{2h}\|v_k^{(\tau)}\circ\Phi_{k-1}^{(\tau)}\|_{\Omega_0}^2 - \frac{\rho_f}{2h}\|w_{f,k-1}^{(\tau)}\|_{\Omega_0}^2.
\end{aligned}$$
Altogether, multiplying by $\tau$ and using these estimates (the error from non convexity estimate is absorbed by dissipation for $\tau$ small enough) we obtain the estimate
$$\begin{aligned} 
E(\eta_k^{(\tau)})&+ 2\kappa^{a_0}\| \nabla^{k_0+2}\eta_k^{(\tau)}\|^2 + \tau 2 R\left(\eta_{k-1}^{(\tau)}, \tfrac{\eta_k^{(\tau)}-\eta_{k-1}^{(\tau)}}{\tau}\right)
+ \tau 2\kappa\left\| \nabla^{k_0+2}\tfrac{\eta_k^{(\tau)}-\eta_{k-1}^{(\tau)}}{\tau} \right\|^2+\tau\frac{\rho_s}{2h}\left\|\tfrac{\eta_k^{(\tau)}-\eta_{k-1}^{(\tau)}}{\tau}\right\|^2 
\\ 
&+\tau\kappa\|\nabla^{k_0} v_k^{(\tau)}\|_{\Omega_{k-1}^{(\tau)}}^2
+\tau\nu \| \varepsilon v_k^{(\tau)}\|_{\Omega_{k-1}^{(\tau)}}^2 
+\tau\frac{\rho_f}{2h} \left\| v_k^{(\tau)} \circ \Phi_{k-1}^{(\tau)}\right\|_{\Omega_{0}}^2
+a \left\|\tfrac{\eta_k^{(\tau)}-\eta_{k-1}^{(\tau)}}{\tau}\circ(\eta_k^{(\tau)})^{-1}  -v_k^{(\tau)}\right\|^2_{\partial \Omega_{k-1}^{(\tau)}}
\\ 
&\leq E(\eta_{k-1}^{(\tau)}) + 2\kappa^{a_0}\| \nabla^{k_0+2}\eta_{k-1}^{(\tau)}\|^2 + \tau\frac{\rho_s}{2h} \|w_{s,k-1}^{(\tau)}\|^2 +\tau\frac{\rho_f}{2h}\|w_{f,k-1}^{(\tau)}\|_{\Omega_0}^2
\\
&\quad\ + \tau\rho_s\left\langle f_k^{(\tau)}\circ\eta_{k-1}^{(\tau)},\tfrac{\eta_k^{(\tau)}-\eta_{k-1}^{(\tau)}}{\tau} \right\rangle +\tau\rho_f\langle f_k^{(\tau)},v_k^{(\tau)}\rangle_{\Omega_{k-1}^{(\tau)}}.
\end{aligned}$$
Summing over $k$, this yields the energy inequality
$$\begin{aligned} 
E(\eta^{(\tau)}(h))&+2\kappa^{a_0}\|\nabla^{k_0+2}\eta^{(\tau)}(h)\|^2+ \frac{1}{2h} \int_0^h \rho_f\|v^{(\tau)}\|^2_{\Omega(t)}+\rho_s\| \partial_t\tilde\eta^{(\tau)} \|^2 \,dt + \int_0^h 2R(\underline\eta^{(\tau)},\partial_t\tilde\eta^{(\tau)})
\\
&  + 2\kappa\|\nabla^{k_0+2}\partial_t\tilde\eta^{(\tau)}\|^2 
+ \kappa\|\nabla^{k_0} v^{(\tau)}\|^2 \,dt
+a \left\|\partial_t\tilde\eta^{(\tau)}\circ(\eta^{(\tau)})^{-1}  -v^{(\tau)}\right\|^2_{\partial \Omega^{(\tau)}(t)}
\\
&\leq E(\eta^{(\tau)}(0))+2\kappa^{a_0}\|\nabla^{k_0+2}\eta^{(\tau)}(0)\|^2+ \frac{1}{2h} \int_0^h\rho_f\|w_f\|^2_{\Omega_0}+\rho_s\| w_s \|^2 \,dt 
\\&\quad\ +\int_0^t \rho_s\left\langle f\circ\eta^{(\tau)},\partial_t\tilde\eta^{(\tau)} \right\rangle
+ \rho_f\langle f,v^{(\tau)}\rangle_{\Omega(t)}\,dt
\end{aligned}$$
which passes to the limit $\tau\to 0$.
\end{proof}

\subsection{\texorpdfstring{Passing to limit with the delay \(h\to 0\)}{Passing to limit with the delay h\textbackslash to 0}}\label{passing-to-limit-with-the-delay-hto-0}

We construct the solution on intervals \((0,h)\), \((h,2h)\), \(\dots\), \((T-h,T)\), as in the previous section and glue this together to get a time-delayed solution on the entire time interval $(0,T)$.

More precisely, for \(\ell=0,\dots,T/h - 1\) we let
\(\eta^{(h)}(\cdot +\ell h)\), \(v^{(h)}(\cdot + \ell h)\) and
\(\Omega^{(h)}(\cdot+\ell h)\), \(\Phi^{(h)}(\cdot + \ell h)\) to be
inductively the solution of the time-delayed equation of Definition \ref{def:time-delayed-solution} constructed
in the previous section, where we put first for \(\ell=0\)
\[w_s(t)\equiv \eta_*,\quad w_f(t)\equiv v_0, \quad \eta^{(h)}(t)=\eta_0, \quad t\in[0,h).\]
and for the subsequent steps \(\ell>0\)
\[\begin{aligned}w_s(t)=\partial_t\eta^{(h)}(t+(\ell-1)h),\quad w_f(t)=v^{(h)}(t +(\ell-1)h ) \circ \Phi^{(h)}(t+(\ell-1)h), \quad t\in [0,h),\\
\Omega_0=\Omega^{(h)}(\ell h ),\quad  \eta_0=\eta^{(h)}(\ell h ),
\end{aligned}\] where all these quantities are given by the solution in
the \((\ell-1)\)-th step.
Summing the time-delayed equations over \(\ell=0,\dots,T/h-1\), we have
the following equations.

\noindent \emph{Fluid-only equation.} It holds
\begin{equation}\label{eqn:fluid-only-eqn-h}
\begin{aligned}
\int_0^T \nu \langle\varepsilon v^{(h)}, \varepsilon \xi \rangle_{\Omega^{(h)}(t)} + \kappa\langle\nabla^{k_0} v^{(h)}, \nabla^{k_0}\xi\rangle_{\Omega^{(h)}(t)}+\rho_f \left\langle \tfrac{v^{(h)} \circ \Phi^{(h)}_t-v^{(h)}(\cdot-h)\circ\Phi^{(h)}_{t-h}}{h}, \xi \circ\Phi^{(h)} \right\rangle_{\Omega_{0}} \\ - \langle \partial_t \eta^{(h)} \circ (\eta^{(h)})^{-1}-v , \xi \rangle- \langle f,\xi \rangle_{\Omega^{(h)}(t)}\,dt=0
\end{aligned}
\end{equation}
for all
\(\xi\in C^\infty([0,T] \times \overline \Omega^{(h)}(t));\mathbb R^d\),
\(\xi\cdot n^{(h)}=0\) on \(\partial \Omega^{(h)}(t)\),
\(\operatorname{div}\xi=0\) in \(\Omega^{(h)}(t)\), with \(\xi(T)=0\).

\noindent \emph{Coupled equation.} It holds
\begin{equation}\label{eqn:coupled-eqn-h}
\begin{aligned} 
\int_0^T DE(\eta^{(h)})\langle\phi\rangle+D_2 R\left(\eta^{(h)}, \partial_t\eta^{(h)} \right) \langle \phi \rangle  + 2\kappa^{a_0}\langle \nabla^{k_0+2}\eta^{(h)}, \nabla^{k_0+2}\phi \rangle \\ + 2\kappa\left\langle \nabla^{k_0+2}\partial_t\eta^{(h)} , \nabla^{k_0+2}\phi \right\rangle
+\rho_s\left\langle\tfrac{\partial_t\eta^{(h)}- \partial_t\eta^{(h)}(\cdot-h)}{h}, \phi\right\rangle-\left\langle f\circ\eta^{(h)},\phi \right\rangle
\\
+\nu \langle\varepsilon v^{(h)}, \varepsilon \xi\rangle_{\Omega^{(h)}(t)} + \kappa\langle\nabla^{k_0} v^{(h)}, \nabla^{k_0}\xi\rangle_{\Omega^{(h)}(t)}\\+\rho_f \left\langle \tfrac{v^{(h)} \circ \Phi^{(h)}_t - v^{(h)}(t-h)\circ\Phi^{(h)}_{t-h}}{h}, \xi\circ\Phi^{(h)}_t \right\rangle_{\Omega_{0}} - \langle f,\xi\rangle_{\Omega^{(h)}(t)}\,dt=0
\end{aligned}
\end{equation}

for all
\(\xi\in C^\infty([0,T]\times \overline \Omega^{(h)}(t);\mathbb R^d)\),
\(\operatorname {div} \xi =0\) in \(\Omega^{(h)}(t)\) and
\(\phi\in W^{k_0+2,2}(Q;\mathbb R^d)\) with \(\phi = \xi\circ\eta\) on
\(Q\).

\subsubsection{Estimates and a weak limit}\label{estimates-and-a-weak-limit}

By the energy inequality \eqref{eqn:energy-inequality-taulimit} we have the following estimate

\begin{equation}\label{eqn:energy-inequality-h}
\begin{aligned} 
E(\eta^{(h)}(t))+2\kappa^{a_0}\|\nabla^{k_0+2}\eta^{(h)}(t)\|^2+ \frac{1}{2h} \int_{t-h}^t\rho_f\|v^{(h)}\|^2_{\Omega(t)}+\rho_s\| \partial_t\eta^{(h)} \|^2 \,dt \\ + \int_0^t 2R(\eta^{(h)},\partial_t\eta^{(h)})  + 2\kappa\|\nabla^{k_0+2}\partial_t\eta^{(h)}\|^2 + \|\nabla^{k_0}v^{(h)}\|^2+ a \|\partial_t \eta^{(h)} \circ (\eta^{(h)})^{-1} - v\|_{\partial \Omega(t)}^2 \,dt
\\
\leq E(\eta^{(h)}(0))+2\kappa^{a_0}\|\nabla^{k_0+2}\eta^{(h)}(0)\|^2+ \frac{1}{2h} \int_{t-h}^t\rho_f\|w_f\|^2_{\Omega_0}+\rho_s\| w_s \|^2 \,dt \\+\int_0^t\rho_s\left\langle f\circ\eta^{(h)},\partial_t\eta \right\rangle
+ \rho_f\langle f,v^{(h)}\rangle_{\Omega(t)}\,dt
\end{aligned}
\end{equation}

and thus for (a subsequence of) \(h\to0\) we have weakly
converging subsequences

\begin{equation}\label{eqn:weak-limit-h}
\begin{aligned}
\eta^{(h)}\overset\ast\rightharpoonup \eta &\quad\text{in }L^\infty((0,T);W^{k_0+2,2}(Q;\mathbb R^d)), \\
\partial_t\eta^{(h)}\rightharpoonup\partial_t\eta &\quad \text{in }L^2((0,T);W^{k_0+2,2}(Q;\mathbb R^d)), \\
v^{(h)}\overset\eta\rightharpoonup v &\quad \text{in }L^2((0,T);W^{k_0,2}(\Omega(t);\mathbb R^d)).
\end{aligned}
\end{equation}

In order to pass to the limit in the inertial terms, we take an approximation of test functions for \eqref{eqn:fluid-only-eqn-h} as in Proposition \ref{prop:approximation-coupled-test-functions}
(resp. for \eqref{eqn:coupled-eqn-h} as in Proposition \ref{prop:approximation-fluid-only} (ii)) and pass to the limits \(h\to0\) in all of the terms
in fluid-only equation and in the coupled equation in analogy to Proposition \ref{prop:fluid-only-eqn-taulimit} and \ref{prop:coupled-equation-taulimit},
except for the inertial terms that we shall deal with below.

For this we shall need an estimate on discrete versions of \(\partial_{tt}\eta^{(h)}\) and $\partial_t v$. In particular, from the equation we get the estimate on $h$-difference quotients as follows.

\begin{lemma}[Solid bounds with length $h$]\label{lem:solid-bounds-length-h}
The following bound exists independent of $h$:
\begin{equation}\label{eqn:bound-dtt-eta,w-k02}
\int_0^T \left\|\tfrac{\partial_t\eta^{(h)}-\partial_t\eta^{(h)}(t-h)}{h} \right\|_{W^{-k_0-2,2}(Q;\mathbb R^d)}^2\,dt \leq C.
\end{equation}

\end{lemma}

\begin{proof}
To show \eqref{eqn:bound-dtt-eta,w-k02} we use a test function $\phi\in W^{k_0+2,2}_0(Q;\mathbb R^d)$ in the equation \eqref{eqn:coupled-eqn-h} (the corresponding $\xi$ is zero in the fluid domain), so we get 
$$
\begin{aligned}
\rho_s&\left|\left\langle\tfrac{\partial_t \eta^{(h)}(t)-\partial_t \eta^{(h)}(t-h)}{h}, \phi\right\rangle \right| \leq\left|\left\langle D E\left(\eta^{(h)}(t)\right), \phi\right\rangle\right|+\kappa^{a_0}\left|\left\langle\nabla^{k_0+2} \eta^{(h)}, \nabla^{k_0+2} \phi\right\rangle\right| \\
&+\left|\left\langle D_2 R\left(\eta^{(h)}(t), \partial_t \eta^{(h)}(t)\right), \phi\right\rangle\right|+\kappa\left|\left\langle\nabla^{k_0+2} \partial_t \eta^{(h)}, \nabla^{k_0+2} \phi\right\rangle\right|+\left|\langle f_s(t), \phi\rangle\right| \\
& \leq\left(\left\|D E\left(\eta^{(h)}(t)\right)\right\|_{W^{-2, q'}}+\kappa^{a_0}\left\|\nabla^{k_0+2} \eta^{(h)}(t)\right\| +\|f_s\|_{\infty}\right)\|\phi\|_{W^{k_0+2, 2}} \\
& \quad\ +\left(\left\|D_2 R\left(\eta^{(h)}(t), \partial_t \eta^{(h)}(t)\right)\right\|_{W^{-1,2}}+\kappa\left\|\nabla^{k_0} \partial_t \eta^{(h)}(t)\right\|\right)\|\phi\|_{W^{k_0+2, 2}} \leq c(t) \|\phi\|_{W^{k_0+2, 2}},
\end{aligned}
$$
with the bound $c\in L^2((0,T))$ thanks to \eqref{eqn:energy-inequality-h}, after integration in time we get the estimate \eqref{eqn:bound-dtt-eta,w-k02}.
\end{proof}

For the solid, convergence of the inertial term is not difficult to prove using the preceding estimate.
Let \(b^{(h)}(t)=\frac1h\int_{t-h}^h \partial_t\eta^{(h)}\). By \eqref{eqn:energy-inequality-h} it is uniformly bounded in \(L^2((0,T);W^{k_0+2,2}(Q;\mathbb R^d))\) Since
we have
\[\partial_t b ^{(h)}= \tfrac{\partial_t\eta^{(h)}(\cdot)-\partial_t\eta^{(h)}(\cdot - h)}{h}\]
we see that by Lemma \ref{lem:solid-bounds-length-h}   \(\partial_t b^{(h)}\) is uniformly bounded in
\(L^{2}((0,T);W^{-k_0-2,2}(Q;\mathbb R^d))\)
 which using the Aubin-Lions lemma immediately yields
the convergence (for a subsequence)
\[b^{(h)}\to \partial_t\eta \quad\text{ in }C([0,T]; L^2(Q;\mathbb R^d)).\]

The convergence of the fluid inertial term is a more delicate matter, for we first need to prove use some additional flow map estimates, see also \cite[Lemma 4.14]{benesovaVariationalApproachHyperbolic2023}.

\begin{proposition}[Flow map $h$-estimates]\label{prop:flow-map-h-estimate}
It holds 
\begin{equation}
\left\|\tfrac{\xi(t+h)\circ\Phi^{(h)}_h-\xi(t)}{h}\right\|_{L^\infty L^2}\leq C
\end{equation}
for all \(\xi\in C^\infty([0,T] \times \overline \Omega^{(h)}(t);\mathbb R^d)\),
\(\xi\cdot n^{(h)}=0\) on \(\partial \Omega^{(h)}(t)\),
\(\operatorname{div}\xi=0\) in \(\Omega^{(h)}(t)\).
\end{proposition}
%\noteMalte[inline]{We should be more consistent with how we write norms. This one also might want to reference some predecessor in the original paper.}
\begin{proof}
Remember that 
\(\partial_s \Phi^{(h)}_s=u(s)\circ \Phi^{(h)}_s\) from \eqref{eqn:flow-map-dt-equal-v}. We rewrite
\begin{equation}\label{eqn:flow-map-integration}
\xi(t+h)\circ\Phi^{(h)}_h-\xi(t)=\int_0^h \partial_s (\xi(t+s)\circ\Phi^{(h)}_s))\,ds=\int_0^h \partial_t \xi(t+s)\circ\Phi^{(h)}_s + \nabla\xi(t+s)\circ \Phi^{(h)}_s\cdot u^{(h)}(s)\,ds.
\end{equation}
 We have   
$$\left\|\tfrac{\xi(t+h)\circ\Phi^{(h)}_h-\xi(t)}{h}\right\|_{L^\infty L^2}^2\leq C \operatorname{Lip}_t\xi +C \operatorname{Lip}_x\xi\sup_t\frac1h\int_t^{t+h}|v^{(h)}|^2\,dt\leq C$$
and we conclude by the \(L^\infty\) estimate of \(h\)-average of $v^{(h)}$ from \eqref{eqn:energy-inequality-h}.
\end{proof}

Now we shall estimate the inertial term. We now work with the \emph{global velocity field} $u^{(h)}$ defined by
\begin{equation}\label{eqn:global-velocity-field}
u^{(h)}(t,y)= \begin{cases}
v^{(h)}(t,y), & y\in \Omega^{(h)}(t) \\
\eta^{(h)}(t,(\eta^{(h)})^{-1}(t,y)), & y\in \eta^{(h)}(t,Q).
\end{cases}
\end{equation}
Note that $u^{(h)}$ has a tangential jump along $\partial\eta(\cdot,Q)$. For this global velocity field, we have the following estimate of its time derivative.

\begin{lemma}[Fluid bounds with length $h$]\label{lem:fluid-bounds-length-h}
Let $u^{(h)}$ be the global velocity field \eqref{eqn:global-velocity-field}. Then the for $m$ large enough we have the estimate
\[\int_0^T \left|\left\langle \tfrac{u^{(h)}(t)-u^{(h)}(t-h)}{h} , \xi(t)\right\rangle_{\Omega^{(h)}(t)}\right|\,dt \leq C\|\xi\|_{L^2((0,T);W^{m,2}(\Omega^{(h)}(t);\mathbb R^d))}\] for
every
\(\xi\in C^\infty([0,T] \times \overline \Omega^{(h)}(t);\mathbb R^d)\),
\(\nabla^\ell(\xi\cdot n^{(h)})=0\) on \(\partial \Omega^{(h)}(t)\), $\ell=0,\dots,k_0$,
\(\operatorname{div}\xi=0\) in \(\Omega^{(h)}(t)\).
\end{lemma}

\begin{proof}
First we put in the flow map as follows 
\[\begin{aligned}
\int_0^T \left\langle \tfrac{u^{(h)}(t)-u^{(h)}(t-h)}{h} , \xi(t)\right\rangle_{\Omega^{(h)}(t)} \,dt= \int_0^T \left\langle \tfrac{u^{(h)}(t)-u^{(h)}(t-h)\circ\Phi^{(h)}_{-h}(t)}{h} , \xi(t)\right\rangle_{\Omega^{(h)}(t)} \\+\left\langle \tfrac{u^{(h)}(t-h)\circ \Phi^{(h)}_{-h}(t)-u^{(h)}(t-h)}{h} , \xi(t)\right\rangle_{\Omega^{(h)}(t)}\,dt.
\end{aligned}\]

We estimate by using the test
function \(\xi\) in the fluid only equation \eqref{eqn:fluid-only-eqn-h}. Then we estimate
\begin{equation}\label{eqn:u-dual-estimate}
\begin{aligned}
\int_0^T &\left\langle \tfrac{u^{(h)}(t)-u^{(h)}(t-h)\circ\Phi^{(h)}_{-h}(t)}{h} , \xi(t)\right\rangle_{\Omega^{(h)}(t)}\,dt  
\\&=
-\int_0^T \nu \langle\varepsilon u^{(h)}, \varepsilon \xi \rangle_{\Omega^{(h)}(t)} + \kappa\langle\nabla^{k_0} u^{(h)}, \nabla^{k_0}\xi\rangle_{\Omega^{(h)}(t)} - \langle f,\xi \rangle_{\Omega^{(h)}(t)}\,dt \\
&\leq C \|\xi\|_{L^\infty(C^{k_0})}
\end{aligned}
\end{equation}
and choose $m$ so large that $W^{m,2}$ embeds into $C^{k_0}$.
\end{proof}

\subsubsection{\texorpdfstring{Limit \(h\to 0\) in the fluid inertial term}{Limit h\textbackslash to 0 in the fluid inertial term}}\label{limit-hto-0-in-the-fluid-inertial-term}

We now want to pass to the limit \(h\to 0\) in
\[\int_0^T\rho_f\left\langle \tfrac{u^{(h)}(t)\circ\Phi^{(h)}-u^{(h)}(t-h)}{h},\xi^{(h)}_\delta(t) \right\rangle \,dt.\]
We want to do this in both the fluid-only equation \eqref{eqn:fluid-only-eqn-h}, and also the
coupled equation \eqref{eqn:coupled-eqn-h}. We
shall prove now a version of Aubin-Lions lemma for the fluid. %Below in the computations we shall omit the constant \(\rho_f\).
%\paragraph{Fluid Aubin-Lions}
Denote \[\widetilde m^{(h)}(t) = \frac1h\int_{t-h}^t \rho v^{(h)}\,dt\]
Our aim now is to show the convergence
\[\int_0^T\langle u^{(h_i)},A\widetilde m^{(h_i)}\rangle\,dt\to \int_0^T\langle u, A\widetilde m\rangle\,dt\]
where \(A\in C^\infty_0((0,T)\times\Omega)\) is a cutoff function chosen
that \(u^{(h)}\) and also \(u^{(h)}(\cdot +\sigma)\),
\(\sigma\in (0,h)\) is always well defined on \(\operatorname{supp} A\).

First we show the Fluid Aubin-Lions in the form that
\[\int_0^T \langle (u^{(h_i)})_\delta,A\widetilde m^{(h_i)}\rangle\,dt\to \int_0^T \langle u_\delta, A\widetilde m\rangle\,dt.\]
%where \(\delta\) is the divergence-free extension of \(u\) to
%\(\delta\)-neighborhood by Proposition \ref{prop:approximation-coupled-test-functions}. The reason is that we
%have the estimate \eqref{eqn:u-dual-estimate} only for divergence free test
%functions.
We carefully construct our approximation of $(\cdot)_\delta$ by first truncating the functions on a $\delta$ strip. For that we use the cutoff $\psi_\delta$ defined in Proposition~\ref{prop:approximation-fluid-only} and the uniform Bogovskii operator defined in Theorem~\ref{thm:bogovskii-close-deformations} and take
\[
(u^{(h_i)})_{\delta}^*:=(1-\psi_\delta)v^{(h_i)}-\mathcal{B}(\mathrm{div}((1-\psi_\delta)v^{(h_i)})).
\]
Please observe that since $v^{(h_i)}\in L^p((0,T)\times \Omega)$ for $p>2$ we find (using the $L^2$-bound of the Bogovskii) that for $1<q\leq p$
\[
\|(u^{(h_i)})_{\delta}^{*}-u^{(h_i)}\|_{L^q((0,T)\times \Omega)}\leq C\delta^{\frac{1}{q}-\frac{1}{p}}\|(u^{(h_i)})_{\delta}^{*}-u^{(h_i)}\|_{L^2((0,T)\times \Omega)}\leq C.
\]
Note further that $(u^{(h_i)})_{\delta}^{*}$ is uniformly in $L^2(0,T;W^{1,2}_0(\Omega^{(h_j)}(t))$ for $i,j$ large enough with bounds depending on $\delta$. We take the standard mollifier $\phi_{\frac{\delta}{2}}$ and define
\begin{equation}\label{eqn:uh-def}
(u^{(h_i)})_{\delta}:=(u^{(h_i)})_{\delta}^{*}*\phi_{\frac{\delta}{2}}.
\end{equation}
Analogous to the proof of Proposition~\ref{prop:approximation-coupled-test-functions} one finds the respective bounds and convergence properties in dependence of $\delta$. Moreover 
\[
\|(u^{(h_i)})_{\delta}-u^{(h_i)}\|_{L^q((0,T)\times \Omega)}\leq C\delta^{\frac{1}{q}-\frac{1}{p}}.
\]
%We shall decompose the ($\delta$-approximated) global velocity field $(u^{(h)})_\delta$ defined in \eqref{eqn:global-velocity-field} into two components,
%reminiscent of the distinction for fluid-only and coupled equations.
%
%First, the component which is continuous in the entire \(\Omega\) will be defined by
%\begin{equation}\label{eqn:uc-def}
%(u_c^{(h)})_\delta = \mathcal E_\delta(\partial_t\eta^{(h)}\circ(\eta^{(h)})^{-1})
%\end{equation}
%where
%\(\mathcal E_\delta(t)\colon W^{k_0,2}(\Omega^{(h)}(t);\mathbb R^d)\to W^{k_0,2}(\Omega;\mathbb R^d)\)
%is a divergence-free extension operator such that also it is continuous as
%\(\mathcal E_h\colon L^2((0,T);W^{k_0,2}(\Omega^{(h)}(t);\mathbb R^d))\to L^2((0,T);W^{k_0,2}(\Omega;\mathbb R^d))\).
%The regularity of \(\eta^{(h)}\) is sufficient for such construction.
%
%The second, fluid-only part of $u^{(h)}_\delta$ is defined as
%\begin{equation}\label{eqn:uf-def}
%(u_f^{(h)})_\delta= \tilde {\mathcal E}_\delta( v^{(h)}-u_c^{(h)})
%\end{equation}
% where
%\(\tilde{\mathcal E}_\delta(t)\) is an extension from \(\Omega^{(h)}(t)\) to
%a neighborhood of it, and bounded in \(W^{k_0,2}\) norm.
%
%By linearity of $\mathcal E$ and of $(\cdot)_\delta$ we see that $(u^{(h)})_\delta = (u_c^{(h)})_\delta+(u_f^{(h)})_\delta$.

Thus we formulate the fluid Aubin Lions lemma in the following way.

\begin{theorem}[Fluid Aubin-Lions]\label{thm:aubin-lions-fluid}
 For every \(\delta>0\) it holds
that
\[\int_0^T\langle (u^{(h_i)})_\delta,A\widetilde m^{(h_i)}\rangle\,dt\to \int_0^T\langle u_\delta, A\rho u\rangle\,dt ,\]
where \((\cdot)_\delta\) is the approximation defined in \eqref{eqn:uh-def}.
\end{theorem}

\begin{proof} Let \(\delta>0\). 
%As above we decompose $(u^{(h)})_\delta$ into\((u^{(h)}_c)_\delta\) and
%\((u^{(h)}_f)_\delta\) defined by \eqref{eqn:uc-def} and \eqref{eqn:uf-def}, and in subsequent estimates write \((u_\sim^{(h)} )_\delta\) for either \((u^{(h)}_c)_\delta\) or \((u^{(h)}_f)_\delta\). Note that both of these velocities are defined at least
%bit outside \(\Omega^{(h)}(t)\), and in particular that \(u_\sim^{(h)}(t-h)\)
%is always well-defined in \(\Omega^{(h)}(t)\). 
Observe that $(u^{(h_i)})_\delta$ is constructed so that the functions are divergence-free in a neighborhood in space-time. This means that $(u^{(h_i)})_\delta$ is a valid test functions in a small (but fixed) neighborhood of $t$ uniformly for $i$ large enough.

 To obtain the desired convergence, we
will show that
\(\int_0^T\langle (u^{(h_i)})_\delta,A\widetilde m^{(h_i)}\rangle\,dt\) is a
Cauchy sequence. For this we write
\[\langle (u^{(h_i)})_\delta,A\widetilde m^{(h_i)}\rangle-\langle (u^{(h_j)})_\delta,A\widetilde m^{(h_j)}\rangle = \langle (u^{(h_i)})_\delta,A\widetilde m^{(h_i)}-A\widetilde m^{(h_j)}\rangle + \langle (u^{(h_i)})_\delta-(u^{(h_j)})_\delta,A\widetilde m^{(h_j)}\rangle.\]
We now focus on the first term, namely
\[\langle (u^{(h_i)})_\delta,A\widetilde m^{(h_i)}-A\widetilde m^{(h_j)}\rangle.\]
%(the second term will work somewhat symmetrically  )

We partition the time with \(\sigma>0\) steps and replace the
\(A\widetilde m^{(h_i)}\) with piecewise constant, i.e.~we write
\begin{equation}\label{eqn:aubin-lions-insert-constant}
\begin{aligned}\langle (u^{(h_i)}(t))_\delta,A\widetilde m^{(h_i)}(t)-A\widetilde m^{(h_j)}(t)\rangle = \langle (u^{(h_i)}(t))_\delta,A\widetilde m^{(h_i)}(t)-A\widetilde m^{(h_i)}(\sigma k)\rangle \\ +\langle (u^{(h_i)}(t))_\delta,A\widetilde m^{(h_i)}(\sigma k)-A\widetilde m^{(h_j)}(\sigma k)\rangle + \langle (u^{(h_i)}(t))_\delta,A\widetilde m^{(h_j)}(\sigma k)-A\widetilde m^{(h_j)}(t)\rangle\end{aligned}
\end{equation}
Now we use that we have the uniform bound from Proposition \ref{lem:solid-bounds-length-h}
\[\|\widetilde m^{(h_i)}(\sigma k)\|_{L^2}^2\leq \frac1h\int_{\sigma k -h}^{\sigma k}\|\rho v (t)\|_{L^2}^2\,dt \leq C.\]
So by the compact embedding
\[L^2(\Omega) \subset\subset (W^{1,2}_{\operatorname{div}}(\Omega_\delta))^*\]
we have for a subsequence 
\[\widetilde m^{(h_i)}(\sigma k)\to \widetilde m(\sigma k) \quad \text{in }(W^{1,2}_{\operatorname{div}}(\Omega_\delta))^*,\]
so that, since \(\| (u^{(h_i)}(t))_\delta\|_{W^{1,2}(\Omega)}\leq C\) by the
apriori estimates \eqref{eqn:energy-inequality-h},
\[\| (u^{(h_i)}(t))_\delta\|_{W^{1,2}(\Omega)} \|A\widetilde m^{(h_i)}(\sigma k)-A\widetilde m^{(h_j)}(\sigma k)\|_{(W^{1,2}_{\operatorname{div}}(\Omega_\delta))^*}\to 0.\]

It remains in \eqref{eqn:aubin-lions-insert-constant} to estimate the term
\begin{equation}\label{eqn:aubin-lions-the-term}
\int_0^T\langle (u^{(h_i)}(t))_\delta,A\widetilde m^{(h_j)}(t)-A\widetilde m^{(h_j)}(\sigma k)\rangle\,dt.
\end{equation}
The similar term where towards the end $j$ is replaced by $i$ is dealt with in the same way.

To estimate this we replace the \(\sigma\)-difference
quotient with an \(h\)-difference quotient, with the aim to then use the
bound on \(h\)-difference quotient. So for this  we write
\[A\widetilde m^{(h_j)}(t)-A\widetilde m^{(h_j)}(\sigma k) = A\int_{\sigma k}^{t}\partial_\theta \widetilde m^{(h_j)}(\theta)\,d\theta\]
so that we get in \eqref{eqn:aubin-lions-the-term}
\[\langle (u^{(h_i)}(t))_\delta,A\widetilde m^{(h_j)}(t)-A\widetilde m^{(h_j)}(\sigma k)\rangle =\left\langle (u^{(h_i)}(t))_\delta,A\int_{\sigma k}^{t}\partial_\theta \widetilde m^{(h_j)}(\theta)\,d\theta\right\rangle.\]
Realize
that by definition of \(\widetilde m^{(h_j)}\) it is
\[\partial_\theta \widetilde m^{(h_j)}(\theta) = \tfrac{\rho u^{(h_j)}(\theta)-\rho u^{(h_j)}(\theta-h_j)}{h_j}.\]
So that now we have
\[\int_0^T \!\! \left\langle (u^{(h_i)}(t))_\delta,A\widetilde m^{(h_j)}(t)-A\widetilde m^{(h_j)}(\sigma k)\right\rangle  dt =\! \int_0^T \!\! \left\langle (u^{(h_i)}(t))_\delta,A\int_{\sigma k}^t \!\! \tfrac{\rho u^{(h_j)}(\theta)-\rho u^{(h_j)}(\theta-h_j)}{h_j} d\theta \right\rangle dt.\]
Now comes a switch of the order of integration which results in
\[\leq \|A\|_\infty \int_0^\sigma  \int_0^T \left|\left\langle (u^{(h_i)}(\theta+s))_\delta, \tfrac{\rho u^{(h_j)}(\theta)-\rho u^{(h_j)}(\theta-h_j)}{h_j} \right\rangle\right|\,ds\,d\theta\]
and we use the bound \eqref{eqn:u-dual-estimate}  to obtain, recall also the
estimate from Proposition \ref{prop:approximation-coupled-test-functions}
\[\leq \|A\|_\infty\int_0^\sigma \|(u^{(h_i)})_\delta(\theta+\cdot)\|_{L^2W^{k_0,2}_{\operatorname{div}}}\,d\theta \leq \|A\|_\infty C_\delta \sigma \|u^{(h_i)}\|_{L^2W^{1,2}}\]
so this vanishes for \(\sigma \to 0\), as \(\|u^{(h_i)}\|_{L^2W^{1,2}}\)
is bounded by the energy inequality \eqref{eqn:energy-inequality-h}.

This proves the
convergence\[\langle (u^{(h_i)})_\delta,A\widetilde m^{(h_i)}\rangle\to \langle u_\delta, A\widetilde m\rangle,\]
where \(\widetilde m^{(h_i)}\overset\ast\rightharpoonup\widetilde m\) in
\(L^\infty((0,T);L^2(\Omega;\mathbb R^d))\) as the limit exists by the estimate \eqref{eqn:energy-inequality-h}.
It is not difficult to check \(\widetilde m=\rho u\), since for
\(\xi\in C_0((0,T)\times \Omega;\mathbb R^d)\) we have
\[\int_0^T\int_\Omega \widetilde m^{(h_i)}\cdot \xi\,dx\,dt=\int_0^T\int_\Omega\rho u^{(h_i)}\cdot \frac1h\int_0^h \xi(t+s)\,ds\,dx\,dt \to\int_0^T\int_\Omega  \rho u\cdot \xi\,dx\,dt\]
which concludes the proof.
\end{proof}
Now we are equipped for the rest of the $h\to 0$ limit passage in the fluid inertial term.
\begin{theorem}[Limit passage $h\to 0$]\label{thm:solution-hlimit}
 The limit $\eta,h$ from \eqref{eqn:weak-limit-h} satisfies the following equations.
\noindent\emph{Fluid-only equation}. It holds
\begin{align}\label{eqn:fluid-only-eqn-hlimit}
\int_0^T \nu \langle\varepsilon v, \varepsilon \xi \rangle_{\Omega(t)} &+ \kappa\langle\nabla^{k_0} v, \nabla^{k_0}\xi\rangle_{\Omega(t)}-\rho_f \left\langle  \partial_t v,\xi\rangle_{\Omega(t)} + \rho_f\langle v , v\cdot \nabla \xi \right\rangle_{\Omega(t)} \nonumber\\
&- a\langle \partial_t \eta \circ \eta^{-1}-v,\xi \rangle - \langle f,\xi \rangle_{\Omega(t)}\,dt=0
\end{align}
for all
\(\xi\in C^\infty([0,T] \times \overline \Omega(t));\mathbb R^d)\),
\(\xi\cdot n=0\) on \(\partial \Omega(t)\), \(\operatorname{div}\xi=0\)
in \(\Omega(t)\), with \(\xi(T)=0\).

\noindent\emph{Coupled equation.} It holds
\begin{equation}\label{eqn:coupled-eqn-hlimit}
\begin{aligned} 
\int_0^T DE(\eta)\langle\phi\rangle+D_2 R\left(\eta, \partial_t\eta \right) \langle \phi \rangle  + 2\kappa^{a_0}\langle \nabla^{k_0+2}\eta, \nabla^{k_0+2}\phi \rangle + 2\kappa\left\langle \nabla^{k_0+2}\partial_t\eta , \nabla^{k_0+2}\phi \right\rangle
\\
-\rho_s\left\langle \partial_{t}\eta, \partial_t\phi\right\rangle-\left\langle f\circ\eta,\phi \right\rangle
\\
+\nu \langle\varepsilon v, \varepsilon \xi\rangle_{\Omega(t)} + \kappa\langle\nabla^{k_0} v, \nabla^{k_0}\xi\rangle_{\Omega(t)}-\rho_f \left\langle  \partial_t v,\xi\rangle_{\Omega(t)} + \rho_f\langle v , v\cdot \nabla \xi \right\rangle_{\Omega(t)} - \langle f,\xi\rangle_{\Omega(t)}\,dt=0
\end{aligned}
\end{equation}
for all
\(\xi\in C^\infty([0,T]\times \overline \Omega(t);\mathbb R^d)\),
\(\operatorname {div} \xi =0\) in \(\Omega(t)\) and
\(\phi\in W^{k_0+2,2}(Q;\mathbb R^d)\) with \(\phi = \xi\circ\eta\) on
\(Q\).

Further, it satisfies the following energy inequality 
\begin{equation}\label{eqn:energy-inequality-hlimit}
\begin{aligned} 
&E(\eta(t))+2\kappa^{a_0}\|\nabla^{k_0+2}\eta(t)\|^2+ \frac{\rho_f}{2} \|v(t)\|^2_{\Omega(t)}+ \frac{\rho_s}{2}\| \partial_t\eta(t) \|^2  
\\
&+ \int_0^t 2R(\eta,\partial_t\eta)  + 2\kappa\|\nabla^{k_0+2}\partial_t\eta\|^2 + \|\nabla^{k_0}v\|^2 + a \| \partial_t \eta \circ \eta^{-1} - v\|_{\partial \Omega(t)}^2 \,dt
\\
&\leq E(\eta(0))+2\kappa^{a_0}\|\nabla^{k_0+2}\eta(0)\|^2+ \frac{\rho_f}{2}\|v(0) \|^2_{\Omega_0}+\frac{\rho_s}{2}\| \partial_t\eta(0) \|^2 \,dt +\int_0^t\rho_s\left\langle f\circ\eta,\partial_t\eta \right\rangle
+ \rho_f\langle f,v\rangle_{\Omega(t)}\,dt.
\end{aligned}
\end{equation}
\end{theorem}

\begin{proof}
We are equipped to pass to the limit with the fluid inertial term, in
particular
\[\int_0^T \left\langle\tfrac{u^{(h)}(t)\circ\Phi_h^{(h)}(t-h)-u^{(h)}(t-h)}{h},\xi(t)\circ\Phi^{(h)}_h(t-h) \right\rangle_{\Omega^{(h)}(t-h)} \,dt.\]
After discrete partial integration in time we have
\[=-\int_0^T\left\langle u^{(h)}(t),\tfrac{\xi(t+h)\circ\Phi^{(h)}-\xi(t)}{h} \right\rangle_{\Omega^{(h)}(t)}\,dt.\]
We write only the term under the integral. Use the \(\delta\)-divergence free
approximation
\[\begin{aligned}&\left\langle u^{(h)}(t),\tfrac{\xi(t+h)\circ\Phi^{(h)}-\xi(t)}{h} \right\rangle_{\Omega^{(h)}(t)} \\& = \left\langle (u^{(h)}(t))_\delta-u^{(h)}(t),\tfrac{\xi(t+h)\circ\Phi^{(h)}-\xi(t)}{h} \right\rangle_{\Omega^{(h)}(t)}+\left\langle (u^{(h)}(t))_\delta,\tfrac{\xi(t+h)\circ\Phi^{(h)}-\xi(t)}{h} \right\rangle_{\Omega^{(h)}(t)} \end{aligned}.\]

In the first term, use
\[\leq\left\| (u^{(h)}(t))_\delta-u^{(h)}(t) \right\|_{L^1L^2}\left\|\tfrac{\xi(t+h)\circ\Phi^{(h)}-\xi(t)}{h} \right\|_{L^\infty L^2}.\]
The former is by the \(\delta\)-approximation of Proposition \ref{prop:approximation-coupled-test-functions} bounded by
\[
\left\| (u^{(h)}(t))_\delta-u^{(h)}(t) \right\|_{L^1L^2}\leq \delta^{\frac{d}{d+2}}\left\|u^{(h)}(t)\right\|_{L^2(W^{1,2})},
\]
and the latter is bounded by Proposition \ref{prop:flow-map-h-estimate}, 
so that in total for the first term
\[\left|\int_0^T\left\langle (u^{(h)}(t))_\delta-u^{(h)}(t),\tfrac{\xi(t+h)\circ\Phi^{(h)}-\xi(t)}{h} \right\rangle_{\Omega^{(h)}(t)}\,dt \right|\leq \delta^{\frac{d}{d+2}}\tilde C.\]
For the second term
\[\left\langle (u^{(h)}(t))_\delta,\tfrac{\xi(t+h)\circ\Phi^{(h)}-\xi(t)}{h} \right\rangle_{\Omega^{(h)}(t)}\]
we perform again the manipulations \eqref{eqn:flow-map-integration}
so that
\[\begin{aligned}&\left\langle (u^{(h)}(t))_\delta,\tfrac{\xi(t+h)\circ\Phi^{(h)}-\xi(t)}{h} \right\rangle_{\Omega^{(h)}(t)} \\&\qquad= \left\langle(u^{(h)}(t))_\delta,\frac1h\int_0^h \left(\partial_t \xi(t+s) + \nabla\xi(t+s)\cdot u^{(h)}(t+s)\right)\circ\Phi^{(h)}_s\,ds\right\rangle. \end{aligned}\]
This is now a sum where the first part is fine: after the limit
\(h\to 0\)
\[\int_0^T \left\langle(u^{(h)}(t))_\delta,\frac1h\int_0^h \partial_t \xi(t+s)\circ\Phi^{(h)}_s \,ds\right\rangle \,dt \to \int_0^T \left\langle (u(t))_\delta, \partial_t\xi(t) \right\rangle \,dt.\]
We now handle the second term:
\[\int_0^T\left\langle(u^{(h)}(t))_\delta,\frac1h\int_0^h \left(\nabla\xi(t+s)\cdot u^{(h)}(t+s)\right) \circ \Phi^{(h)}_s\,ds\right\rangle\,dt.\]
We change the domain to obtain
\[\begin{aligned}
\int_0^T&\left\langle(u^{(h)}(t))_\delta,\frac1h\int_0^h \left(\nabla\xi(t+s)\cdot u^{(h)}(t+s)\right) \circ \Phi^{(h)}_s\,ds\right\rangle\,dt
\\
&=\int_0^T\frac1h\int_0^T \left\langle(u^{(h)}(t))_\delta \circ\Phi^{(h)}_{-s}, \nabla\xi(t+s)\cdot u^{(h)}(t+s) \right\rangle\,ds\,dt \\
&=\int_0^T\frac1h\int_0^h \left\langle(u^{(h)}(t))_\delta \circ\Phi^{(h)}_{-s} - (u^{(h)}(t))_\delta ,\nabla\xi(t+s)\cdot u^{(h)}(t+s) \right\rangle\,ds\,dt \\
&\quad \ + \int_0^T\left\langle (u^{(h)}(t))_\delta ,\frac1h\int_0^h \nabla\xi(t+s)\cdot u^{(h)}(t+s) \,ds\right\rangle\,dt. \end{aligned}\]
Convergence in the first term follows from Proposition \ref{prop:flow-map-h-estimate}:
\[\left\|(u^{(h)}(t))_\delta \circ\Phi^{(h)}_{-s} - (u^{(h)}(t))_\delta \right\|_{L^2}\leq c h \operatorname{Lip}_x(u^{(h)})_\delta\leq h C_\delta \|u^{(h)}(t)\|_{W^{1,2}}.\]
Convergence in the last term is obtained from 
\[\begin{aligned}
\int_0^T&\left\langle (u^{(h)}(t))_\delta ,\frac1h\int_0^h \nabla\xi(t+s)\cdot u^{(h)}(t+s) \,ds\right\rangle\,dt \\
&= \int_0^T\left\langle (u^{(h)}(t))_\delta ,\frac1h\int_0^h (\nabla\xi(t+s) - \nabla \xi(t))\cdot u^{(h)}(t+s) \,ds\right\rangle\,dt\\
&\quad\ +\int_0^T\left\langle (u^{(h)}(t))_\delta ,\nabla \xi(t)\cdot\frac1h\int_0^h u^{(h)}(t+s) \,ds\right\rangle\,dt.
\end{aligned}\] 
In the first term we have
\[\|\nabla \xi(t+s)-\nabla\xi(t)\|_{L^\infty}\leq h\|\partial_t\nabla\xi\|_{L^\infty} \to 0 \text{ with } h\to 0.\]
In the second, we use Theorem \ref{thm:aubin-lions-fluid}, where we take \(A\) to
be an approximation of \(\nabla \xi \chi_{\Omega(t)}\). More precisely,
write
\[\begin{aligned}\int_0^T\left\langle (u^{(h)}(t))_\delta ,\nabla \xi(t)\cdot\frac1h\int_0^h u^{(h)}(t+s) \,ds\right\rangle\,dt =\int_0^T\left\langle (u^{(h)}(t))_\delta ,A_\delta(t)\cdot\frac1h\int_0^h u^{(h)}(t+s) \,ds\right\rangle\,dt \\ + \int_0^T\left\langle (u^{(h)}(t))_\delta ,\nabla (\xi(t)-A_\delta(t))\cdot\frac1h\int_0^h u^{(h)}(t+s) \,ds\right\rangle\,dt \end{aligned}\]
and in the first term we have by Theorem \ref{thm:aubin-lions-fluid} 
\[\int_0^T\left\langle (u(t))_\delta ,A_\delta(t)\cdot\frac1h\int_0^h u^{(h)}(t+s) \,ds\right\rangle\,dt\to \int_0^T\left\langle (u^{(h)}(t))_\delta ,A_\delta(t)\cdot u(t) \,ds\right\rangle\,dt,\]
in the second term we estimate by H\"older's inequality and Sobolev
embedding for \(a<d/(d-2)\) \[\begin{aligned}
&\left|\int_0^T\left\langle (u^{(h)}(t))_\delta ,\nabla (\xi(t)-A_\delta(t))\cdot\frac1h\int_0^h u^{(h)}(t+s) \,ds\right\rangle\,dt \right|\\
 &\leq \int_0^T\left\|\left(u^{(h)}(t)\right)_\delta\right\|_{L^{2 a}(\Omega)} \frac1h\int_0^h\left\|u^{(h)}(t+s)\right\|_{L^a(\Omega)}\left\|A_\delta(t)-\chi_{\Omega^{(h)}}(t+s)\right\|_{L^{2 a^{\prime}(\Omega)}} d s d t \\
 &\leq c\left\|\left(u^{(h)}(t)\right)_\delta\right\|_{L^2\left([0, T] ; W^{1,2}(\Omega)\right)} \sup _{t \in T}\left(\frac1h\int_0^h\left\|u^{(h)}(t+s)\right\|^2 d s\right)^{1 / 2} \\
 &\quad\ \times\left(\frac1h\int_0^h\left\|A_\delta(t)-\chi_{\Omega^{(h)}}(t+s)\right\|_{L^2\left([0, T] ; L^{2 a^{\prime}}(\Omega)\right)}^2 d s\right)^{1 / 2} \\
&\leq  c\left(\frac1h\int_0^h\left\|A_\delta(\cdot)-\chi_{\Omega^{(h)}}(\cdot+s)\right\|_{L^2\left([0, T] ; L^{2 a^{\prime}}(\Omega)\right)}^2 d s\right)^{1 / 2} .
\end{aligned}
\] By the uniform convergence of \(\eta^{(h)} \rightarrow \eta\), we
find that \[
\begin{aligned}
    \lim _{h \rightarrow 0}\left(\frac1h\int_0^h\left\|A_\delta(\cdot)-\chi_{\Omega^{(h)}}(\cdot+s)\right\|_{L^2\left([0, T] ; L^{\left.2 a^{\prime}(\Omega)\right)}\right.}^2\right. & d s)^{1 / 2} 
& =\left\|A_\delta-\chi_{\Omega}\right\|_{L^2\left([0, T] ; L^{2 a^{\prime}}(\Omega)\right)} .
\end{aligned}
\] Finally, choosing
\(A_\delta\in C([0,T];C^{k_0}_0(\Omega_{-\delta}))\) with
\(A_\delta\to \nabla\xi\chi_{\Omega(t)}\) and gives, collecting all
above, \[\begin{aligned}
\int_0^T \left\langle\tfrac{u^{(h)}(t)\circ\Phi^{(h)}(t-h)-u^{(h)}(t-h)}{h},\xi(t)\circ\Phi^{(h)}(t-h) \right\rangle_{\Omega^{(h)}(t-h)} \,dt \\ \to
\int_0^T \left\langle (u(t))_\delta, \partial_t\xi(t) +\nabla \xi(t)\cdot u(t) \right\rangle\,dt \end{aligned}\]
from which we obtain the \(\delta\)-regularized equations.

We can now pass to the limit \(\delta \to 0\) and obtain the limiting equation, as all the terms converge in the same fashion as before.
Finally, using \((\partial_t\eta,v)\) as a test function (we still have
enough regularity for that due to the regularizing terms) we get, as in \eqref{eqn:energy-inequality-taulimit} the energy inequality.
\end{proof}

%\subsection{Estimate of the flow map and a no contact result}

%\noteMalte[inline]{I am not sure if this needs to be its own section}

We will now see that since we keep at this point the regularizing terms
with \(\kappa>0\), that this in fact means that the flow map remains
Lipschitz regular. Indeed, by the estimate of Proposition \ref{prop:estimates-of-the-flow-map} we have that
\[\operatorname{Lip}\Phi(t)\leq \exp(\sqrt T \sqrt{\|\nabla^{k_0} v\|_{L^2((0,T);L^2(\Omega(t);\mathbb R^d))}})\]
and the norm on the right is by \eqref{eqn:energy-inequality-hlimit} finite (although depending on
\(\kappa\)). Then we can argue that
\(\Phi(t)\colon \Omega_0\to\Omega(t)\) is a diffeomorphism. This in
particular means that there is \emph{no change of topology} and
consequently no contact between any solid parts.

\begin{corollary}[No contact with regularization]\label{cor:no-contact-regularization}
The solution obtained in Theorem \ref{thm:solution-hlimit} does not reach a collision.
\end{corollary}

\subsection{\texorpdfstring{Passing to the limit with the regularization \(\kappa \to 0\)}{Passing to the limit with the regularization \textbackslash kappa \textbackslash to 0}}\label{passing-to-the-limit-with-the-regularization-kappa-to-0}

We now reveal the dependence of \(v,\eta\)  in Theorem \ref{thm:solution-hlimit} on \(\kappa\), so we
write \(v^{(\kappa)},\eta^{(\kappa)}\). Recall that so far we have shown
\(v^{(\kappa)},\eta^{(\kappa)}\) to be a solution of the equation with
\(W^{k_0,2}\) (resp. $W^{k_0+2,2}$)-regularizer depending on \(\kappa>0\).

Note that by Corollary \ref{cor:no-contact-regularization} we know that the deformation
\(\eta^{(\kappa)}\) never reaches a collision, as long as \(\kappa> 0\).
However this is no longer guaranteed after \(\kappa\to 0\). Thus below
in Theorem \ref{thm:full-problem}, to get a limiting weak equation, we take the absence of
collisions in the limit \(\eta\) as an assumption (which is true at
least for short times, see \cite[Cor. 2.19]{benesovaVariationalApproachHyperbolic2023}). We aim to incorporate the possibility
of collisions and description of the corresponding Lagrange multiplier
in a future work.

Until now we have had a regularized initial conditions, so we need to
approximate the initial conditions now. That is, for given initial
conditions
\[\eta_0\in\mathcal E,\quad \eta_*\in W^{1,2}(Q;\mathbb R^d),\quad  v_0\in W^{1,2}(\Omega_0;\mathbb R^d)\]we
approximate it by
\[\eta_0^{(\kappa)}\in \mathcal E\cap W^{k_0+2,2}(Q;\mathbb R^d), \quad \eta_*^{(\kappa)} \in W^{k_0+2,2}(Q;\mathbb R^d), \quad v_0^{(\kappa)}\in W^{k_0,2}(\Omega_0;\mathbb R^d)\]
such that we have \[\begin{aligned}
\eta_0^{(\kappa)} \to \eta_0 \quad\text{in }W^{2,q}(Q;\mathbb R^d),\\ \eta_*^{(\kappa)} \to \eta_* \quad\text{in } W^{1,2}(Q;\mathbb R^d),
\\ v_0^{(\kappa)} \to  v_0 \quad \text{in } W^{1,2}(\Omega_0;\mathbb R^d)
\end{aligned}\qquad \text{as }\kappa\to 0\] and below the \(v^{(\kappa)}\), \(\eta^{(\kappa)}\)
solution will be corresponding to these initial conditions.

\begin{theorem}[Full problem]\label{thm:full-problem}
There exists a subsequence $\kappa\to0$ such that the limit $(\eta,v)$ is a weak solution to the full problem as defined in Definition \ref{def:weak-solution-full}, until the time of the first collision.
\end{theorem}

\begin{proof}
\textit{Fluid-only equation}.

The fluid-only equation \eqref{eqn:fluid-only-eqn-hlimit} now reads as
\begin{equation}\label{eqn:fluid-only-eqn-kappa}
\begin{aligned}\int_0^T &-\rho_f\langle v^{(\kappa)},\partial_t \xi\rangle _{\Omega^{(\kappa)}(t)} + \rho_f \langle v^{(\kappa)},v^{(\kappa)}\cdot \nabla\xi \rangle_{\Omega^{(\kappa)}(t)} + \nu \langle \varepsilon v^{(\kappa)}, \varepsilon \xi\rangle_{\Omega^{(\kappa)}(t)} \\
&+ \kappa\langle \nabla ^{k_0}v^{(\kappa)},\nabla^{k_0}\xi\rangle - a \langle \partial_t \eta \circ \eta^{-1} -v,\xi\rangle \,dt \\
&= \int_0^T \rho_f\langle f,\xi\rangle_{\Omega^{(\kappa)}(t)}\,dt+\rho_f \langle v_0^{(\kappa)}, \xi(0)\rangle_{\Omega^{(\kappa)}(0)} \end{aligned}
\end{equation}
for all
\(\xi\in C^\infty([0,T] \times \overline \Omega^{(\kappa)}(t))\),
\(\xi\cdot n^{(\kappa)}=0\) on \(\partial \Omega^{(\kappa)}(t)\),
\(\operatorname{div}\xi=0\) in \(\Omega^{(\kappa)}(t)\), with
\(\xi(T)=0\).

We further have the uniform bounds on \(v^{(\kappa)}\) in
\(L^2((0,T);W^{1,2}(\Omega^{(\kappa)}(t);\mathbb R^d))\) and a
(\(\kappa\)-independent) bound
\[\sqrt\kappa \|\nabla^{k_0}v^{(\kappa)}\|_{L^2((0,T);L^2(\Omega^{(\kappa)}(t);\mathbb R^d))}\leq C.\]
Thus, for a subsequence \(\kappa\to 0\) we have
\[v^{(\kappa)}\overset\eta\rightharpoonup v \quad \text{in }L^2((0,T);W^{1,2}(\Omega(t);\mathbb R^d)).\]
Moreover, we can read the estimate on \(\partial_t v^{(\kappa)}\) in
distributional sense, as from \eqref{eqn:fluid-only-eqn-kappa} we have
\[\left|\int_0^T \langle v,\partial_t \xi\rangle\,dt \right| \leq C\|\xi\|_{L^\infty(C^{k_0})}.\]
Thus from the Aubin-Lions lemma \cite[Corollary 2.9]{breitCompressibleFluidsInteracting2024} we get the strong
convergence
\begin{equation}\label{eqn:vkappa-strongly-l2}
v^{(\kappa)}\overset\eta\to v \quad\text{in }L^2((0,T);L^2(\Omega(t);\mathbb R^d)).
\end{equation}
We desire to show that the limiting equation holds, in particular
\[\begin{aligned}\int_0^T -\rho_f\langle v,\partial_t \xi\rangle _{\Omega(t)} + \rho_f \langle v,v\cdot \nabla\xi \rangle_{\Omega(t)} + \nu \langle \varepsilon v, \varepsilon \xi\rangle_{\Omega(t)} + = \int_0^T \rho_f\langle f,\xi\rangle_{\Omega(t)}\,dt+\rho_f \langle v_0, \xi(0)\rangle_{\Omega(0)} \end{aligned}
\]
for all
\(\xi\in C^\infty([0,T] \times \overline \Omega(t);\mathbb R^d)\),
\(\xi\cdot n=0\) on \(\partial \Omega(t)\), \(\operatorname{div}\xi=0\)
in \(\Omega(t)\), with \(\xi(T)=0\).

For this, let us now fix a test function for the limit, that is
\(\xi\in C^\infty([0,T] \times \overline \Omega(t));\mathbb R^d\),
\(\xi\cdot n=0\) on \(\partial \Omega(t)\), \(\operatorname{div}\xi=0\)
in \(\Omega(t)\), with \(\xi(T)=0\).

For this \(\xi\) find \(\xi^{(\kappa)}_\delta\) as defined in Proposition \ref{prop:approximation-fluid-only} (i), this \(\xi_\delta^{(\kappa)}\) is now a valid test function
for \eqref{eqn:fluid-only-eqn-kappa}. We thus have
\[\begin{aligned}\int_0^T -\rho_f\langle v^{(\kappa)},\partial_t \xi_\delta^{(\kappa)}\rangle _{\Omega^{(\kappa)}(t)} + \rho_f \langle v^{(\kappa)},v^{(\kappa)}\cdot \nabla\xi_\delta^{(\kappa)} \rangle_{\Omega^{(\kappa)}(t)} + \nu \langle \varepsilon v^{(\kappa)}, \varepsilon \xi_\delta^{(\kappa)}\rangle_{\Omega^{(\kappa)}(t)} \\+ \kappa\langle \nabla ^{k_0}v^{(\kappa)},\nabla^{k_0}\xi_\delta^{(\kappa)}\rangle - a \langle \partial_t \eta \circ \eta^{-1} - v ,\xi \rangle_{\partial \Omega} \,dt = \int_0^T \rho_f\langle f,\xi_\delta^{(\kappa)}\rangle_{\Omega^{(\kappa)}(t)}\,dt+\rho_f \langle v_0^{(\kappa)}, \xi_\delta^{(\kappa)}(0)\rangle_{\Omega^{(\kappa)}(0)} \end{aligned}\]
Now we use the weak convergence of \(v^{(\kappa)}\rightharpoonup v\) strong convergence
\(\xi_\delta^{(\kappa)} \overset\eta\to\xi_\delta\) and the estimate (which follows from \eqref{eqn:fluid-only-eqn-kappa})
\[\begin{aligned}
\int_0^T\kappa|\langle\nabla^{k_0}v^{(\kappa)},\nabla^{k_0}\xi_\delta^{(\kappa)}\rangle_{\Omega^{(\kappa)}(t)}|\,dt \leq \kappa \|\nabla^{k_0}v^{(\kappa)}\|_{L^2((0,T);L^2(\Omega^{(\kappa)}(t);\mathbb R^d)}\|\nabla^{k_0}\xi_\delta^{(\kappa)}\|_{L^2((0,T);L^2(\Omega^{(\kappa)}(t);\mathbb R^d)} \\ \leq \sqrt\kappa C C_\delta\to 0
\end{aligned}\]
with \(\kappa\to 0\) and \(\delta>0\) fixed. So that after passing to
\(\kappa \to 0\) we have for all \(\delta>0\) (passing to the
limit in
\(\int_0^T \langle v , v\cdot \nabla\xi_\delta \rangle \,dt\) is
due to the strong convergence \eqref{eqn:vkappa-strongly-l2})

\[\begin{aligned}\int_0^T -\rho_f\langle v,\partial_t \xi_\delta\rangle _{\Omega(t)} + \rho_f \langle v,v\cdot \nabla\xi_\delta \rangle_{\Omega(t)} + \nu \langle \varepsilon v, \varepsilon \xi_\delta\rangle_{\Omega(t)} \,dt=\int_0^T \rho_f\langle f,\xi_\delta\rangle_{\Omega(t)}\,dt+\rho_f \langle v_0, \xi_\delta(0)\rangle_{\Omega(0)}. \end{aligned}\]
Since then we can, as before, pass to \(\delta \to 0\) and see that
we have the desired limiting equation
\[\begin{aligned}\int_0^T -\rho_f\langle v,\partial_t \xi\rangle _{\Omega(t)} + \rho_f \langle v,v\cdot \nabla\xi \rangle_{\Omega(t)} + \nu \langle \varepsilon v, \varepsilon \xi\rangle_{\Omega(t)} \,dt = \int_0^T \rho_f\langle f,\xi\rangle_{\Omega(t)}\,dt+\rho_f \langle v_0, \xi(0)\rangle_{\Omega(0)} \end{aligned}.\]
By a density argument we can see that this continues to hold for
\(\xi \in W^{1,2}((0,T); L^2(\Omega(t);\mathbb R^d))\cap L^2((0,T);W^{1,2}_n(\Omega(t);\mathbb R^d))\).

\noindent \textit{Coupled equation.}

The coupled equation \eqref{eqn:coupled-eqn-hlimit} is now \[\begin{aligned} 
\int_0^T &DE(\eta^{(\kappa)})\langle\phi^{(\kappa)}\rangle+D_2 R\left(\eta^{(\kappa)}, \partial_t\eta^{(\kappa)} \right) \langle \phi^{(\kappa)} \rangle  + 2\kappa^{a_0}\langle \nabla^{k_0+2}\eta^{(\kappa)}, \nabla^{k_0+2}\phi^{(\kappa)} \rangle 
\\&+ 2\kappa\left\langle \nabla^{k_0+2}\partial_t\eta^{(\kappa)} , \nabla^{k_0+2}\phi^{(\kappa)} \right\rangle 
-\rho_s\left\langle \partial_{t}\eta^{(\kappa)}, \partial_t \phi^{(\kappa)}\right\rangle-\rho_s\left\langle f\circ\eta^{(\kappa)},\phi^{(\kappa)} \right\rangle
\\
&+\nu \langle\varepsilon v^{(\kappa)}, \varepsilon \xi\rangle_{\Omega^{(\kappa)}(t)} + \kappa\langle\nabla^{k_0} v^{(\kappa)}, \nabla^{k_0}\xi\rangle_{\Omega^{(\kappa)}(t)}-\rho_f \langle  \partial_t v^{(\kappa)},\xi\rangle_{\Omega(t)} + \rho_f\langle v^{(\kappa)} , v^{(\kappa)}\cdot \nabla \xi \rangle_{\Omega(t)} 
\\ &- \rho_f\langle f,\xi\rangle_{\Omega^{(\kappa)}(t)}\,dt=0
\end{aligned}\] for all
\(\xi\in C^\infty([0,T]\times \overline \Omega^{(\kappa)}(t);\mathbb R^d)\),
\(\operatorname {div} \xi =0\) in \(\Omega^{(\kappa)}(t)\) and
\(\phi^{(\kappa)}\in W^{k_0+2,2}(Q;\mathbb R^d)\) with
\(\phi^{(\kappa)} = \xi\circ\eta^{(\kappa)}\) on \(Q\).

We shall now pass to the limit \(\kappa\to 0\). We have estimates by the energy inequality \eqref{eqn:energy-inequality-hlimit} which now read as
 \[\begin{aligned} 
E(\eta^{(\kappa)}(t))&+2\kappa^{a_0}\|\nabla^{k_0+2}\eta^{(\kappa)}(t)\|^2+ \frac{\rho_f}{2} \|v^{(\kappa)}(t)\|^2_{\Omega(t)}+ \frac{\rho_s}{2}\| \partial_t\eta^{(\kappa)}(t) \|^2  
\\ &+ \int_0^t 2R(\eta^{(\kappa)},\partial_t\eta^{(\kappa)})  + 2\kappa\|\nabla^{k_0+2}\partial_t\eta^{(\kappa)}\|^2 + \|\nabla^{k_0}v^{(\kappa)}\|^2 \,dt
\\
&\leq E(\eta^{(\kappa)}(0))+2\kappa^{a_0}\|\nabla^{k_0+2}\eta^{(\kappa)}(0)\|^2+ \frac{\rho_f}{2}\|v^{(\kappa)}(0) \|^2_{\Omega_0}+\frac{\rho_s}{2}\| \partial_t\eta^{(\kappa)}(0) \|^2 \,dt 
\\ 
&\quad\ +\int_0^t\rho_s\left\langle f\circ\eta^{(\kappa)},\partial_t\eta^{(\kappa)} \right\rangle
+ \rho_f\langle f,v^{(\kappa)}\rangle_{\Omega(t)}\,dt
\end{aligned}\] to obtain weak convergences

\begin{align*}
\eta^{(\kappa)}&\overset\ast\rightharpoonup \eta &\quad\text{in }&L^\infty((0,T);W^{2,q}(Q;\mathbb R^d)), \\
\partial_t\eta^{(\kappa)}&\rightharpoonup\partial_t\eta &\quad\text{in }&L^2((0,T)W^{1,2}(Q;\mathbb R^d)),\\
v^{(\kappa)} &\overset\eta\rightharpoonup v &\quad \text{in }&L^2((0,T);W^{1,2},(\Omega(t);\mathbb R^d))
\intertext{ as well as the strong convergence}
\eta^{(\kappa)}&\to \eta &\quad\text{in }&C^0((0,T);C^{1,\alpha}(Q;\mathbb R^d))
\end{align*}
by the Aubin-Lions theorem.

Next we want to employ the Minty property \eqref{as:E-minty-property}, in order to improve the convergence of $\eta^{(\kappa)}$ to strong convergence.

As in Lemma \ref{lem:ext-normal-tangent}, we can pick a set $Q_\delta$ on which $\eta^{(\kappa)}$ can be extended injectively independently of $\kappa$. Now by the uniform convergence $\eta^{(\kappa)} \to \eta$, for $\kappa$ large enough, we can pick a cutoff function $\psi \in C^\infty((0,T) \times \Omega)$ such that $\operatorname {supp} \psi \subset \eta^{(\kappa)}(t,Q_\delta)$ for all $t$ and $\kappa$, while at the same time $\psi = 1$ on $\eta^{(\kappa)}(t,Q)$. Define $\eta_{\varepsilon_\kappa}$ as a mollification of $\eta$ with a smooth kernel of size $\delta_\kappa \to 0$ for $\kappa \to 0$ and
\begin{align*}
 \tilde{\xi}^{(\kappa)} := \psi (x - \eta_{\delta_\kappa} \circ (\eta^{(\kappa)})^{-1}) = \psi (\eta^{(\kappa)} - \eta_{\delta_\kappa}) \circ (\eta^{(\kappa)})^{-1}
\end{align*}
where we continue the function by zero where it is not defined. Then since both $\eta^{(\kappa)} \to \eta$ and $\eta_{\delta_\kappa} \to \eta$ in $C^{1,\alpha}(Q)$ we know that $\tilde{\xi}^{(\kappa)} \to 0$ in $C^{1,\alpha}(\Omega)$. By choosing $\delta_\kappa$ we can also guarantee that $\|\tilde{\xi}^{(\kappa)}\|_{L^2W^{k_0,2}} =o(\kappa^{-1})$.

Now take $b \in C_c^\infty(\Omega)$ to be a bump function with $\int_\Omega b dx = 1$ and support a sufficient distance away from $\eta^{(\kappa)}(t,Q)$ for any $\kappa$ large enough, as in the proof of the universal Bogovskii operator Theorem~\ref{thm:bogovskii-close-deformations}.\footnote{Such a $b$ always exists, since we assume the set $\eta^{(\kappa)}(t,Q)$ always has uniformly bounded distance to $\partial \Omega$.} Then applying this operator we obtain.
\begin{align*}
 \xi^{(\kappa)} := \tilde{\xi}^{(\kappa)} - \mathcal{B}\left(\operatorname{div} \tilde{\xi}^{(\kappa)} - \lambda(t) b \right)
\end{align*}
where $\lambda(t) := \int_{\Omega^{(\kappa)}(t)} \operatorname{div}\tilde{\xi}^{(\kappa)} dx$. Per construction now $\xi^{(\kappa)} = \tilde{\xi}^{(\kappa)}$ on $\eta^{(\kappa)}(Q)$ and 

$\operatorname{div} \xi^{(\kappa)} =0 $ on $\Omega^{(\kappa)}(t)$. In other words, the pair $(\phi^{(\kappa)},\xi^{(\kappa)}) := (\eta^{(\kappa)}-\eta,\xi^{(\kappa)})$ is an admissible pair of test functions for the coupled equation. We also note that by the properties of the Bogovskii operator $\xi^{(\kappa)}$ converges to $0$ in the same spaces as $\tilde{\xi}^{(\kappa)}$.

With this we now finally consider \eqref{as:E-minty-property} and estimate
\begin{align*}
 0 \leq \int_0^T (DE(\eta^{(\kappa)})-DE(\eta))\langle\eta^{(\kappa)} -\eta\rangle dt = \int_0^T DE(\eta^{(\kappa)})\langle\eta^{(\kappa)} -\eta\rangle - DE(\eta)\langle\eta^{(\kappa)} -\eta\rangle dt.
\end{align*}
The second term converges to zero by the weak convergence of $\eta^{(\kappa)}$. For the first term we now use the equation and note that by the coupled equation
\begin{align*}
 - \int_0^T DE(\eta^{(\kappa)})\langle\eta^{(\kappa)} -\eta\rangle = \int_0^T D_2 R\left(\eta^{(\kappa)}, \partial_t\eta^{(\kappa)} \right) \langle \phi^{(\kappa)} \rangle  + 2\kappa^{a_0}\langle \nabla^{k_0+2}\eta^{(\kappa)}, \nabla^{k_0+2}\phi^{(\kappa)} \rangle +
\\ 2\kappa\left\langle \nabla^{k_0+2}\partial_t\eta^{(\kappa)} , \nabla^{k_0+2}\phi^{(\kappa)} \right\rangle
-\rho_s\left\langle \partial_{t}\eta^{(\kappa)}, \partial_t \phi^{(\kappa)}\right\rangle-\rho_s\left\langle f\circ\eta^{(\kappa)},\phi^{(\kappa)} \right\rangle
\\
+\nu \langle\varepsilon v^{(\kappa)}, \varepsilon \xi^{(\kappa)}\rangle_{\Omega^{(\kappa)}(t)} + \kappa\langle\nabla^{k_0} v^{(\kappa)}, \nabla^{k_0}\xi^{(\kappa)}\rangle_{\Omega^{(\kappa)}(t)}-\rho_f \langle  \partial_t v^{(\kappa)},\xi^{(\kappa)}\rangle_{\Omega(t)} + \rho_f\langle v^{(\kappa)} , v^{(\kappa)}\cdot \nabla \xi^{(\kappa)} \rangle_{\Omega(t)} \\ - \rho_f\langle f,\xi^{(\kappa)}\rangle_{\Omega^{(\kappa)}(t)}\,dt
\end{align*}
Now all terms on the right hand side are a product of a bounded quantity and a test function converging to zero.
%\noteMalte[inline]{I am torn here, if I should list all the convergences. Most of them are obvious and this would probably add at least another page to the proof.}
So
by the Minty property \eqref{as:E-minty-property} that
\begin{align*}
\eta^{(\kappa)}(t)&\to \eta(t)\quad &\text{in } &W^{2,q}(Q;\mathbb R^d) \text{ for a.a. }t\in (0,T)
\intertext{and, as above, we have}
v^{(\kappa)}&\overset\eta\to v \quad&\text{in }&L^2((0,T);L^2(\Omega(t);\mathbb R^d)).
\end{align*}
Passing now to the limit in the coupled equation, we obtain that
\[\begin{aligned}
\int_0^T\kappa^{a_0}|\langle\nabla^{k_0+2}\eta^{(\kappa)}, \nabla^{k_0+2}\phi^{(\kappa)} \rangle| \,dt &\leq \kappa^{a_0}\|\nabla^{k_0+2}\eta^{(\kappa)}\|_{L^\infty L^2} \|\nabla^{k_0+2}\phi^{(\kappa)}\|_{L^1L^2} \leq \kappa^{a_0}C, 
\\
\int_0^T \kappa\left\langle \nabla^{k_0+2}\partial_t\eta^{(\kappa)} , \nabla^{k_0+2}\phi^{(\kappa)} \right\rangle \,dt &\leq \kappa\|\nabla^{k_0+2}\partial_t\eta^{(\kappa)}\|_{L^2 L^2} \|\nabla^{k_0+2}\phi^{(\kappa)}\|_{L^2L^2} \leq \kappa C,
\end{aligned}\] 
so that these regularizing terms vanish as \(\kappa\to 0\).

We have thus shown enough to pass to the limit \(\kappa \to 0\) and
solve the limit problem

\[\begin{aligned} 
\int_0^T DE(\eta)\langle\phi\rangle+D_2 R\left(\eta, \partial_t\eta \right) \langle \phi \rangle
-\rho_s\left\langle \partial_t\eta, \partial_t \phi\right\rangle-\rho_s\left\langle f\circ\eta,\phi \right\rangle
\\
+\nu \langle\varepsilon v, \varepsilon \xi\rangle_{\Omega(t)} -\rho_f \left\langle  \partial_t v,\xi\rangle_{\Omega(t)} + \rho_f\langle v , v\cdot \nabla \xi \right\rangle_{\Omega(t)} - \rho_f\langle f,\xi\rangle_{\Omega(t)}\,dt=0. 
\end{aligned}\]
\end{proof}
 
\subsection*{Acknowledgements}

\noindent
The authors acknowledge the support of the ERC-CZ grant LL2105. S. S. further acknowledges the support of the VR-Grant 2022-03862 by the Swedish Research Council. A.\v C. and M.K. further acknowledge the support of the Czech Science Foundation (GAČR) grant No. 23-04766S, OP JAK grant FerrMion and Charles University grant PRIMUS/24/SCI/020 (M.K. only) and Research Centre program No. UNCE/24/SCI/005. A. \v C. further acknowledges the support of UKRI Frontier Research Guarantee (ERC guarantee) grant EP/Z000297/1 (ERC CONCENTRATE).

%\noteMalte[inline]{Please check your acknowledgements and affiliation. (e.g. Warwick)}
 
\bibliographystyle{alpha}
\bibliography{biblio}
\end{document}